\documentclass{article}

\usepackage[utf8]{inputenc}
\usepackage[english]{babel}

\usepackage[left=2cm,right=2cm,top=2cm,bottom=2cm]{geometry}

\usepackage{cancel}
\usepackage{amstext}
\usepackage{amssymb}
\usepackage{amsmath}
\usepackage{amsfonts}
\usepackage{mathrsfs}
\usepackage{bm}
\usepackage{bbm}
\usepackage{nicefrac, xfrac}
\usepackage{upgreek}
\usepackage{booktabs}
\usepackage{xfrac}

\usepackage{xcolor}
\usepackage{color}

\newif\ifproposals \proposalstrue
\DeclareRobustCommand{\prop}[1]{\ifproposals\textcolor{blue}{#1}\else#1\fi}
\newenvironment{propblock}{\ifproposals\color{blue}\fi}{}

\usepackage{graphicx}
\usepackage{pgf,tikz,pgfplots}
\pgfplotsset{compat=1.15}
\usetikzlibrary{arrows, arrows.meta, calc, math, shapes, positioning}

\usepackage{wrapfig}

\usepackage[colorlinks=true]{hyperref}

\usepackage{ifpdf}

\usepackage[babel=true]{csquotes}
\MakeOuterQuote{"}
\usepackage{textcomp}

\usepackage{graphicx}
\usepackage{longtable}
\usepackage{caption}
\usepackage{subcaption}

\usepackage{svg}

\usepackage{url}
\usepackage{float}
\usepackage{multicol}

\usepackage[framed, amsthm, amsmath, thmmarks]{ntheorem}

\definecolor{mygray}{gray}{0.95}

\newcommand{\theoremcolorbox}[2]{\def\theoremframecommand{\fcolorbox{#1}{#2}}}
\theoremcolorbox{black}{mygray}
\newtheorem{theorem}{Theorem}[section]
\newtheorem{definition}[theorem]{Definition}
\newtheorem{corollary}[theorem]{Corollary}
\newtheorem{lemma}[theorem]{Lemma}
\newtheorem{proposition}[theorem]{Prop.}

\newcounter{customcondition}

\theoremstyle{remark}
\newtheorem{remark}{Remark}

\newtheorem{example}{Example}
\newframedtheorem{exampleFramed}{Example}

\newframedtheorem{updateFramed}{Update Rule}

\makeatletter
\NewDocumentEnvironment{assum}{m O{}}{%
  \par\addvspace{\topsep}%
  \noindent
  \def\@currentlabel{#1}%
  {\normalfont\bfseries Assumption #1%
    \if\relax\detokenize{#2}\relax\else
      \ \normalfont(#2)%
    \fi
    \normalfont\bfseries.}%
  \nobreak\hspace{.5em}\ignorespaces
}{%
  \par\addvspace{\topsep}%
}
\NewDocumentEnvironment{cond}{m O{}}{%
  \par\addvspace{\topsep}%
  \noindent
  \def\@currentlabel{#1}%
  {\normalfont\bfseries Condition #1%
    \if\relax\detokenize{#2}\relax\else
      \ \normalfont(#2)%
    \fi
    \normalfont\bfseries.}%
  \nobreak\hspace{.5em}\ignorespaces
}{%
  \par\addvspace{\topsep}%
}
\makeatother

\makeatletter
\NewDocumentEnvironment{update}{m O{}}{%
  \par\addvspace{\topsep}%
  \noindent
  \def\@currentlabel{#1}%
  {\normalfont\bfseries Update #1%
    \if\relax\detokenize{#2}\relax\else
      \ \normalfont(#2)%
    \fi
    \normalfont\bfseries.}%
  \nobreak\hspace{.5em}\ignorespaces
}{%
  \par\addvspace{\topsep}%
}
\makeatother

\usepackage{enumerate}
\usepackage[shortlabels]{enumitem}

\numberwithin{equation}{section}

\usepackage[ruled,vlined]{algorithm2e}
\allowdisplaybreaks
\usepackage{algorithmicx, algpseudocode}
\usepackage[capitalize]{cleveref}

\newcommand{\rme}{{\mathrm{e}}}

\newcommand{\R}{{\mathbb{R}}} % reals
\newcommand{\NN}{{\mathbb{N}}} % natural numbers {1, 2, ...}
\newcommand{\RR}{{\mathbb{R}}} % reals
\newcommand{\reals}{\mathbb{R}}

\DeclareSymbolFont{bbold}{U}{bbold}{m}{n}
\DeclareSymbolFontAlphabet{\mathbbold}{bbold}

\newcommand{\calL}{{\mathcal{L}}}

\newcommand{\calO}{{\mathcal{O}}}

\newcommand{\abs}[1]{\left \vert #1 \right \vert}

\newcommand{\norm}[1]{\left \lVert #1 \right \rVert}

\newcommand*\Rp{\reals^p}
\newcommand*\expo[1]{\rme^{#1} }

\newcommand*\mykappa[1]{ \kappa_{\scriptscriptstyle \text{#1}} }
\newcommand*\myeta[1]{ \eta_{\scriptscriptstyle \text{#1}} }

\newcommand{\Rdelta}{\textsf{R-delta}}
\newcommand{\Rstep}{\textsf{R-step}}
\newcommand{\RAdaptstep}{\textsf{R-adaptive-step}}
\newcommand{\Rdeltastep}{\textsf{R-delta-step}}
\newcommand{\Rtruncgrad}{\textsf{R-trunc-grad}}
\newcommand{\Rdfo}{\textsf{R-DFO}}
\newcommand{\Rgrad}{\textsf{R-grad}}

\newcommand{\PartII}{Part~II}

\newcommand{\PartIIref}{Part~II~\citep{Jeremy2026-II}}

\newcommand{\gone}{\gamma_1}
\newcommand{\gtwo}{\gamma_2}

\usepackage[square,numbers]{natbib}
\setcitestyle{semicolon, aysep={,}}

\title{{\bf \prop{Structural conditions on trust-region radius updates.  \\
            Part I: first-order convergence and complexity}} }

\author{{\bf J\'{e}r\'{e}my Rieussec} \quad {\bf Fabian Bastin}\\
	University of Montreal, Dept. Computing Science and Operational Research\\
	Email: jeremy.rieussec@umontreal.ca, bastin@iro.umontreal.ca \\
}

\date{\today}

\proposalsfalse  % review markup off for the public build
\begin{document}

\maketitle

\begin{abstract}
\begin{propblock}
Trust-region methods differ chiefly in how they update the radius, and each rule has traditionally been analysed on its own terms. We show that three structural conditions on the radius sequence suffice: a lower bound relative to the gradient norm, a contraction on unsuccessful iterations, and a controlled expansion on successful ones. A mechanism is \emph{weakly admissible} if it satisfies the first two, and \emph{strongly admissible} if it satisfies the first and third.

Strong admissibility yields the optimal worst-case complexity $\calO(\varepsilon^{-2})$ for first-order stationarity, and extends global convergence to model Hessians growing linearly with the iteration; weak admissibility gives $\lim_{k\to\infty}\|\nabla f(x_k)\| = 0$ under uniformly bounded model Hessians. Both hold for every rule in the class, so a new mechanism is covered by verifying two inequalities rather than by a fresh convergence proof.

We verify admissibility for five mechanism classes---fixed-factor, step-driven, retrospective, criticality-anchored and gradient-scaled---prove convergence of the retrospective update under linearly growing model Hessians, and extend the Lyapunov framework of~\citet{CurtSche20, WangYuan2022} to three distinct scaling factors with decoupled step acceptance, covering $\eta = 0$; its stochastic version specialises to the deterministic gradient-scaled rule.

\PartIIref{} treats the asymptotic behaviour of the radius and convergence to second-order critical points.
\end{propblock}
\end{abstract}

\section{Introduction} \label{sec:introduction}

    In this work, we study unconstrained, possibly nonconvex, optimisation problems of the form
    \begin{equation}
        \min_{x \in \reals^{p}} \; f(x),
        \label{eqn: deterministic obj function}
    \end{equation}
    where $f: \reals^p \to \R$ is at least continuously differentiable. 
    Since global minima are in general hard to locate, we seek \emph{first-order critical points}, i.e., points where $\nabla f(x^*) = 0$. \\

    \noindent\textbf{Trust-region methods.}
    Originating in the work of Levenberg and Marquardt~\citep{Levenberg_1944, Marquardt_1963}, trust-region methods compute steps by approximately minimising a local quadratic model of $f$ within a ball of radius $\Delta_k$. The radius adapts to model quality: it expands when the model is accurate and contracts when it is not.
    For comprehensive overviews, see~\citet{ConnGoulToin00, NoceWrig06, More83, Yuan00, Yuan15}. The framework has been extended in many directions: discretised problems arising in PDE-constrained optimisation~\citep{GrattonSartenaerToint08},  constrained optimisation in Hilbert spaces~\citep{Toint88}, composite optimisation~\citep{CartisGouldToint11-composite}, and nonsmooth optimisation~\citep{Yuan85, GarmanjaniJudiceVicente16}.
    Recent work has focused on worst-case complexity bounds~\citep{CartGoulToin10, CartGoulToin12, CartGoulToin22, CurtisRobinsonSamadi2017, CurtisLubbertsRobinson2018, CurtisRobinsonRoyerWright2021, GrapigliaYuanYuan15, GrapigliaYuanYuan16}, extensions to derivative-free and stochastic settings~\citep{Powell2002, ConnScheVice09, BandScheVice14, BlanCartMeniSche19, ChenMeniSche18, CurtSche20, WangYuan2022}, and unified frameworks across trust-region and line-search methods~\citep{CurtSche20}.

    \noindent\textbf{The radius update problem.}
    The radius update rule drives this adaptation of $\Delta_k$.
    Many rules have been proposed: ratio-based, step-tracking, retrospective, criticality-based, gradient-proportional.
    Each has been analysed in isolation, under its own assumptions and with its own convergence proof.
    No unified perspective exists on which structural properties make a rule work.
    This makes it hard to compare mechanisms, to identify what is intrinsic to the trust-region framework, and to transfer guarantees across settings.

    \noindent\textbf{Beyond the deterministic setting.}
    Modern applications, such as machine learning, often lack exact function and derivative evaluations. Several methods extend classical algorithms to noisy settings, through stochastic gradient~\citep{RobbMonro51,BottCurtNoce18}, variable sampling~\citep{FrieSchm12, ByrdChinNoceWu12, BollByrdNoce18SIAM, BollByrdNoce18IMA}.  For example, \citet{Cart91, XuRoosMaho20MP, XuRoosMaho20SIAM, YaoXuRoostaMahoney2021} accommodate inexact gradient and Hessian information in a trust-region framework. 
    Recent advances extend convergence frameworks~\citet{CurtSche20} to stochastic settings  through martingale analysis and renewal processes~\citep{WangYuan2022, BandScheVice14, BlanCartMeniSche19, ChenMeniSche18}.
    A central concept in these analyses is model accuracy, formalised through the notions of $\kappa$-fully linear and $\kappa$-fully quadratic models~\citep{ConnScheVice09, ConnScheVice09_introduction}.
    For $\kappa = (\mykappa{ef}, \mykappa{eg}) > 0$, a model $m_k$ is $\kappa$-fully linear on the ball of radius $\Delta_k$ if, for all $s$ with $\|s\| \leq \Delta_k$,
    \begin{equation}
        \| \nabla f(x_k) - \nabla m_k(0) \| \leq \mykappa{eg} \Delta_k, \quad \text{and} \quad
        | f(x_k + s) - m_k(s) | \leq \mykappa{ef} \Delta_k^2,
        \label{eq:fully_linear}
    \end{equation}
    with an analogous $\kappa$-fully quadratic notion for second-order models, which requires $\kappa = (\mykappa{ef}, \mykappa{eg}, \mykappa{eh}) > 0$ such that, for all $s$ with $\|s\| \leq \Delta_k$,
    \begin{equation}
        \| \nabla^2 f(x_k) - \nabla^2 m_k(0) \| \leq \mykappa{eh} \Delta_k, \quad
        \| \nabla f(x_k) - \nabla m_k(0) \| \leq \mykappa{eg} \Delta_k^2, \quad
        | f(x_k + s) - m_k(s) | \leq \mykappa{ef} \Delta_k^3.
        \label{eq:fully_quadratic}
    \end{equation}
    These model conditions~\eqref{eq:fully_linear} and~\eqref{eq:fully_quadratic} tie the trust-region radius to model accuracy. The radius must converge to zero to guarantee accurate models near optimality.
    By contrast, classical trust-region methods~\citep{ConnGoulToin00, NoceWrig06} maintain $\Delta_k$ bounded below near a minimiser, allowing eventual Newton or quasi-Newton steps and faster local convergence under standard assumptions.
    \citet{Fan06, FanLu15, FanPan11} bridge these two regimes. $\Delta_k$ may converge to zero at the rate  $\calO(\|x_k - x^*\|)$ while preserving global and superlinear convergence. \citet{YuanFan01, FanPan16} propose update rules that achieve this regime.

    \noindent\textbf{Contributions and outline.}
    Mechanisms have been proposed across contexts, with distinct convergence analyses and asymptotic behaviours. We address the following questions:
    \begin{enumerate}[(i)]
        \item Which trust-region radius update rules have been proposed in the literature?
        \item Which properties are intrinsic to the trust-region framework, i.e.\ independent of the specific radius update rule?
        \item Do all mechanisms ensure global convergence? Do they compare in their guarantees?
        \item Does the update rule affect the worst-case complexity of the algorithm?
        \item Which updates can offer superlinear convergence under suitable conditions?
    \end{enumerate}
    We isolate three structural conditions on the radius update rule and study the class of mechanisms they define.
    Let $M_k$ denote a running bound on the model Hessian norm, and let $g_k$ denote the model gradient at iteration $k$. 
    The three conditions are, with $\underline{\gamma_2} \in (0,1)$, $\overline{\gamma_3} > 1$, and some constant $\tilde{\eta} \in (0,1)$:
    \begin{align}
        &\textbf{Lower bound:} \quad
        \Delta_k \;\geq\; \frac{\mykappa{lbd}}{M_k}
        \min_{0 \leq j \leq k} \|g_j\|
        \quad \text{for some } \mykappa{lbd} > 0,
        \label{cond: delta_k lower bound intro} \\[4pt]
        &\textbf{Contraction:} \quad
        \Delta_{k+1} \;\leq\; \underline{\gamma_2} \,\Delta_k
        \quad \text{when } \rho_k < \tilde{\eta},
        \label{cond: decrease intro} \\[4pt]
        &\textbf{Expansion:} \quad
        \Delta_{k+1} \;\leq\; \overline{\gamma_3} \,\Delta_k
        \quad \text{when } \rho_k \geq \tilde{\eta},
        \label{cond: increase intro}
    \end{align}
    A mechanism is \emph{weakly admissible} if it satisfies the first two conditions, and \emph{strongly admissible} if it satisfies all three.

    Under uniformly bounded model Hessians,  $\liminf_{k\to\infty}\norm{g_k} = 0$ implies $\lim_{k\to\infty}\norm{g_k} = 0$ independently of the update rule (Theorem~\ref{th: firt-order TR CV}).  
    Adding weak admissibility yields $\lim_{k\to\infty}\norm{g_k} = 0$ (Theorem~\ref{thm: Liminf of the gradient norm under strong assumption}), and strong admissibility  yields the optimal complexity $\calO(\varepsilon^{-2})$ (Theorem~\ref{thm: Upper bound successful iter BTR}).
    Strong admissibility extends the liminf guarantee to linearly growing model Hessians (Theorem~\ref{thm: Liminf of the gradient norm under weaker assumption}).  
    We verify admissibility for five mechanism classes: fixed-factor~\citep{ConnGoulToin00}, step-driven~\citep{Powe75, Hei03}, retrospective~\citep{BastMalmMoufToinToma10}, criticality-anchored~\citep{ConnScheVice09_introduction, BandScheVice14}, and gradient-scaled~\citep{FanPan16, YuanFan01}.
    We extend the retrospective analysis of~\citet{BastMalmMoufToinToma10} to linearly growing model Hessians, complementing~\citet{Yuan15}.
    We also adapt the Lyapunov framework surveyed by~\citet{CurtSche20} and developed
    in~\citet{CartSche18, BlanCartMeniSche19} with three multiplicative factors and decoupled step acceptance, this allows $\eta = 0$ which is not treated there, and specialise the stochastic analysis of~\citet{WangYuan2022} to the deterministic case.

    \noindent\textbf{Relation to existing work.}
    The closest predecessor of the present work is the survey of~\citet{Yuan15}, which reviews trust-region algorithms for unconstrained optimisation, nonsmooth problems, equality- and inequality-constrained problems, and conic and subspace variants. Yuan's survey is broad in scope and catalogues mechanisms, applications, and open questions across the trust-region landscape. His Section~2 covers the present subject directly, giving the step-driven rule
    of~\citet{Hei03}, the gradient-scaled rule of~\citet{YuanFan01} with the variant
    $\Delta_k = \mu_k\norm{g_k}^{\alpha}$, and the retrospective rule
    of~\citet{BastMalmMoufToinToma10}, and observing that a radius bounded away from zero is
    sufficient but not necessary for a fast local rate, so that rules driving $\Delta_k$ to zero
    deserve study. Where we differ is not in scope but in the organising principle: he arranges the
    mechanisms one at a time, we arrange them by three structural conditions on the radius update,
    and this \emph{condition-based} viewpoint unifies the convergence analyses and makes the
    guarantees directly comparable. \PartII{} carries the same viewpoint into the asymptotic and
    local regimes, where it separates the mechanisms along an axis that the vanishing of
    $\Delta_k$ does not.
    \citet{CartGoulToin22} introduce the lower-bound property on the radius (our C1) for the classical fixed-factor update. We elevate it to a named condition, verify it holds for all mechanisms studied, and show it is key for global convergence.  
    \citet{GrapigliaYuanYuan15,GrapigliaYuanYuan16, Toint13} propose a nonlinear stepsize control framework that unifies trust-region and regularisation schemes. The radius follows a specific function  \(\Delta(\delta, \chi) = \delta^{\alpha} \chi^{\beta}\) with \( \alpha \in (0, 1] \) and \( \beta \in [0, 1] \). They establish optimal complexity bounds for these update rules. Our complexity argument follows their pattern and arguments, but applies to any strongly admissible mechanism.
    The overlap is wider than complexity alone, and we state it plainly. Theorem~1
    of~\citet{GrapigliaYuanYuan15} already proves convergence for that framework under
    $\sum_k 1/M_k = +\infty$, which is the growth condition used here, so the conclusion of
    Theorem~\ref{thm: Liminf of the gradient norm under weaker assumption} is theirs for the rules
    their framework covers. What differs is the class it is proved for. Their hypothesis prescribes
    the algebraic form of the rule, $\delta$ being rescaled by reading $\rho_k$ alone and the radius
    being $\delta^{\alpha}\chi^{\beta}$; ours prescribes a behaviour, and any mechanism exhibiting
    it is covered whatever its algebra. Taking $(\alpha,\beta) = (1,0)$ recovers the fixed-factor
    rule and $(1,1)$ the gradient-scaled rule, so those two are common ground; the step-driven and
    retrospective mechanisms verified in Section~\ref{sec:radius update mechanisms} are not of that
    form, the first reading its radius off the realised step and the second selecting its branch by
    a ratio computed from the next model.
    Finally, unlike ARC-type frameworks~\citep{CartGoulToin11, BirgGardMartSantToin17,NestPoly06}, which tie convergence to regularisation, our conditions act directly on the radius. 

    \noindent\textbf{What is new here.}
    \prop{A framework earns its place by what it subsumes, and it is worth saying which statements
    are ours.} The three structural conditions, the notions of weak and strong admissibility built on
    them, and the verification of the five mechanism classes against them are new, and so is the
    extension of the retrospective analysis to linearly growing model Hessians. The convergence and
    complexity theorems themselves are not: their statements are those of the literature recalled
    in the preceding paragraphs, and what changes is the class they are proved for, a class defined
    by a condition on $\{\Delta_k\}$ rather than by a parametric form. That change is the point of
    the paper, and it is what lets mechanisms outside any single parametric family be compared on
    equal terms.

    \noindent\textbf{Outline.}
    Section~\ref{sec: generic TR} introduces the generic trust-region algorithm together with its standing assumptions. Section~\ref{sec:global convergence results} states the three structural conditions, defines weak and strong admissibility, and establishes the corresponding convergence results.
    We then analyse the worst-case complexity of strongly admissible mechanisms in Section~\ref{sec:first-order worst-case complexity}.
    Section~\ref{sec:Retrospective trust-region radius} extends the retrospective trust-region analysis to linearly growing model Hessians, and Section~\ref{sec:radius update mechanisms} verifies admissibility for the five mechanism classes.
    In Section~\ref{sec:iteration complexity framework}, we build on the framework discussed in~\citet{CurtSche20}, extending it to three multiplicative factors and decoupled step acceptance, and we specialise the stochastic analysis of~\citet{WangYuan2022} to the deterministic setting.
    Section~\ref{sec:conclusion} concludes and points to \PartII{}, which addresses asymptotic radius behaviour and superlinear convergence.

    \noindent\textbf{Notation.}
    At iteration $k$, we write $f_k = f(x_k)$, $\nabla f_k = \nabla f(x_k)$, and 
    $\nabla^2 f_k = \nabla^2 f(x_k)$.
    For a matrix $A \in \mathbb{R}^{n \times n}$, $\norm{A}$ denotes the operator 
    norm induced by $\norm{\cdot}$.
    For the Euclidean norm, $\norm{A}$ coincides with the spectral norm.
    Unless stated otherwise, $\norm{\cdot}$ denotes the Euclidean norm for vectors 
    and the induced spectral norm for matrices.
    The model at iteration $k$ is denoted $m_k(s)$, with gradient $g_k$ 
    and Hessian $H_k$, where $H_k$ is either $\nabla^2 f(x_k)$ 
    or an approximation.
    The trial step $s_k$ yields the candidate point $x_k^{\text{cand}} = x_k + s_k$, accepted 
    or rejected based on the ratio of actual to predicted reduction $\rho_k$.
    We fix $0 \leq \eta \leq \eta_1 \leq \eta_2 < 1$ throughout.
    The threshold $\eta$ governs step acceptance, and we define the sets of successful and unsuccessful iterations 
    respectively as $\mathcal{S} = \{ k \geq 0 : \rho_k \geq \eta \}$ and $\mathcal{U} = \{ k \geq 0 : \rho_k < \eta \}$.
    The thresholds $\eta_1$ and $\eta_2$ are used by certain mechanisms to adjust 
    $\Delta_k$ more finely.
    The scaling parameters satisfy $0 < \gamma_1 \leq \gamma_2 < 1 < \gamma_3$: 
    $\gamma_1$ and $\gamma_2$ govern contraction, and $\gamma_3$ governs expansion.
    Finally, following~\citet{BlanCartMeniSche19, ChenMeniSche18, ConnGoulToin00}, constants 
    are denoted $\kappa$ with a descriptive subscript:
    \begin{multicols}{2}
        \begin{description}
        \item[$\mykappa{lbf}$] \emph{\underline{l}ower \underline{b}ound on the \underline{f}unction}
        \item[$\mykappa{mdc}$] \emph{\underline{m}odel \underline{d}e\underline{c}rease}
        \item[$\mykappa{umh}$] \emph{\underline{u}pper bound on the \underline{m}odel's \underline{H}essian}
        \item[$\mykappa{lbd}$] \emph{\underline{l}ower \underline{b}ound on \underline{d}elta}
        \end{description}
    \end{multicols}

\section{Generic trust-region algorithm and its subproblem} \label{sec: generic TR}

    \subsection{Basic trust-region algorithm (BTR)} \label{subsec:basic trust-region algorithm}

    Algorithm~\ref{alg:generic BTR}  outlines the \emph{basic trust-region} (BTR) method.
	At each iteration $k$, a model $m_k$ is built to serve as a local approximation of $f$ in a neighbourhood of the current iterate $x_k$ called the \emph{trust region}.
    We consider  quadratic models, commonly used in the literature, of the form
    \begin{equation}
        m_k(s) = f_k + g_k^\top s + \frac{1}{2} s^\top H_k s,
        \label{eqn:local quadratic model}
    \end{equation} 
    in the \emph{trust region} defined as the ball $\mathcal{B}_k \overset{def}{=} \left\lbrace x \in \Rp \, \text{ s.t. } \, \| x - x_k \|_k \leq \Delta_k \right\rbrace$, and where \( H_k \) is the Hessian of the objective function at \( x_k \), or an approximation.
    \citet{ConnGoulToin00} and~\citet{CartGoulToin22} discuss the choice of the model in detail, including %where they consider 
    linear or quadratic models, and more complex ones based on higher-order derivatives or other approximations.
    A trial step $s_k$ that achieves a ``sufficient'' decrease condition on the model is computed.
    Subsection~\ref{subsec:sufficient decrease condition} discusses these sufficient decrease conditions.
    Then, at Step~3, the ratio $\rho_k$, defined in~\eqref{eqn:definition rho}, evaluates how accurately the model predicts the actual reduction in the objective function. 
    Values close to one indicate good agreement between the model and the function, whereas values close to zero reveal poor predictive quality. 
    A threshold parameter $\eta \in [0,1)$ determines whether the trial step is accepted (\emph{successful}) or rejected (\emph{unsuccessful}). 
    When $\rho_k < \eta$, the model is deemed unreliable, typically because the trust region is too large, and the radius is reduced so that a new trial step is computed in a smaller region where the model is more accurate.
    \begin{algorithm}[htbp] 
        \small
        \SetAlgoLined   
        \DontPrintSemicolon 
        
        \textbf{Step 0: Initialisation.} Choose $\eta \geq 0$. \\
        Set an initial iterate $x_0 \in \reals^p$, an initial trust-region radius $\Delta_0 \in (0 \, , \, \infty) $. 
        
        Compute $f(x_0)$ and set $k = 0$. \; ~\\
        
        \textbf{Step 1: Model definition and step computation.} Define a local model $m_k$ of $f$ in the neighbourhood $\mathcal{B}_k$. \; ~\\
        
        \textbf{Step 2: Model minimisation.} Compute a step $s_k$ that "sufficiently reduces the model" $m_k$ with $x_k + s_k \in \mathcal{B}_k$, and define the candidate point $x_k^{\text{cand}} = x_k + s_k$. \; ~\\
        
        \textbf{Step 3: Acceptance of the trial point.} Compute $f(x_k^{\text{cand}})$ and define
        \begin{equation}
            \rho_k = \dfrac{f(x_k) - f(x_k^{\text{cand}})}{m_k(0) - m_k(s_k)}
            \label{eqn:definition rho}
        \end{equation}
        If $\rho_k \geq \eta$, then define $x_{k+1} = x_k + s_k$ ; otherwise take $x_{k+1} = x_k$. \; ~\\
        
        \textbf{Step 4: Trust-region radius update.} Compute $\Delta_{k+1}$ according to some update rule.  \; ~\\
        
        \textbf{Step 5: Next iteration.} Increment $k$ by 1 and repeat \textbf{Steps 1 to 5} while some stopping criterion is not met.  \;
        
        \caption{Basic trust-region algorithm (BTR).}
        \label{alg:generic BTR}
    \end{algorithm}
    The choice of the acceptance parameter $\eta$ plays a key theoretical role in trust-region algorithms 
    \citep[Chap.~17]{ConnGoulToin00}. See also~\citet{GouldOrbanSartenaerToint05, More83, Yuan98, Yuan00}. 
    Several authors, including~\citet{Powe75, Powe84, Yuan00, Yuan15}, have used $\eta = 0$, 
    thereby accepting any step that decreases the objective function. 
    However, this choice only ensures that at least one accumulation point, if there are any, is first-order critical, that is
    $\liminf_{k \to \infty} \|g_k\| = 0$. 
    \citet{Yuan98} illustrates this limitation by constructing a one-dimensional example that cycles among three points, two of which are non-stationary. 
    For stronger convergence guarantees, it is preferable to select a small positive constant $\eta > 0$ 
    \citep{ConnGoulToin00, NoceWrig06, Yuan00}, ensuring that 
    \begin{equation}
        \lim_{k \to \infty} \|g_k\| = 0.
        \label{eqn:strong convergence}
    \end{equation}

    \subsection{Assumptions on function} \label{subsec:assumptions for global convergence}

        We state the assumptions used to establish first-order convergence.
        The first assumption requires the objective to be bounded below. Points then exist where the gradient norm is arbitrarily small, possibly zero, see~\citet{Ekeland74}.
        \begin{assum}{f.Bb}[bounded below]
				$f$ is bounded below on $\R^p$. Formally, 
				$ \exists \, \mykappa{lbf} \in \RR$ %\vert \,
                such that $\forall x \in \Rp$ ,  $f(x) \geq \mykappa{lbf}$.
				\label{as: Bounded below}
		\end{assum}
        Assumption~\ref{as: Lipschitz continuous gradient} is a common assumption in optimisation. It ensures that the gradient does not change too abruptly. %This can seem strong, but i
        It is required for the worst-case complexity analysis. 
        \citet{CartGoulToin22} show that without Lipschitz continuity, for any finite sequence of test points, one can construct a one-dimensional counterexample $f$ where the gradient norms do not decrease below any threshold.
        Therefore, it is impossible to give any guarantees for a finite sequence of test points to get the norm of the gradient to fall below any prescribed accuracy threshold. 
        \begin{assum}{f.Lg}[Gradient Lipschitz continuity]
            $f$ is differentiable with Lipschitz continuous gradient with constant \( L > 0 \)
            \[
                \|\nabla f(x) - \nabla f(y)\| \leq L \|x - y\|, \quad \forall x, y \in \RR^p.
            \]
            \label{as: Lipschitz continuous gradient}
        \end{assum}        
        
        \begin{remark}
            Another assumption often made in the context of trust-region methods, which can be found in~\citet{Powe75, Powe84, Stei83, Hei03, YuanFan01}, is that the gradient function \( \nabla f \) is uniformly continuous. However, it is implied by Lipschitz continuity of the gradient.
        \end{remark}
        
    \subsection{Assumptions on the model} \label{subsec:sufficient decrease condition}

        Assumption~\ref{as: model agreement with function} requires that the model $m_k$ matches the objective function $f$ at the current iterate $x_k$ up to first order. 
        This condition ensures that the model accurately reflects both the value and the gradient of the objective function at $x_k$,  providing a reliable local approximation for the trust-region step computation.
        \begin{assum}{m.Ex}[exact to first order at $x_k$]
            For all $k$, we have $m_k(0)  = f(x_k)  $ and $g_k \overset{def}{=} \nabla_x m_k(0)  = \nabla_x f(x_k)$.
            \label{as: model agreement with function}
        \end{assum}

        The trial step $s_k$ is obtained by approximately minimising the model $m_k(s)$ in the trust region $\mathcal{B}_k$, ensuring a ``sufficient'' decrease in the model. To motivate this condition, we recall the descent property in gradient descent. With the update $x_{k+1} = x_k - \alpha_k \nabla f(x_k)$ and Lipschitz continuity of the gradient (see \ref{as: Lipschitz continuous gradient} in Section~\ref{subsec:assumptions for global convergence}), one obtains the descent lemma~\citep{Nest14},
        $
        \Bigl(\alpha_k - \tfrac{L}{2}\alpha_k^2\Bigr)\|\nabla f(x_k)\|^2 \leq f(x_k) - f(x_{k+1}),
        $
        which, for $\alpha_k = 1/L$, reduces to
        $
        \frac{1}{2L}\|\nabla f(x_k)\|^2 \leq f(x_k) - f(x_{k+1}).
        $
        This inequality underpins the convergence analysis of gradient descent in nonconvex settings~\citep{Nest14}, implying in particular that, if $f$ is bounded below, then $\liminf_{k \to \infty}\|\nabla f(x_k)\| = 0$. With non-constant step sizes, a line search can ensure sufficient decrease via the Armijo condition.  

        By analogy, trust-region methods enforce sufficient decrease through the \emph{Cauchy step}, obtained by minimising the quadratic model $m_k(s)$ along the gradient direction within the trust region. When the model is quadratic, the exact minimiser along this arc can be computed explicitly.
        The resulting step, referred to as the \emph{Cauchy step}, is defined as
        \begin{equation}
            s_k^C = - t_k^C g_k \quad \text{where} \quad  t_k^C = \underset{t \geq 0, \; - t g_k \in \mathcal{B}_k}{\mathrm{argmin}} \, \left\lbrace m_k\left(- t g_k \right) \right\rbrace .
            \label{eqn:Cauchy step}
        \end{equation}
        and $x^C = x_k + s_k^C$ is defined as  the associated \emph{Cauchy point}, while the model decrease $m_k(0) - m_k(s_k^C)$  at the Cauchy point, is also referred to as the \emph{Cauchy decrease}~\citep{ConnGoulToin00}.
        It is an important feature of the convergence theory as it is possible to determine a minimum decrease when moving to this point, as shown in Theorem~\ref{th: Cauchy decrease} (see~\citealp{ConnGoulToin00, NoceWrig06}). 
        
            \begin{theorem}
                For the quadratic model~\eqref{eqn:local quadratic model}, the model decrease at the Cauchy step~\eqref{eqn:Cauchy step} satisfies the inequality
                \begin{equation}
                    m_k(0) - m_k\left(s_k^C\right) \geq \dfrac{1}{2} \Vert g_k \Vert \min \left[ \dfrac{\Vert g_k \Vert}{\| H_k \|} , \, \Delta_k \right].
                    \label{eqn: Cauchy decrease}
                \end{equation}
                where, by convention, for  \( \|H_k\| = 0 \), we set \( \frac{\Vert g_k \Vert}{\| H_k \|} = +\infty \).
                \label{th: Cauchy decrease}
            \end{theorem}
            \begin{proof}
                See~\citet[Theorem 6.3.1]{ConnGoulToin00} or~\citet[Theorem 10.1]{ConnScheVice09_introduction}. 
            \end{proof}
            \begin{remark}
                When the Hessian is zero, the model is linear and the Cauchy point reduces to the point on the boundary in the direction of the gradient, i.e.\ \( s_k^C = - \Delta_k \frac{g_k}{\|g_k\|} \).
            \end{remark}

            \citet{Powe75,Powe84, ConnGoulToin00} show that only a fraction of the decrease provided by the Cauchy point is needed to ensure global convergence to first-order critical points.
            This is formalised in Assumption~\ref{as: Cauchy decrease} below.
            \begin{assum}{a.Cd}[Sufficient decrease condition: \emph{Cauchy decrease}]
                For all $k$,
                            $$ m_k(0) - m_k(s_k) \geq \dfrac{\mykappa{mdc} }{2} \Vert g_k \Vert \min \left[ \dfrac{\Vert g_k \Vert}{\| H_k \|}, \Delta_k \right]  $$
                    for some constant $ \mykappa{mdc} \in (0, 1)$.
                \label{as: Cauchy decrease}
            \end{assum}

        We turn to the assumptions on the model Hessian. \\
        The model Hessian can be uniformly bounded, as expressed below.
        \begin{assum}{m.Umh}[Uniform bound on the model Hessian]
            There exists a constant \( \mykappa{umh} > 0 \) such that the Hessian of the model \( m_k \) is uniformly bounded across all iterations \( k \), i.e.\ 
                \[
                    \| H_k \| \leq \mykappa{umh} < +\infty, \quad \forall k \in \NN.
                \]
            \label{as: model hessian}
        \end{assum}
        If the Hessian bound is relaxed, convergence can still be guaranteed provided that the growth of the Hessians remains moderate. In particular, results in~\citet{GrapigliaYuanYuan15, GrapigliaYuanYuan16, Powe84, Yuan85, Yuan00} provide convergence guarantees under the following weaker assumption on the model Hessian.
        \begin{assum}{m.AHg}[admissible Hessian growth]
            For all $k \in \NN$, the Hessian of the model $m_k$ satisfies $\| H_k \| \leq M_k < +\infty$,  where \( M_k \) is defined in~\eqref{eqn: M_k definition}, and %the sequence \( \{ M_k \} \) satisfies 
            \begin{equation}
                \sum_{k=0}^\infty \dfrac{1}{M_k} = +\infty,
                \label{eqn: linear growth bound on model Hessian}
            \end{equation}
            \label{as: Linear growth bound model Hessian}
        \end{assum}

        This assumption generalises a standard requirement in the literature~\citep{Powe70, Powe75, Toin79}, which allows the model Hessian to grow with the step size when using secant-based updates such as BFGS, DFP, or SR1,
        \begin{assum}{m.G-St}[growth in the sum of steps]
            The Hessian approximation $H_k$ satisfies, for some constants $\mykappa{umh}, \mykappa{umh}^{\prime} > 0$,
            \[
                \| H_k \| \leq \mykappa{umh} + \mykappa{umh}^{\prime} \sum_{i=0}^{k} \| s_i \|,
            \]
            \label{as: model hessian based on sum of steps}
        \end{assum}

        Another useful practice, subsumed by Assumption~\ref{as: Linear growth bound model Hessian}, is to let the Hessian bound grow linearly with the iteration index.%$k$.
        This situation arises, for instance, when update formulas used for second-derivative approximations ensure that the differences \(\left\{ \| H_{k+1} - H_k \| ; \; k = 1, 2, 3, \ldots \right\}\) are uniformly bounded~\citep{Powe84}.
        \begin{assum}{m.G-Lin}[growth linear in $k$]
        	There exist constants $\mykappa{umh} > 0$, $\mykappa{umh}^{\prime} > 0$, such that, for all $k \in \NN$,
            \[
                M_k \leq \mykappa{umh}  + \mykappa{umh}^{\prime} \, k.
            \]       
            \label{as: growth bound model Hessian with iteration}
        \end{assum}
        
        \subsection{Trust-region subproblem} \label{sec:TR subproblem}

            The trust-region subproblem~\eqref{eq: TR subproblem}, widely studied 
            \citep{CartGoulToin22, ConnGoulToin00, More83, MoreSorensen1983, NoceWrig06, Stei83, Yuan2000}, is a key component of trust-region methods.
            In our context, the subproblem involves minimising a quadratic model~\eqref{eqn:local quadratic model} constrained to a ball of radius \( \Delta_k \) centred at the current iterate \( x_k \), i.e.,
            \begin{equation}
                \min_{s \in \Rp} m_k(s) \quad \text{subject to} \quad \|s\| \leq \Delta_k.
                \label{eq: TR subproblem}
            \end{equation}
            When \( H_k \) is positive definite, the unconstrained exact solution to this problem is given by $H_k s_k = - g_k$, but it can be outside the trust region if \( \Delta_k \) is small.
            Therefore, the solution to the trust-region subproblem is often found on the boundary of the trust region, i.e., \( \|s_k\| = \Delta_k \), especially when  \( H_k \) is not positive definite.
            Solving the trust-region subproblem exactly is expensive for large-scale problems, and iterative methods give an approximate solution instead, such as the truncated Conjugate Gradient (CG) method or Krylov subspace methods~\citep{Stei83, More83, NoceWrig06, CartGoulToin22}.

    \subsection{General dynamics of the trust-region radius} \label{sec:general dynamics of the trust-region radius}

            We establish general dynamics of $\Delta_k$ for Algorithm~\ref{alg:generic BTR}.
            These results are classical in the trust-region literature~\citep{ConnGoulToin00, NoceWrig06, CartGoulToin22}. 
            They hold independently of any specific update rule. We focus on the step norm $\|s_k\|$, which motivates the condition on the radius adopted in the convergence analysis. 

            In the following, we work under Assumptions~\ref{as: Bounded below}, \ref{as: Lipschitz continuous gradient}, \ref{as: model agreement with function}, and \ref{as: Cauchy decrease}.
			\begin{theorem}
				Consider any non-negative constant \( \eta \in (0, 1) \) and let \( \Delta_k \) be the trust-region radius at iteration \( k \). If we have 
					\begin{equation}
						\Delta_k \leq \dfrac{\mykappa{mdc}(1 - \eta) \| g_k \| }{L + \| H_k \|},
                        \label{eqn: successful iteration for low Delta}
					\end{equation}
				then $ \rho_k \geq \eta. $ \\
                This is also true if we assume the inequality holds for $\| s_k \|$ instead of $\Delta_k$, i.e., if we have
                    \begin{equation}
                        \| s_k \| \leq \dfrac{\mykappa{mdc}(1 - \eta)\| g_k \|}{L + \| H_k \|}.
                        \label{eqn: successful iteration for low s_k}
                    \end{equation}
                \label{thm: successful iteration for low Delta}
			\end{theorem}
            \begin{proof}
                The result follows directly from the Cauchy decrease assumption~\ref{as: Cauchy decrease}. A proof can be found in~\citet[Lemma 2.3.3]{CartGoulToin22} or~\citet[Theorem 6.4.2]{ConnGoulToin00} for the case involving $\Delta_k$. The argument for $s_k$ is similar, we include it here for completeness.
                From~\eqref{eqn:definition rho} and~\eqref{eqn:local quadratic model}, we have
                \begin{align*}
                    | \rho_k - 1 | &= \dfrac{| f(x_k + s_k) - f(x_k) - g_k^T \, s_k -\frac{1}{2} s_k^T H_k s_k |}{m_k(0) - m_k(s_k) } \\
                    &\leq \dfrac{ \left(L + \| H_k \| \right) \, \| s_k \|^2 }{\mykappa{mdc} \| g_k \| \min\left\{ \| s_k \| \, , \, \dfrac{ \| g_k \| }{\| H_k \|} \right\}},
                \end{align*}
                where the inequality follows from triangular inequality, the Lipschitz continuity of the gradients (\ref{as: Lipschitz continuous gradient}) with Lemma 1.2.3~\citet{Nest18} and \ref{as: Cauchy decrease} with $\norm{s_k} \leq \Delta_k$.
                Since $\mykappa{mdc}, (1 - \eta) \in (0,1)$ and $L > 0$, hypothesis~\eqref{eqn: successful iteration for low s_k} implies $\|s_k\| \leq \frac{\|g_k\|}{\|H_k\|}$.
                Thus, still from~\eqref{eqn: successful iteration for low s_k}, we have 
                \[
                    |\rho_k - 1|
                    \;\leq\; \frac{(L + \|H_k\|)}{\mykappa{mdc}\,\|g_k\|\,\|s_k\|}\,\|s_k\|^2
                    \;\leq\; 1 - \eta,
                \]
                which yields $\rho_k \geq \eta$.
            \end{proof}

            \begin{corollary}
            For any non-negative constant \( \eta \in (0, 1) \), whenever $\rho_k < \eta$, we have 
            \begin{equation}
                \Delta_k \geq \| s_k \| > \dfrac{\mykappa{mdc}(1 - \eta)\| g_k \| }{L + \| H_k \|}.
                \label{eqn: delta_k lower bound}
            \end{equation}
            \label{cor: delta_k lower bound}
        \end{corollary}
        \begin{proof}
            This is a direct consequence of Theorem~\ref{thm: successful iteration for low Delta} by contraposition.
        \end{proof}
        Theorem~\ref{thm: successful iteration for low Delta} and Corollary~\ref{cor: delta_k lower bound} together show that, whenever the radius is reduced below a certain threshold depending on the gradient norm and curvature, the iteration is guaranteed to be successful. 

        The update rule must adjust $\Delta_k$ dynamically to face two requirements. Reduce the radius when the model is unreliable, and prevent it from staying unnecessarily below the threshold, which would slow progress.
        This motivates Condition~\ref{cond: delta_k lower bound} presented in the next section.

        \subsection{First-order global convergence} \label{subsec:cornerstone result for First order global convergence}

            Theorem~\ref{th: firt-order TR CV} establishes the global success of trust-region methods.
            It is well established and classical in the literature~\citep{ConnGoulToin00, NoceWrig06}, but we include it here to emphasise its role in our analysis. The property is intrinsic to Algorithm~\ref{alg:generic BTR} as it holds \textbf{independently} of the update rule.
            Provided a fixed fraction of the Cauchy decrease and the standard assumptions on the objective function are satisfied, any accumulation points of the sequence, if some exist, are first-order critical.
            \begin{theorem}
                Suppose that the function satisfies Assumptions~\ref{as: Bounded below} 
                and \ref{as: Lipschitz continuous gradient}, 
                the model $m_k$ satisfies \ref{as: model agreement with function} 
                and \ref{as: model hessian}. 
                Furthermore, the steps $s_k$ satisfy the Cauchy decrease condition \ref{as: Cauchy decrease}, and in Algorithm~\ref{alg:generic BTR}, $\eta > 0$. \\
                If we \textbf{already have} that $\liminf_{k \to \infty} \|g_k\| = 0$, then 
                \begin{equation}
                    \lim_{k \to \infty} \|g_k\| = 0.
                \end{equation}
                    \label{th: firt-order TR CV}
            \end{theorem}
\begin{proof}
                The result is classical and we do not reproduce its proof: see
                \citet[Thm.~6.4.6]{ConnGoulToin00} and \citet[Thm.~4.6]{NoceWrig06}, where it is
                established together with the $\liminf$ property for a specified update rule.
                Two points locate what is used here. The $\liminf$ is a hypothesis rather than a
                conclusion, supplied by
                Theorems~\ref{thm: Liminf of the gradient norm under strong assumption}
                and~\ref{thm: Liminf of the gradient norm under weaker assumption} under the
                structural conditions, so only the second half of the classical argument is
                needed, and that half uses no property of the update rule beyond
                $\norm{s_k} \leq \Delta_k$, which is why the statement holds for any rule. And
                the case of finitely many successful iterations, which the classical proof settles
                through the contraction its rule performs, is immediate under that hypothesis: the
                iterate is then constant from some $k_0$ on, so
                $\liminf_{k \to \infty} \norm{g_k} = \norm{g_{k_0}}$, which the hypothesis makes
                zero.
            \end{proof}
            \begin{remark}
                If we only assume that $\eta = 0$, then for iterations with $\rho_k \in (\eta, \eta_1)$,
                the algorithm can accept steps with arbitrarily small decrease allowing to have accumulation points with non-zero gradient norm, as exhibited by~\citet{Yuan98} on a one-dimensional example.
            \end{remark}
        
    \section{Abstract conditions and convergence analysis of the generic TR algorithm} \label{sec:global convergence results}

        Therefore, it suffices to analyse the case $\liminf_{k \to \infty} \|g_k\| = 0 $.   
        Subsection~\ref{subsec:liminf behaviour of the gradient norm under strong assumption} shows that the liminf property holds whenever Condition~\ref{cond: delta_k lower bound} is satisfied and the update rule enforces uniform contraction on unsuccessful iterations (Condition~\ref{cond: decrease of the trust-region radius}).
        In Subsection~\ref{subsec:liminf behaviour of the gradient norm under weaker assumption}, convergence is established with respect to the weaker assumption~\ref{as: Linear growth bound model Hessian}, under Condition~\ref{cond: delta_k lower bound},~\ref{cond: decrease of the trust-region radius}, and~\ref{cond: increase of the trust-region radius}, the latter requiring uniform growth of the trust-region radius on successful iterations. Moreover, we stress the necessity of Condition~\ref{cond: increase of the trust-region radius} to compensate for the weaker bound on the Hessian.

        \subsection{The three structural conditions}

            The three conditions on the radius update rule that drive the convergence analysis are stated below.

            \begin{cond}{C.LbR}[Lower bound on the radius]
                There exists a constant $ 0 < \mykappa{lbd} \leq 1 $ such that the trust-region radius \( \Delta_k \) satisfies the following lower bound for all \( k \geq 0 \),
                \begin{equation}
                    \Delta_k \geq \frac{\mykappa{lbd}}{M_k} \min_{0 \leq i \leq k} \| g_i \|,
                    \label{lem: delta_k lower bound}
                \end{equation}
                where $\{M_k\} $ is a sequence of non-decreasing positive numbers that upper bound the model Hessian norm.
                \label{cond: delta_k lower bound}
            \end{cond}
            We choose $\mykappa{lbd} \leq 1$ without loss of generality. It simplifies formulas in proofs and discussions. Also, the sequence $\{M_k\}$ can be defined in various ways. A natural choice is 
            \begin{equation}
                M_k = L + \max_{0 \leq i \leq k} \| H_i \|.
                \label{eqn: M_k definition}
            \end{equation}	   

            \begin{cond}{C.DecR}[Contraction on unsuccessful iterations]
                There exists $\underline{\gamma_2} \in (0\, , \, 1)$ and $\myeta{dec} > 0$ such that, for all
                $k$, if $\rho_k < \myeta{dec}$, then
                \[
                    \Delta_{k+1} \;\leq\; \underline{\gamma_2} \, \Delta_k.
                \]
                \label{cond: decrease of the trust-region radius}
            \end{cond}

            \begin{cond}{C.Exp/DecR}[Bounded expansion on successful iterations, decrease otherwise]
                There exists $\underline{\gamma_2} \in (0\, , \, 1)$, $\overline{\gamma_3} > 1$ and $\myeta{inc} > 0$ such that, for all $k$, 
                \begin{description}
                        \item[(decrease)] if $\rho_k < \myeta{inc}$, then $\Delta_{k+1} \leq \underline{\gamma_2} \, \Delta_k$;
                        \item[(increase)] else $\rho_k \geq \myeta{inc}$, then $\Delta_{k+1} \leq \overline{\gamma_3} \,  \Delta_k$.
                \end{description}                    
                \label{cond: increase of the trust-region radius}
            \end{cond}
            Condition~\ref{cond: delta_k lower bound} prevents the radius from collapsing faster than the criticality measure. As long as the smallest gradient norm observed so far is bounded away from zero, so is $\Delta_k$.
            Condition~\ref{cond: decrease of the trust-region radius} forces a correction on the radius after a failed step. 
            Finally, Condition~\ref{cond: increase of the trust-region radius} caps the rate at which the radius can grow, preventing arbitrary jumps that would disconnect $\Delta_k$ from the local geometry.
            These three conditions define two classes of update rules.
            \begin{definition}[Weak and strong admissibility]
                A trust-region radius update rule is said to be
                \begin{enumerate}[(i)]
                    \item \emph{weakly admissible} if it satisfies Condition~\ref{cond: delta_k lower bound}~and~\ref{cond: decrease of the trust-region radius};
                    \item  \emph{strongly admissible} if it satisfies Condition~\ref{cond: delta_k lower bound}~ and~\ref{cond: increase of the trust-region radius}.
                \end{enumerate}
                \label{def: admissibility}
            \end{definition}
            Strong admissibility is the property used throughout the analysis. It yields global convergence under linearly growing model Hessians (Theorem~\ref{thm: Liminf of the gradient norm under weaker assumption}) and the optimal worst-case complexity $\calO(\varepsilon^{-2})$ (Theorem~\ref{thm: Upper bound successful iter BTR}).
            Weak admissibility is sufficient under the stronger assumption that $\|H_k\|$ is uniformly bounded (Theorem~\ref{thm: Liminf of the gradient norm under strong assumption}).

            \subsection{Liminf behaviour of the gradient norm under strong assumption} \label{subsec:liminf behaviour of the gradient norm under strong assumption}

            Under uniformly bounded Hessians (\ref{as: model hessian}), the lower bound Condition~\ref{cond: delta_k lower bound} and the standard assumptions, the $\liminf$ of the gradient norm goes to zero as $k \to \infty$, provided the update rule contracts the trust-region radius uniformly on unsuccessful iterations.
            \begin{theorem}[Lim inf of the gradient norm]
                Suppose that Algorithm~\ref{alg:generic BTR} is applied under Assumptions~\ref{as: Bounded below}, \ref{as: Lipschitz continuous gradient}, \ref{as: model agreement with function}, \ref{as: model hessian}, \ref{as: Cauchy decrease}. 
                Assume furthermore that $\Delta_k$ follows the lower bound Condition~\ref{cond: delta_k lower bound} and the decrease Condition~\ref{cond: decrease of the trust-region radius} with $\myeta{dec} \geq \eta$, but $\myeta{dec} \neq 0$.
                Then, we have that
                \[
                    \liminf_{k \to \infty} \|g_k\| = 0,
                \]
                \label{thm: Liminf of the gradient norm under strong assumption}
            \end{theorem}

            \begin{proof}
                We analyse two cases:
                \begin{enumerate}[(a)]
                    \item First, suppose the set $\left\{ k \in \NN \, : \, \rho_k \geq \myeta{dec} \right\}$ is finite.
                    Then there exists an index \( k_0 \) such that for all \( k \geq k_0 \),  \( \rho_k < \myeta{dec} \).
                    Then, by Condition~\ref{cond: decrease of the trust-region radius} and Condition~\ref{cond: delta_k lower bound}, we have for all \( k \geq k_0 \),
                    \[
                        \frac{\mykappa{lbd}}{L + \mykappa{umh}} \min_{0 \leq i \leq k} \| g_i \| \leq  \Delta_k \leq \underline{\gamma_2}^{k - k_0} \Delta_{k_0} \to 0
                    \]
                    which implies that \( \liminf_{k \to \infty} \| g_k \| = 0. \) 

                    \item Now, suppose that $\left\{ k \in \NN \, : \, \rho_k \geq \myeta{dec} \right\}$ is infinite.
                    By the Cauchy decrease condition~\ref{as: Cauchy decrease}, \ref{as: model hessian}, and Condition~\ref{cond: delta_k lower bound}, we have that for all \( k \) such that \( \rho_k \geq \myeta{dec} \),
                    \begin{align*}
                        f(x_k) - f(x_{k+1}) &\geq \myeta{dec} \frac{\mykappa{mdc}}{2} \| g_k \| \min\left[ \frac{\Vert g_k \Vert }{\| H_k \|} \, , \, \Delta_k \right] \\
                        &\geq \myeta{dec} \frac{\mykappa{mdc} \, \mykappa{lbd}}{2 (L+\mykappa{umh})} \min_{0 \leq i \leq k} \| g_i \|^2.
                    \end{align*}
                    Therefore, we have that
                    \begin{equation}
                        \begin{aligned}
                            \frac{\myeta{dec}\, \mykappa{mdc} \, \mykappa{lbd}}{2 (L+\mykappa{umh})} \sum_{ \left\{ k \, : \, \rho_k \geq \myeta{dec} \right\}} \left(\min_{0 \leq i \leq k} \| g_i \|^2\right) &\leq \sum_k f(x_k) - f(x_{k+1}) \\
                            &\leq f(x_0) - \mykappa{lbf} \\
                            &< +\infty.
                        \end{aligned}
                        \label{eq:bound_normgradient}                    	
                    \end{equation}
                        Then, %from \eqref{eq:bound_normgradient},
                        we have that \( \min_{0 \leq i \leq k} \| g_i \|^2 \) goes to 0 as \( k \) grows to $+\infty$, which implies that, 
                    $
                        \liminf_{k \to \infty} \| g_k \| = 0.
                    $
                \end{enumerate}
                The proof is complete.
            \end{proof}

        \subsection{Liminf behaviour of the gradient norm under weaker assumption} \label{subsec:liminf behaviour of the gradient norm under weaker assumption}

                We show that we can relax some assumptions on the Hessian but still guarantee that the $\liminf$ of the gradient norm goes to zero when $k \to \infty$, provided the update rule satisfies an additional condition and the growth of the model Hessians remains moderate. 
                The following Lemma~\ref{lem: sum of M_k} can be found in~\citet{Powe84, Yuan85, Yuan00, GrapigliaYuanYuan15, GrapigliaYuanYuan16}.
                \begin{lemma}
                    Let $\{\Delta_k \}$ and $\{M_k\}$ be two sequences of real numbers such that $\Delta_k \geq \dfrac{\tau}{M_k} > 0$, where $\tau > 0$.
                    Let $J$ be a subset of $\NN$, $\underline{\gamma_2} < 1$, $\overline{\gamma_3} > 1$.
                    Assume furthermore that
                    \begin{align}
                        \Delta_{k+1} & \leq \overline{\gamma_3} \, \Delta_k \text{ for all }k \in J, \label{eq:sumMk_first}\\
                        \Delta_{k+1} & \leq \underline{\gamma_2} \, \Delta_k \text{ for all }k \notin J, \notag \\
                        M_{k+1} & \geq M_k \text{ for all } k \geq 0, \notag \\
                        \sum_{k \in J} \frac{1}{M_k} & < +\infty. \notag 
                    \end{align}
                    Then, we have that
                    \[
                    \sum_{k=1}^\infty \frac{1}{M_k} < +\infty.
                    \]
                    \label{lem: sum of M_k}
                \end{lemma}
                \begin{proof}
                    See~\citet{Powe84} for a proof of the lemma.
                \end{proof}
                The following discussion highlights the importance of Condition~\ref{cond: increase of the trust-region radius} to ensure the liminf behaviour of the gradient norm. We present an example where all conditions in Lemma~\ref{lem: sum of M_k} are satisfied except for~\eqref{eq:sumMk_first}, 
                and the sum \( \sum_{k=1}^\infty \frac{1}{M_k} \) diverges.
                Take  \( M_k = k \), $\tau = 1$, some constant $\gamma \in (0, 1)$, and  define the sequence \( \{\Delta_k\}_{k \geq 0} \) as
                \[
                    \begin{cases}
                        \Delta_k = \dfrac{1}{k}, & \text{if } k = \ell^2 \text{ for some } \ell \geq 1, \\
                        \Delta_{k+p} = \dfrac{\gamma^{-2\ell - 1 + p}}{k+p} & \text{ for }  p = 1, 2, \ldots, 2\ell
                    \end{cases}
                \]
                This means that \( \Delta_k \) decreases by a factor of \( \gamma \) for \( k \) between two consecutive squares, and at the squares it jumps to a much larger value, for example $\Delta_1 = 1$, $\Delta_2 = \gamma^{-2}/2$, $\Delta_3 = \gamma^{-1}/3$, $\Delta_4 = 1/4$, $\Delta_5 = \gamma^{-4}/5$, $\Delta_6 = \gamma^{-3}/6$, $\Delta_7 = \gamma^{-2}/7$, $\Delta_8 = \gamma^{-1}/8$, $\Delta_9 = 1/9$, and so on.

                Therefore, we see that all $M_k$ are admissible, meaning that $\Delta_k \geq \frac{1}{M_k}$, for all $k \geq 1$.
                Now consider the set \( J = \{ \ell^2 \, : \, \ell \geq 1 \} \). Then, 
                \[
                    \sum_{k \in J} \frac{1}{M_{k}} = \sum_{\ell \geq 1} \frac{1}{\ell^2} < +\infty.
                \]
                We also have for all $k$, $M_{k+1} \geq M_k$, and for all $k \notin J$, $\Delta_{k+1} = \gamma \Delta_k$. No constant $\overline{\gamma_3} > 1$ can be defined such that $\Delta_{k+1} \leq \overline{\gamma_3} \Delta_k$ for all $k \in J$.
                Indeed, for $k \in J$, i.e.\ $k = \ell^2$ for some $\ell \geq 1$, we have that $\Delta_{k+1} = \dfrac{\gamma^{-2\ell}}{\ell^2+1}$ and $\Delta_k = \frac{1}{\ell^2}$, which means that
                \[
                    \limsup_{k \to +\infty\, ; \, k \in J} \frac{\Delta_{k+1}}{\Delta_k} = \limsup_{\ell \to +\infty} \frac{\gamma^{-2\ell} \, \ell^2}{\ell^2 + 1} = +\infty.
                \]
                Therefore, we have that the Condition~\eqref{eq:sumMk_first} is crucial in Lemma~\ref{lem: sum of M_k} to ensure that the sum \( \sum_{k=1}^\infty \frac{1}{M_k} \) converges. \\
                Moreover, when \( \limsup_{k \to +\infty\, ; \, k \in J} \frac{\Delta_{k+1}}{\Delta_k} < +\infty \), then we can find a constant \( \overline{\gamma_3} > 1 \) such that \( \Delta_{k+1} \leq \overline{\gamma_3} \Delta_k \) for all \( k \in J \). 
                This leads to the following question.  
                Is Condition~\eqref{eq:sumMk_first} necessary, or can it be relaxed to \( \limsup_{k \to +\infty\, ; \, k \in J} \frac{\Delta_{k+1}}{\Delta_k} = +\infty \) with some specific growth? 
                Also, is it sufficient to require that the sequence \( \{\Delta_k\} \) is bounded above, which is a standard condition in mathematical programming, i.e.\ \( \Delta_k \leq \Delta_{\max} \) for all \( k \geq 0 \)?
                \begin{theorem}[Liminf of the gradient norm]
                    Suppose that the trust-region algorithm (Algorithm~\ref{alg:generic BTR}) is applied under Assumptions~\ref{as: Bounded below}, \ref{as: Lipschitz continuous gradient}, \ref{as: model agreement with function}, \ref{as: Linear growth bound model Hessian}, and the Cauchy decrease \ref{as: Cauchy decrease}. 
                    Furthermore, $\Delta_k$ satisfies the lower bound Condition~\ref{cond: delta_k lower bound}, and the increase-decrease Condition~\ref{cond: increase of the trust-region radius} with $\myeta{inc} \geq \eta$, but $\myeta{inc} \neq 0$.
                    Then, we have
                    \begin{equation}
                        \liminf_{k \to \infty} \|g_k\| = 0.
                    \end{equation}
                    \label{thm: Liminf of the gradient norm under weaker assumption}
                \end{theorem}

                \begin{proof}
                    Again, we analyse two cases.
                    \begin{enumerate}[(a)]
                        \item Suppose $| \left\{ k \in \NN \, : \, \rho_k \geq \myeta{inc} \right\} | < +\infty$, then there exists an index \( k_0 \) such that for all \( k \geq k_0 \),  \( \rho_k < \myeta{inc} \).
                        By Condition~\ref{cond: delta_k lower bound} and Condition~\ref{cond: decrease of the trust-region radius}, we have that for all \( k \geq k_0 \),
                        \[
                            \frac{\mykappa{lbd}}{M_k}  \| g_k \| \leq \Delta_k \leq \underline{\gamma_2}^{k - k_0} \Delta_{k_0}.
                        \]
                        Since $\sum_{k \ge 0} \frac{\|g_k\|}{M_k} < +\infty$, then $\liminf_{k \to \infty} \|g_k\| > 0$ would imply $\sum_{k \ge 0} \frac{1}{M_k} < +\infty$, contradicting  Assumption~\ref{as: Linear growth bound model Hessian}.  Hence, $\liminf_{k \to \infty} \|g_k\| = 0$.

                        \item Now consider that $| \left\{ k \in \NN \, : \, \rho_k \geq \myeta{inc} \right\} | = +\infty$.
                        Again, suppose that \( \liminf_{k \to \infty} \| g_k \| > 0 \). 
                        By the Cauchy decrease condition~\ref{as: Cauchy decrease} and Condition~\ref{cond: delta_k lower bound}, we have for all large enough \( k \) such that \( \rho_k \geq \myeta{inc} \),
                        \begin{align*}
                            f(x_k) - f(x_{k+1}) &\geq \myeta{inc} \frac{\mykappa{mdc}}{2} \Vert g_k \Vert \min\left[ \frac{\Vert g_k \Vert }{\| H_k \|} \, , \, \Delta_k \right] \\
                            &\geq \myeta{inc} \frac{\mykappa{mdc} \, \mykappa{lbd}}{2} \left(\frac{\min_{0 \leq i \leq k} \| g_i \|^2}{M_k}\right) \\
                            &\geq \frac{\myeta{inc} \, \mykappa{mdc} \, \mykappa{lbd} \, \varepsilon^2}{2} \frac{1}{M_k}.
                        \end{align*}
                        Therefore, by similar arguments seen before, $\sum_{\{k : \rho_k \geq \myeta{inc}\}} \frac{1}{M_k} < +\infty$. 
                        Also, from the lower bound Condition~\ref{cond: delta_k lower bound},  $\Delta_k \geq \frac{\tau}{M_k}$, where $ \tau = \mykappa{lbd} \, \varepsilon > 0$. 
                        Then, by applying Lemma~\ref{lem: sum of M_k} with \( J = \left\{ k \in \NN \, : \, \rho_k \geq \myeta{inc} \right\}  \), we have that $ \sum_{k \geq 0} \frac{1}{M_k} < +\infty$.  
                        This contradicts Assumption~\ref{as: Linear growth bound model Hessian} meaning that \( \liminf_{k \to \infty} \| g_k \| = 0 \).
                    \end{enumerate}
                    The proof is complete.
                \end{proof}

    \section{First-order worst-case complexity} \label{sec:first-order worst-case complexity}
        
            Given \( \varepsilon > 0 \), a \emph{first-order complexity result} is a bound on the number of iterations required until \( \|g_k\| \leq \varepsilon \), 
            \begin{equation}
                T_{\varepsilon}^g = \inf \left\{ k \in \mathbb{N} \, : \, \|g_k\| \leq \varepsilon \right\},
                \label{eqn: first-order stop time}
            \end{equation}
            called an $\varepsilon$\emph{-first-order point}.
            
            Following~\citet{ConnGoulToin00, CartGoulToin10, CartGoulToin12, CartGoulToin22, CurtisLubbertsRobinson2018, GrapigliaYuanYuan15, GrapigliaYuanYuan16, Toint13}, we bound the total number of iterations to reach an $\varepsilon$-first-order point. All update rules satisfying the conditions of Theorem~\ref{thm: Liminf of the gradient norm under weaker assumption} share the same worst-case complexity.
            Their restrictions up to iteration $j$ are
			$
                \mathcal{S}_j = \{ 0 \leq k \leq j \, : \, k \in \mathcal{S} \}
            $
            and
            $
            \mathcal{U}_j = \{ 0 \leq k \leq j \, : \, k \in \mathcal{U}\}.
            $ 
        \begin{theorem}[First-order worst-case complexity]
            Suppose that Algorithm~\ref{alg:generic BTR} is applied with $\eta > 0$ under assumptions~\ref{as: Bounded below}, \ref{as: Lipschitz continuous gradient}. The model follows \ref{as: model agreement with function}  with uniform bound on the Hessian \ref{as: model hessian} and the steps verify the Cauchy decrease \ref{as: Cauchy decrease}. Moreover, $\Delta_k$ follows the lower bound Condition~\ref{cond: delta_k lower bound} and the decrease Condition~\ref{cond: decrease of the trust-region radius} with $\myeta{dec} \geq \eta$.  
            Then, the trust-region method attains a  $\varepsilon$-first-order point in at most
            $
                T_{\varepsilon}^g = \calO\left( \varepsilon^{-2} \right)
            $
            where \( \varepsilon > 0 \) is the desired accuracy in the norm of the gradient. 
            We get the following upper bounds,
                \begin{equation}
                    \vert \{ k < T_{\varepsilon} \, : \, \rho_k \geq \myeta{dec} \} \vert \leq \dfrac{2(f(x_0) - \mykappa{lbf}) \big(L + \mykappa{umh}\big)}{\myeta{dec}   \, \mykappa{mdc} \, \mykappa{lbd}} \, \varepsilon^{-2}  
                    \label{eqn: Upper bound successful iter BTR}
                \end{equation}
                and, with the increase-decrease Condition~\ref{cond: increase of the trust-region radius} ($\myeta{inc} \geq \eta$),  the number of unsuccessful iterations can be bounded by,
                \begin{equation}
                    \begin{aligned}
                        \vert \mathcal{U}_{T_{\varepsilon} - 1} \vert
                        &\leq \vert \{ k < T_{\varepsilon} \, : \, \rho_k < \myeta{inc} \} \vert \\
                        &\leq \dfrac{1}{\vert \ln\underline{\gamma_2} \vert}
                            \ln\left(\frac{\Delta_0 (L+\mykappa{umh})}{\mykappa{lbd}} \varepsilon^{-1} \right)
                            + \dfrac{\ln \overline{\gamma_3}}{\vert \ln \underline{\gamma_2} \vert }
                            \vert \{ k < T_{\varepsilon} \, : \, \rho_k \geq \myeta{inc} \} \vert
                    \end{aligned} 
                    \label{eqn: Upper bound unsuccessful iter BTR}
                \end{equation}
			\label{thm: Upper bound successful iter BTR}
        \end{theorem}
        \begin{proof}
            By Theorem~\ref{thm: Liminf of the gradient norm under strong assumption}, we have that \( \liminf_{k \to \infty} \| g_k \| = 0 \), and $ T_{\varepsilon} < +\infty $ for all \( \varepsilon > 0 \). 
            Then, for all \( k < T_{\varepsilon} \), we have with Condition~\ref{cond: delta_k lower bound} and \ref{as: model hessian} that $ \Delta_k \geq \frac{\mykappa{lbd} }{L + \mykappa{umh}}  \, \varepsilon$. Moreover, when $k \in  \{ k < T_{\varepsilon} \, : \, \rho_k \geq \myeta{dec} \} $, we have that
            $
                f(x_k) - f(x_{k+1})  \geq  \frac{\myeta{dec}  \, \mykappa{mdc} \, \mykappa{lbd}}{2 \big(L + \mykappa{umh}\big)}  \varepsilon^2,
            $
            from the Cauchy decrease condition~\ref{as: Cauchy decrease}. Therefore, 
            \begin{equation*}
                +\infty > f(x_0) - \mykappa{lbf}
                \geq \sum_{\rho_k \geq \myeta{dec}} \big( f(x_k) - f(x_{k+1}) \big)
                \geq \vert \{ k < T_{\varepsilon} \, : \, \rho_k \geq \myeta{dec} \} \vert
                    \frac{\myeta{dec} \, \mykappa{mdc} \, \mykappa{lbd}}{2 \big(L + \mykappa{umh}\big)} \varepsilon^2,
            \end{equation*}
			which gives~\eqref{eqn: Upper bound successful iter BTR}.
            \\
			We now bound the number of unsuccessful iterations $\vert \mathcal{U}_{T_{\varepsilon} - 1} \vert$.
            For this purpose, we assume that the update rule satisfies the increase-decrease Condition~\ref{cond: increase of the trust-region radius} with $\myeta{inc} \geq \eta$.
            The following arguments follow~\citet{GrapigliaYuanYuan15, GrapigliaYuanYuan16}. We adapt it to the class of strongly admissible mechanisms.
            From Condition~\ref{cond: delta_k lower bound} that
            $ \dfrac{1}{\Delta_k}  \leq \dfrac{L + \mykappa{umh}}{\mykappa{lbd}} \varepsilon^{-1} $, for all \( k < T_{\varepsilon}  \). 
            Then, from Condition~\ref{cond: increase of the trust-region radius},  
            for $ \rho_k \geq \myeta{dec} $, we have that
            $\dfrac{1}{\overline{\gamma_3} \, \Delta_k} \leq \dfrac{1}{ \Delta_{k+1}}$, otherwise $\dfrac{1}{\underline{\gamma_2} \, \Delta_k} \leq \dfrac{1}{\Delta_{k+1}}$.
            Therefore, we can write
            \begin{equation*}
                \frac{1}{\Delta_0} \underline{\gamma_2}^{{}^{- \vert \{ k < T_{\varepsilon} \, : \, \rho_k < \myeta{dec} \} \vert}} \overline{\gamma_3}^{{}^{- \vert \{ k < T_{\varepsilon} \, : \, \rho_k \geq \myeta{dec} \} \vert}}  \leq \frac{L+\mykappa{umh}}{\mykappa{lbd}} \varepsilon^{-1}
            \end{equation*}
            then, through logarithm and some algebraic manipulations, we get~\eqref{eqn: Upper bound unsuccessful iter BTR}. The proof is complete.
        \end{proof}
            If we consider Algorithm~\ref{alg:generic BTR} is applied with $\eta > 0$, and $ \myeta{dec} = \myeta{inc} = \eta$, we recover the same worst-case complexity bound $\calO\left( \varepsilon^{-2} \right)$ as the literature. Specifically, 
            \begin{equation}
                T_{\varepsilon} = \vert \mathcal{S}_{T_{\varepsilon} - 1} \vert + \vert \mathcal{U}_{T_{\varepsilon} - 1} \vert \quad \in \calO\left( \varepsilon^{-2} \right)
            \end{equation}
            where the number of successful iterations is bounded by,
            \begin{equation}
                \vert \mathcal{S}_{T_{\varepsilon} - 1} \vert  \leq \dfrac{2(f(x_0) - \mykappa{lbf}) \big(L + \mykappa{umh}\big)}{\eta  \, \mykappa{mdc} \, \mykappa{lbd}} \, \varepsilon^{-2} 
            \end{equation}
            and, with the increase-decrease Condition~\ref{cond: increase of the trust-region radius} ($\myeta{inc} \geq \eta$),  the number of unsuccessful iterations can be bounded by,
            \begin{equation}
                \vert \mathcal{U}_{T_{\varepsilon} - 1} \vert  \leq \dfrac{1}{\vert \ln\underline{\gamma_2} \vert} \ln\left(\frac{\Delta_0 (L+\mykappa{umh})}{\mykappa{lbd}} \varepsilon^{-1} \right) + \dfrac{\ln \overline{\gamma_3}}{\vert \ln \underline{\gamma_2} \vert }	\vert \mathcal{S}_{T_{\varepsilon} - 1} \vert 
            \end{equation}
        \begin{remark}
            \label{rem: tight complexity bound}
            The bound $\calO(\varepsilon^{-2})$ is tight.
            \citet{CartGoulToin22} construct a worst-case example showing that no
            first-order algorithm, one that uses only function and gradient
            evaluations, can reach an $\varepsilon$-first-order point in fewer
            iterations.
            \end{remark}

        \begin{remark}
            \label{rem: C3 needed for complexity}
            Weak admissibility does not suffice for the complexity bound. Without Condition~\ref{cond: increase of the trust-region radius}, the radius can grow arbitrarily on successful iterations, and the number of unsuccessful iterations can be unbounded.
        \end{remark}
        
    \section{Retrospective trust-region under linearly growing Hessians} \label{sec:Retrospective trust-region radius}

            The retrospective trust-region method, introduced by~\citet{BastMalmMoufToinToma10}, refines the radius update by re-evaluating the previous step in light of the current model. The full algorithm is given in Algorithm~\ref{alg:RTR}. 
            At iteration $k$, after constructing the model $m_k$ centred at the current iterate $x_k$ (Step 1.a), the algorithm computes the \emph{retrospective ratio}
            \begin{equation}
                \widetilde\rho_k
                \;=\;
                \frac{f(x_{k-1}) - f(x_k)}
                    {m_k(x_{k-1}) - m_k(x_k)}\,,
                \label{eq:retrospective ratio}
            \end{equation}
            in Step 1.b to assess the quality of the previous step taken at iteration \( k-1 \) from the perspective of the new model.
            A large $\widetilde\rho_k$ confirms that the previous step is consistent with the updated model and justifies maintaining or enlarging the radius. A small or negative $\widetilde\rho_k$ exposes a model mismatch revealed only after the step was taken, and triggers a corrective contraction.
            The radius update at iteration $k$ uses two cases, based on the standard ratio $\rho_{k-1}$.
            If $\rho_{k-1} < \eta_1$, the previous step was rejected at the time and the radius is contracted using a safeguard factor $\theta_{k-1}$.
            If $\rho_{k-1} \geq \eta_1$, the previous step was accepted. The algorithm then computes the retrospective ratio $\widetilde\rho_k$ and the radius is updated according to its value relative to thresholds $\tilde\eta_1, \tilde\eta_2 \in (0,1)$.

            \citet{BastMalmMoufToinToma10} analysed the retrospective update under uniform boundedness of the model Hessian. We extend their analysis in two directions. 
            First, we prove that the retrospective update is strongly admissible, satisfying Conditions~\ref{cond: delta_k lower bound}, and~\ref{cond: increase of the trust-region radius} under both the uniform boundedness assumption~\ref{as: model hessian} and the linear-growth assumption~\ref{as: Linear growth bound model Hessian}.
            Strong admissibility recovers the convergence guarantees of~\citet{BastMalmMoufToinToma10} via Theorems~\ref{thm: Liminf of the gradient norm under strong assumption} and~\ref{th: firt-order TR CV}, and yields the worst-case complexity $\calO(\varepsilon^{-2})$ to reach an $\varepsilon$-first-order point via Theorem~\ref{thm: Upper bound successful iter BTR}, a new result.
            Second, the analysis extends to the linear-growth setting, where $\liminf \|g_k\| = 0$ holds without uniform boundedness of the model Hessian via Theorem~\ref{thm: Liminf of the gradient norm under weaker assumption}.
            
                \begin{algorithm}[htbp]
                    \small
                    \SetAlgoLined
                    \DontPrintSemicolon

                    \textbf{Step 0: Initialisation.}
                    Choose constants $0 \leq \eta \leq \eta_1 < 1$ (with $\eta_1 \neq 0$), $0<\tilde\eta_1 \le \tilde\eta_2 < 1$, and $0<\gamma_1 \leq \gamma_2 < 1 < \gamma_3$. \; ~\\

                    \textbf{Step 1.a: Model definition.}
                    Define $m_k(s) = f(x_k) +  g_k^\top s  + \tfrac12  s^\top H_k s $. \; ~\\

                    \textbf{\phantom{Step} 1.b: Retrospective radius update.}
                    \noindent If $k=0$, skip to \textbf{Step 2}. \\[6pt]

                    \noindent If $\rho_{k-1} < \eta_1$ then 
                    $
                        \Delta_k =
                        \begin{cases}
                        \displaystyle\min\!\big\{\gamma_2 \|s_{k-1}\|,\, \max[\gamma_1,\theta_{k-1}]\,\Delta_{k-1}\big\}, & \rho_{k-1}<0 \text{ and }  \theta_{k-1} \text{ defined in~\eqref{eq:theta_k}},\\[6pt]
                        \gamma_2 \|s_{k-1}\|, & \text{ otherwise}.
                        \end{cases}
                    $

                    \noindent Else, compute the retrospective ratio $\widetilde\rho_k$ in~\eqref{eq:retrospective ratio} and  $\tilde\theta_k$  computed similarly to~\eqref{eq:theta_k}, then set
                    \[
                        \Delta_k =
                        \begin{cases}
                        \max\{\gamma_3 \|s_{k-1}\|,\; \Delta_{k-1}\}, & \widetilde\rho_k \ge \tilde\eta_2, \\[6pt]
                        \Delta_{k-1}, & \widetilde\rho_k \in [\tilde\eta_1,\tilde\eta_2), \\[6pt]
                        \gamma_2 \|s_{k-1}\|, & \widetilde\rho_k \in [0,\tilde\eta_1), \\[6pt]
                        \displaystyle\min\!\big\{\gamma_2\|s_{k-1}\|,\, \max[\gamma_1,\tilde\theta_k]\,\Delta_{k-1}\big\}, & \widetilde\rho_k < 0,
                        \end{cases}
                    \]
                    
                    \; ~\\
                    \textbf{Step 2: Model minimisation.} $\ldots$ \;

                    \caption{Retrospective trust-region algorithm (RTR).}
                    \label{alg:RTR}
                \end{algorithm}
                \begin{lemma}
                    Suppose \ref{as: Lipschitz continuous gradient}, \ref{as: model agreement with function} and \ref{as: Linear growth bound model Hessian} hold. Then
                    \begin{align*}
                        \big| f(x_k) - m_{k-1}(x_k) \big| \; &\le \; \tfrac12 \, (L + \norm{H_{k-1}}) \,\Delta_{k-1}^2, \quad \text{and} \quad \big| f(x_{k-1}) - m_k(x_{k-1}) \big| \; \\
                        &\le \; \tfrac12 \, (L + \norm{H_{k}}) \,\Delta_{k-1}^2.
                    \end{align*}
                \end{lemma}
                \begin{proof}
                    This follows directly from \ref{as: Lipschitz continuous gradient} and \ref{as: Linear growth bound model Hessian}.
                \end{proof}

                \begin{lemma}
                    Assuming \ref{as: Lipschitz continuous gradient} and \ref{as: Linear growth bound model Hessian}, for all \( k \geq 1 \), we have 
                    \begin{multline}
                        \abs{ \left( m_{k-1}(x_{k-1}) -  m_{k-1}(x_{k}) \right)
                              - \left(  m_{k}(x_{k-1}) -  m_{k}(x_{k})  \right) } \\
                        \leq \, \tfrac12 (2L +\norm{H_{k-1}} + \norm{H_{k}}) \Delta_{k-1}^2.
                    \label{eqn: difference in model reduction retrospective}
                    \end{multline}
                \end{lemma}
                \begin{proof}
                    First, note that 
                    $m_k(x_{k-1}) - m_k(x_k) =  -  s_{k-1}^\top  g_k
                        + \tfrac{1}{2} s_{k-1}^\top \, H_k \, s_{k-1}$  and $m_{k-1}(x_{k-1}) - m_{k-1}(x_k)  = - s_{k-1}^\top g_{k-1} - \tfrac12  s_{k-1}^\top H_{k-1} s_{k-1}$. 
                    Therefore, by Cauchy-Schwarz inequality, \ref{as: Lipschitz continuous gradient} and \ref{as: Linear growth bound model Hessian}, we have
                    \begin{align*}
                        \abs{ [m_{k-1}(x_{k-1}) - m_{k-1}(x_k)] - [m_k(x_{k-1}) - m_k(x_k)] }
                        &= \abs{- s_{k-1}^\top (g_{k-1} - g_k) - \tfrac12 s_{k-1}^\top (H_{k-1} + H_k) s_{k-1}} \\
                        &\leq \| s_{k-1} \| \| g_{k-1} - g_k \| + \tfrac12 \| s_{k-1} \|^2 \| H_{k-1} + H_k \| \\
                        &\leq \bigl( L + \tfrac12 (\norm{H_{k-1}} + \norm{H_{k}}) \bigr) \| s_{k-1} \|^2
                        \leq \tfrac12 (2L +\norm{H_{k-1}} + \norm{H_{k}}) \Delta_{k-1}^2.
                    \end{align*}
                \end{proof} 

                \begin{theorem}
                    Suppose \ref{as: model agreement with function}, \ref{as: Lipschitz continuous gradient} and \ref{as: Linear growth bound model Hessian} hold. Let \( k \geq 1 \), we have, for all \( \eta \in [0, 1) \),
                    \[ \text{if } \quad 
                    \Delta_{k-1} \;\le\; \frac{\mykappa{mdc}(1-\eta)}{L + \norm{H_{k-1}}} \|g_{k-1}\|
                    \quad \text{then} \quad
                    \rho_{k-1} \ge \eta,\]
                    and, for all \( \tilde\eta \in [0, 1) \),
                    \begin{equation}
                        \text{if } \quad 
                        \Delta_{k-1} \;\le\; \frac{\mykappa{mdc}(1-\tilde\eta)}{(3-2\tilde\eta)L + (2-\tilde\eta)\norm{H_k} + (1-\tilde\eta)\norm{H_{k-1}}} \|g_{k-1}\|
                        \quad \text{then} \quad
                        \widetilde{\rho}_k \ge \tilde\eta.
                        \label{eqn: low Delta retrospective implies high retrospective ratio}
                    \end{equation}
                \end{theorem}

                \begin{proof}
                The first part is a direct consequence of Theorem~\ref{thm: successful iteration for low Delta}.
                For the second part, we have 
                \begin{align*}
                    \widetilde{\rho}_k - 1 &= \frac{f(x_{k-1}) - f(x_k)}{m_k(x_{k-1}) - m_k(x_k)} - 1 \\
                    &= \frac{f(x_{k-1}) - f(x_k) - (m_k(x_{k-1}) - m_k(x_k))}{m_k(x_{k-1}) - m_k(x_k)} \\
                    &= \frac{f(x_{k-1}) - m_k(x_{k-1}) }{m_k(x_{k-1}) - m_k(x_k)}
                \end{align*}
                By the previous lemmas, we have
                \[
                    \abs{[m_k(x_{k-1}) - m_k(x_k)]}
                    \geq \abs{[m_{k-1}(x_{k-1}) - m_{k-1}(x_k)]}
                         - \tfrac12 (2L +\norm{H_{k-1}} + \norm{H_{k}}) \Delta_{k-1}^2.
                \]
                Then, by the Cauchy decrease, and  $\Delta_{k-1} \leq \frac{\mykappa{mdc}(1-\tilde\eta)}{(1-\tilde\eta)\norm{H_{k-1}}} \|g_{k-1}\| \leq \frac{\|g_{k-1}\|}{\|H_{k-1}\|} $, we have
                \begin{align*}
                    & \abs{1 - \widetilde{\rho}_k} \\
                        & \leq \frac{ \tfrac12(L + \norm{H_{k}}) \Delta_{k-1}^2}
                                    {\dfrac{\mykappa{mdc}}{2} \, \| g_{k-1} \| \, \Delta_{k-1}
                                     - \tfrac12 (2L + \norm{H_{k-1}} + \norm{H_{k}}) \Delta_{k-1}^2} \\
                        & \leq \frac{(L +  \norm{H_{k}}) \Delta_{k-1}}
                                    {\mykappa{mdc} \, \| g_{k-1} \|
                                     - (2L + \norm{H_{k-1}} + \norm{H_{k}}) \Delta_{k-1}}
                \end{align*}
                And, from~\eqref{eqn: low Delta retrospective implies high retrospective ratio} and the last inequality, we get $\abs{1 - \widetilde{\rho}_k} \leq 1 -\tilde\eta $ which implies  \( \widetilde{\rho}_k \geq \tilde\eta \).
                \end{proof}
                \begin{corollary}
                    By considering $M_k = L + \max_{0 \leq i \leq k} \norm{H_i}$, as defined in~\eqref{eqn: M_k definition}, we have for all \( k \geq 1 \),
                    \begin{equation}
                        \text{if } \quad \Delta_{k-1} \;\le\; \frac{\mykappa{mdc}(1-\tilde\eta)}{(3-2 \, \tilde\eta) M_k} \|g_{k-1}\| \quad 
                        \text{then}
                        \quad
                        \widetilde{\rho}_k \ge \tilde\eta.
                        \label{eqn: low Delta retrospective implies high retrospective ratio M_k}
                    \end{equation}
                \end{corollary}
                \begin{proof}
                    This is a direct consequence of the previous theorem and the definition of \( M_k \) in~\eqref{eqn: M_k definition}.
                \end{proof}
                The retrospective update satisfies Condition~\ref{cond: delta_k lower bound} as follows.
                \begin{theorem}
                    For all \( k \geq 0 \), the trust-region radius \( \Delta_k \) satisfies Condition~\ref{cond: delta_k lower bound}, with 
                    \begin{equation}
                        \mykappa{lbd} \;=\; \min\!\bigg\{1,\;
                                \frac{\Delta_0 M_0}{\|g_0\|},\;
                                \gamma_1(1-\eta)\mykappa{mdc},\;
                                \gamma_2(1-\xi),\;
                                \frac{\gamma_1 (1-\tilde\eta_2) \mykappa{mdc}}{3-2\tilde\eta_2}
                            \bigg\}.
                    \end{equation}
                    \label{thm: RTR lower bound}
                \end{theorem}

                \begin{proof}
                    We prove the result by induction on \( k \). \emph{Base case.} The inequality holds by the choice of \(\mykappa{lbd}\).
                    
                    \noindent\emph{Induction step.} Assume the bound
                    \(\Delta_k \ge \dfrac{\mykappa{lbd}}{M_k}\min_{0\le i\le k}\|g_i\|\). 
                    We will show it holds for \(k+1\).
                    \begin{enumerate}[(i)]
                        \item If $\rho_k < \eta_1$, by the update rule we set 
                            $
                                \Delta_{k+1} \ge \min \{\gamma_2\|s_{k}\|, \max[\gamma_1,\theta_{k}]\Delta_{k}\} \ge \min \{\gamma_2\|s_{k}\|, \gamma_1 \, \Delta_{k}\}.
                            $
                            Because, $\rho_k < \eta_1$, we have $\Delta_k \geq \frac{(1 - \eta_1) \, \mykappa{mdc}}{M_k} \| g_k \|$. 
                            By Remark~\ref{rem: lower bound s_k}, we have \(\|s_{k}\|\ge  \frac{(1-\xi)}{M_k} \norm{g_k}\).
                            Therefore, 
                            \[
                                \Delta_{k+1} \ge \frac{\min\{\gamma_2(1-\xi) \, , \, \gamma_1(1-\eta_1)\mykappa{mdc} \}}{M_{k+1}}\min_{0\le i\le k+1}\|g_i\|.
                            \]
                        \item Else, $\rho_k \geq \eta_1$. \\
                            First, consider the case $\widetilde\rho_k < \tilde\eta_2$.
                            By contraposition of~\eqref{eqn: low Delta retrospective implies high retrospective ratio M_k}, this implies
                            \(\Delta_{k} > \dfrac{\mykappa{mdc}(1-\tilde\eta_2)}{(3-2\tilde\eta_2)M_{k+1}}\|g_{k}\|\).
                            By a similar argument as in case (i), we have
                            \begin{align*}
                                \Delta_{k+1} &\ge \min \left\{\dfrac{\gamma_2(1-\xi)}{M_k} \norm{g_k}, \dfrac{\gamma_1(1-\tilde\eta_2)\mykappa{mdc}}{(3 - 2 \tilde\eta_2) M_{k+1}} \norm{g_k}\right\} \\
                                &\ge \frac{\min\{\gamma_2 \, (1-\xi) \, , \, \frac{\gamma_1 \, (1-\tilde\eta_2)\mykappa{mdc}}{3-2\tilde\eta_2}\}}{M_{k+1}}\min_{0\le i\le k+1}\|g_i\|.
                            \end{align*}
                            If instead \(\widetilde\rho_k\ge\tilde\eta_2\), the retrospective rule sets \(\Delta_{k+1} \geq  \Delta_k \geq  \dfrac{\mykappa{lbd}}{M_k}\min_{0\le i\le k}\|g_i\|\ge \dfrac{\mykappa{lbd}}{M_{k+1}}\min_{0\le i\le k+1}\|g_i\|\).
                    \end{enumerate}
                    Thus in every case \(\Delta_{k+1}\) satisfies the required lower bound, completing the induction and the proof.
                \end{proof}

                To avoid the dependence of the lower bound on the constant $\xi$ from the subproblem solver, we can consider the following update rule in Step 1.b. of Algorithm~\ref{alg:RTR}:
                \begin{enumerate}[ ]
                    \item If $\rho_{k-1} < \eta_1$, set $\Delta_k \in [\gamma_1 \, \Delta_{k-1}, \, \gamma_2 \Delta_{k-1} ) $.
                    \item Else, compute the retrospective ratio $\widetilde\rho_k$ in~\eqref{eq:retrospective ratio} and set
                    \[
                        \Delta_k =
                        \begin{cases}
                        ( \Delta_{k-1}, \, \gamma_3 \, \Delta_{k-1}], & \widetilde\rho_k \ge \tilde\eta_2, \\[6pt]
                        [\gamma_2 \, \Delta_{k-1}, \, \Delta_{k-1}], & \widetilde\rho_k \in [\tilde\eta_1,\tilde\eta_2), \\[6pt]
                        [\gamma_1 \, \Delta_{k-1}, \, \gamma_2 \Delta_{k-1} ), & \widetilde\rho_k < \tilde\eta_1.
                        \end{cases}
                    \]
                \end{enumerate}
                Considering this update rule, we can show that the trust-region radius \( \Delta_k \) satisfies Condition~\ref{cond: delta_k lower bound} with \[ \mykappa{lbd} = \min\!\bigg\{1,\;
                                \frac{\Delta_0 M_0}{\|g_0\|},\;
                                \gamma_1 (1-\eta)\mykappa{mdc},\;
                                \frac{\gamma_1 (1-\tilde\eta_2) \mykappa{mdc}}{3-2\tilde\eta_2}
                            \bigg\}, \] which does not depend on $\xi$.

                Building on this retrospective trust-region scheme, \citet{FanPan16} proposed an algorithm combining it with the idea of tying the trust-region radius to the gradient norm, as in the previous section.
                They consider  $\Delta_k = \mu_k \| g_k \|$, where $\mu_k$ follows the retrospective ratio $\widetilde{\rho}_k$, see Algorithm~\ref{alg:RTR-NRTR}.
                \begin{algorithm}[htbp]
                    \small
                    \SetAlgoLined
                    \DontPrintSemicolon

                    $ \ldots $ \; ~\\

                    \textbf{Step 1.b: Retrospective scaling update.}
                    If $k=0$, skip to \textbf{Step 2}.
                    \begin{enumerate}[ ]
                        \item If $\rho_{k-1} < \eta_1$, set $\mu_k \in [\gamma_1 \, \mu_{k-1} \, , \, \gamma_2 \mu_{k-1})$.
                        \item Else, compute the \emph{retrospective} ratio $\widetilde\rho_k$ in~\eqref{eq:retrospective ratio} and set
                    \[
                        \mu_k \in
                        \begin{cases}
                            [\gamma_1 \, \mu_{k-1} \, , \, \gamma_2 \mu_{k-1}), & \widetilde\rho_k < \tilde\eta_1, \\[6pt]
                            [\gamma_2  \, \mu_{k-1}, \, \mu_{k-1}), & \widetilde\rho_k \in [\tilde\eta_1,\tilde\eta_2), \\[6pt]
                            (\mu_{k-1}, \, \gamma_3 \, \mu_{k-1}], & \widetilde\rho_k \ge \tilde\eta_2 \text{ and } \norm{s_{k-1}} > \frac{1}{2} \Delta_{k-1}, \\[6pt]
                            \mu_{k-1}, & \text{ otherwise}.
                        \end{cases}
                    \]
                        \item  Set $ \Delta_k \;\gets\; \mu_k\,\|g_k\| $. 
                    \end{enumerate}
                    $ \ldots $ \; ~\\

                    \caption{Retrospective trust-region algorithm with radius tied to the gradient norm.}
                    \label{alg:RTR-NRTR}
                \end{algorithm}
                The convergence analysis of this algorithm can be adapted from the analysis of the gradient-scaled update rule presented in Section~\ref{subsec:Relative gradient}.
                \begin{proposition}
                    For all \( k \geq 0 \), we have the following lower bounds on the trust-region parameter \( \mu_k \) and the trust-region radius \( \Delta_k \):
                    \begin{equation}
                        \mu_k \geq \frac{\mykappa{lbd}}{M_k} \quad \text{and} \quad  \Delta_k \geq \dfrac{\mykappa{lbd}}{M_k} \, \| g_k \|,
                    \end{equation}
                    where \( \mykappa{lbd} \;=\; \min\!\bigg\{1,\;
                                \mu_0 M_0,\;
                                \gamma_1 (1-\eta_1) \mykappa{mdc},\;
                                \dfrac{\gamma_1 (1-\tilde\eta_2) \mykappa{mdc}}{3-2\tilde\eta_2}
                            \bigg\}\).
                    \label{prop: lower bound mu_k and Delta_k retrospective}
                \end{proposition}
                \begin{proof}
                    The proof follows the same structure as the proof of the previous proposition and Proposition~\ref{prop: mu_k lower bound relative gradient}.
                \end{proof}

    \section{Verification of conditions for individual mechanisms} \label{sec:radius update mechanisms}

        We review trust-region update rules and verify Condition~\ref{cond: delta_k lower bound}.  
        Conditions~\ref{cond: decrease of the trust-region radius} and~\ref{cond: increase of the trust-region radius} hold by construction or are imposed directly.

        \subsection{Some useful assumptions on the subproblem solution} \label{subsec: assumptions on the subproblem solution}

            The conditions below on the subproblem solution are used in the proofs involving the step length $\|s_k\|$, at the end of Subsection~\ref{subsec:Incremental radius update mechanisms} and in Subsection~\ref{subsec:Step-size dependent updates}. 
            
            When conditions on the step are necessary, we will assume that the trust-region subproblem is solved using an iterative method, that satisfies conditions inspired from the truncated CG method~\citep{Stei83}.
            We consider an early termination strategy for solving iteratively the (quasi-)Newton equation 
            $ H_k s = - g_k $, as outlined in~\citet[][Chap. 7]{NoceWrig06}.
            This approach is based on monitoring the \emph{residual} $r_k$, defined as $H_k s_k =  - g_k + r_k$, 
            where \( s_k \) represents the inexact (quasi-)Newton step.
            \begin{cond}{C.InSp}[termination when the region is inactive]
                When the trust-region is inactive, i.e.\ \( \| s_k \| < \Delta_k \), then the procedure is terminated with the residual satisfying
                \begin{equation}
                    \Vert r_k \Vert \leq \xi_k \Vert g_k \Vert,
                    \label{eqn:stop criteria forcing sequence}
                \end{equation}
                where \( \{\xi_k\} \) is referred to as the \emph{forcing sequence}, and \( 0 < \xi_k < \xi \) for all \( k \), with $\xi \in (0,1)$.
                \label{cond: inactive trust-region condition subproblem}
            \end{cond}
            The forcing sequence \( \{\xi_k\} \) is a sequence of positive numbers that need to be strictly less than 1 to discard zero as a possible solution. 
            \begin{assum}{a.TrSub}[subproblem solver]
                The trust-region subproblem~\eqref{eq: TR subproblem} is solved by 
                an iterative method that produces a step \( s_k \) satisfying Condition~\ref{cond: inactive trust-region condition subproblem}.
                \label{as: solver of the TR subproblem}
            \end{assum}     
            Proposition~\ref{prop: lower bound on s_k} relates the step length $\|s_k\|$ to the trust-region radius $\Delta_k$ and to the gradient norm when the subproblem is solved approximately, under Assumption~\ref{as: solver of the TR subproblem}.
            \begin{proposition}
                When the trust-region subproblem is solved using an iterative method satisfying Assumption~\ref{as: solver of the TR subproblem}, then the step $s_k$ satisfy the following inequalities
                \begin{equation}
                    \min \left\{ \Delta_k, \dfrac{1 - \xi}{\| H_k \|} \| g_k \| \right\} \leq \| s_k \| \leq \Delta_k,
                    \label{eqn:stop criteria forcing sequence consequence for s_k}
                \end{equation}
                where we have used the same convention as in Theorem~\ref{th: Cauchy decrease}, i.e., for \( \|H_k\| = 0 \), we set \( \frac{1 - \xi}{\| H_k \|} = +\infty \).
                \label{prop: lower bound on s_k}
            \end{proposition}
            \begin{proof}
                When the Hessian is zero, the model is linear and the step is obtained at a point on the boundary of the trust region, i.e.\ \( \|s_k\| = \Delta_k \). 
                When \( H_k \neq 0 \), we have \( \|s_k\| \leq \Delta_k \) by definition of the trust-region subproblem.
                If the trust-region is inactive, i.e.\ \( \|s_k\| < \Delta_k \), then from Condition~\ref{cond: inactive trust-region condition subproblem}, we have, using the triangular inequality, \eqref{eqn:stop criteria forcing sequence} and the matrix norm induced by $\|\cdot\|$,
                \[
                    \| g_k \| \leq \| H_k s_k + g_k \| + \| H_k s_k \| \leq \| r_k \| + \| H_k \| \| s_k \|,
                \]
                and therefore
                \[
                \| s_k \| \geq \dfrac{1 - \xi}{\| H_k \|} \| g_k \|.
                \]
            \end{proof}
            
            \begin{remark}
            By imposing in the subproblem that the step \( s_k \) satisfies the Cauchy decrease condition \ref{as: Cauchy decrease}, and
            either $\| s_k \| = \Delta_k$ or $\| H_k s_k + g_k \| \leq \xi_k \| g_k \|,$
            where \( \xi_k \leq \xi < 1 \) is a forcing sequence, we have
            \begin{equation}
                \| s_k \| \geq \dfrac{\mykappa{lbd}^\prime}{M_k} \min_{0 \leq i \leq k} \| g_i \|,
                \label{eqn: lower bound s_k}
            \end{equation}
            where \( \mykappa{lbd}^\prime = \min \left\{ \mykappa{lbd} \, ,\, 1 - \xi \right\} > 0 \).
            \label{rem: lower bound s_k}
        \end{remark}

        \subsection{Fixed-factor update} \label{subsec:Incremental radius update mechanisms}

            The most common and simplest trust-region mechanism is the \emph{incremental} update, in which the radius is scaled by fixed multiplicative factors from one iteration to the next.
            The rule shrinks the radius after poor model agreement, maintains or cautiously expands it under moderate agreement, and enlarges it after good agreement. 
            The rule is simple to implement, hence widely adopted in practice, and is studied in~\citet{ConnGoulToin00, NoceWrig06, More83, CartGoulToin12, CartGoulToin22}. The update is defined as follows.
            \begin{update}{\Rdelta}
                    \begin{equation}
                        \Delta_{k+1} \in 
                            \begin{cases}
                            [\gamma_1 \Delta_k, \gamma_2 \Delta_k), & \text{if } \rho_k < \eta_1, \\
                            [\gamma_2 \Delta_k, \Delta_k), & \text{if } \eta_1 \leq \rho_k < \eta_2, \\
                            [\Delta_k, +\infty), & \text{if } \rho_k \geq \eta_2.
                            \end{cases}
                    \end{equation}
                    where \( 0 < \gamma_1 \leq \gamma_2 < 1 \) and \( \eta \leq \eta_1 \leq \eta_2 < 1 \) with $\eta_1 \neq 0$. 
                    \label{eqn:ratio-based update}
            \end{update}
            \begin{remark}
                When necessary for the analysis, we can impose the following  condition on the update rule when $\rho_k \geq \eta_2$,
                $ 
                    \Delta_{k+1} \in [\Delta_k, \overline{\gamma_3} \Delta_k],
                $
                where \( \overline{\gamma_3} > 1 \) is a constant. 
                This condition is satisfied by most of the ratio-based update rules in the literature, and it is useful to prove $\liminf_{k \to \infty} \| g_k \| = 0 $ under the weaker assumptions in Theorem~\ref{thm: Liminf of the gradient norm under weaker assumption}, and also 
                to bound the number of unsuccessful iterations in the worst-case complexity analysis, as we have seen in Theorem~\ref{thm: Upper bound successful iter BTR}.
            \end{remark}
            \begin{example}[Ratio-based update rules from the literature]
                We summarise three representative ratio-based rules side by side.

                \[
                \begin{array}{llll}
                & \textbf{Simple rule} &
                \textbf{\citet{NoceWrig06}} &
                \textbf{\citet{ConnGoulToin00, More83}} \\[0.7em]
                \Delta_{k+1} =
                &
                \begin{cases}
                \gamma^{-1} \Delta_k, & \rho_k < \eta_1, \\
                \gamma \Delta_k, & \rho_k \geq \eta_1,
                \end{cases}
                &
                \begin{cases}
                \gamma_1 \Delta_k, & \rho_k < \eta_1, \\
                \gamma_3 \Delta_k, & \rho_k > \eta_2 \;\text{ and }\; \|s_k\|=\Delta_k, \\
                \Delta_k, & \text{otherwise},
                \end{cases}
                &
                \begin{cases}
                \gamma_1 \Delta_k, & \rho_k < \eta_1, \\
                \gamma_2 \Delta_k, & \eta_1 \leq \rho_k < \eta_2, \\
                \gamma_3 \Delta_k, & \rho_k \geq \eta_2,
                \end{cases}
                \end{array}
                \]
                where $0 < \gamma_1 \leq \gamma_2 < 1 < \gamma_3,\; 0 < \eta_1 \leq \eta_2 < 1$, and $\gamma > 1$.
        \end{example}
        The incremental update satisfies the lower bound condition on the trust-region radius.
        \begin{proposition}
            For all \( k \geq 0 \), the trust-region radius \( \Delta_k \) satisfies Condition~\ref{cond: delta_k lower bound}, with 
            \begin{equation}
                 \mykappa{lbd} = \min \left\{ 1 \, , \, \frac{\Delta_0 \, M_0 }{\| g_0 \|} \, , \, \gamma_1 \, \mykappa{mdc} \, (1 - \eta_2)  \right\}.
            \end{equation}
            \label{prop: delta_k lower bound incremental}
        \end{proposition}
        \begin{proof}
            The argument follows~\citet[Chap.~2]{CartGoulToin22} where the proof proceeds by induction on $k$ using a  contradiction argument. For completeness, we provide a direct proof here using contraposition. \\
            \emph{Base case.} For \( k = 0 \), we have \( \Delta_0 \geq \frac{\mykappa{lbd}}{M_0} \, \| g_0 \| \) by definition of \( \mykappa{lbd} \). \\
            \emph{Induction step.} Suppose the result holds for some \( k \geq 0 \). We will show it holds for \( k + 1 \).

            \smallskip 
            \noindent (i) If $\rho_k \geq \eta_2$, the update rule~\eqref{eqn:ratio-based update} yields \( \Delta_{k+1} \geq \Delta_k \). Since \( M_{k+1} \geq M_k \), we obtain
            \[
                \Delta_{k+1} \geq \Delta_k \geq \frac{\mykappa{lbd}}{M_k} \, \, \min_{0 \leq j \leq k } \| g_j \| \geq \frac{\mykappa{lbd}}{M_{k+1}} \, \, \min_{0 \leq j \leq k +1 } \| g_j \|.
            \]
            (ii) If $\rho_k < \eta_2$, the contrapositive of Corollary~\ref{cor: delta_k lower bound} implies
                    \[
                    \Delta_k \;>\; \frac{\mykappa{mdc}(1-\eta_2)}{L+\|H_k\|}\,\|g_k\|
                    \;\ge\; \frac{\mykappa{mdc}(1-\eta_2)}{M_k}\,\min_{0\le j\le k}\|g_j\|.
                    \]
                    and, from update rule~\eqref{eqn:ratio-based update}, 
                    \[
                        \Delta_{k+1} \geq \gamma_1 \Delta_k \geq \frac{\gamma_1  \, \mykappa{mdc} \, (1 - \eta_2) }{M_k} \, \min_{0 \leq j \leq k } \| g_j \| \geq \frac{\mykappa{lbd}}{M_{k+1}} \, \min_{0 \leq j \leq k + 1 } \| g_j \|,
                    \]
            \smallskip
            In both cases the claim holds at $k+1$, completing the induction. The proof is thus complete.
        \end{proof}
        \citet[Chap.~10]{ConnGoulToin00} propose an aggressive update that contracts the radius sharply when the model fails. This prevents overly large radii from producing steps that minimise the model rather than the actual objective function. 
        The motivation is that when the computed step \( \|s_k\| \) is significantly smaller than the current trust-region radius \( \Delta_k \), the ratio $\rho_k$ evaluates the model only within a ball of radius $\|s_k\|$, not the prescribed $\Delta_k$. 
        If \( \rho_k \) is small or negative (i.e., \( \rho_k < \eta_1 \)), several reductions of $\Delta_k$ may be needed before excluding $x_k+s_k$, each costing an extra function evaluation. To mitigate this inefficiency, they propose the following more flexible update rule, defined as follows.
        \begin{update}{\Rdeltastep}           
            \[
            \Delta_{k+1} \in 
            \begin{cases}
                [\gamma_1 \|s_k\|, \gamma_2 \|s_k\|) & \text{if } \rho_k < \eta_1, \\
                [\gamma_2 \Delta_k, \Delta_k) & \text{if } \rho_k \in [\eta_1, \eta_2), \\
                [\Delta_k, +\infty) & \text{if } \rho_k \geq \eta_2,
            \end{cases}
            \]
            where \( 0 < \gamma_1 \leq \gamma_2 \leq 1 \). 
            \label{eqn:classical rule with aggressive decrease}
        \end{update}    
        \begin{remark}
            The update rule~\ref{eqn:classical rule with aggressive decrease} can be seen as a special case of the update rule~\ref{eqn:ratio-based update} with carefully chosen parameters.  See~\citet[sec. 10.5.2]{ConnGoulToin00} for more details. 
        \end{remark}
        \begin{example}[Variant analysed in~\citet{GouldOrbanSartenaerToint05}]
            A commonly used variant, valued for its simplicity and effectiveness in adjusting the trust-region radius according to the quality of the trial step, is:
            \[
            \Delta_{k+1} \in 
                \begin{cases}
                    \gamma_1 \|s_k\|, & \text{if } \rho_k < \eta_1, \\[0.4em]
                    \Delta_k, & \text{if } \rho_k \in [\eta_1, \eta_2), \\[0.4em]
                    \max\{ \gamma_3 \|s_k\|, \Delta_k \}, & \text{if } \rho_k \geq \eta_2,
                \end{cases}
            \]
            where $0 < \gamma_1 < 1 \leq \gamma_3$.
        \end{example}
        Typical parameter choices in the literature are \( \gamma_1 = 0.5 \), \( \gamma_3 = 2 \), \( \eta_1 = 0.25 \), and \( \eta_2 = 0.75 \). However, \citet{GouldOrbanSartenaerToint05} provide a sensitivity analysis and recommend \( \gamma_1 = 0.25 \), \( \gamma_3 = 3.5 \), \( \eta_1 = 0.0001 \), and \( \eta_2 = 0.99 \) for improved robustness.
        
        \citet{ConnGoulToin00} propose a heuristic to handle cases where the model is a poor approximation of the objective, especially when $\rho_k < 0$. 
        The strategy rescales the step so that the ratio is more likely to be classified as very successful in the next iteration. 
        Let $\phi_f(t) := f(x_k + t s_k)$ and $\phi_m(t) := m_k(x_k + t s_k)$ for $t\in[0,1]$. Consider a quadratic interpolant using the value and directional derivative at $t=0$, and the value at $t=1$,
        \[
            \phi_f^{\mathrm{quad}}(t) \;=\; f_k \;+\; (g_k^\top s_k)\, t \;+\; a_f\, t^2,
            \qquad
            a_f \;=\; f(x_k+s_k) - f_k - g_k^\top s_k.
        \]
        Considering a similar quadratic interpolant for the model, we find $t>0$ such that  $\frac{\phi_f^{\mathrm{quad}}(t)-\phi_f^{\mathrm{quad}}(0)}{\phi_m^{\mathrm{quad}}(t)-\phi_m^{\mathrm{quad}}(0)} =  \eta_2 $.
        Then,  
        \begin{equation}
            \theta_k 
            = \frac{(1 - \eta_2)\, \nabla f(x_k)^{\top} s_k }{(1 - \eta_2)\bigl[f(x_k) + \nabla f(x_k)^{\top} s_k \bigr] 
            + \eta_2 m_k(x_k + s_k) - f(x_k + s_k)}
            \label{eq:theta_k}
        \end{equation}
        The trust-region radius is then updated as follows, $\Delta_{k+1} = \min \Bigl[ \, \gamma_1 \, \|s_{k}\|, \; \max \bigl[ \gamma_0, \, \theta_{k} \bigr] \, \Delta_{k} \Bigr]$, where $0 <\gamma_0 < \gamma_1$. 
        \begin{corollary}
            When considering these aggressive decrease strategies under additional assumption~\ref{as: solver of the TR subproblem}, the lower bound on the trust-region radius is satisfied with
            \begin{equation}
                \mykappa{lbd} = \min \left\{ 1 \, , \, \frac{\Delta_0 \, M_0 }{\| g_0 \|} \, , \, \gamma_0 \, \mykappa{mdc} \, (1 - \eta_2) \, , \, 1 - \xi \right\}.
            \end{equation}
            \label{cor: aggressive decrease lower bound}
        \end{corollary}
        \begin{proof}
            This can be obtained by using Proposition~\ref{prop: lower bound on s_k} and adapting the proof of Proposition~\ref{prop: delta_k lower bound incremental}.
        \end{proof}

        \subsection{Step-driven update} \label{subsec:Step-size dependent updates}

            The trust-region radius can also be updated based on the step length $\|s_k\|$. 
            This strategy decreases the radius aggressively when the model proves inaccurate and increases it more conservatively when the model is reliable. 
            By tying the radius to the actual step taken, this approach naturally drives $\Delta_k$ toward zero in the limit when the iterates converge to a stationary point.
            \citet[sec. 10.5.2]{ConnGoulToin00} and~\citet{Powe70,Powe75,Powe84} define an update rule of the form,
                \begin{update}{\Rstep}
                    \begin{equation}
                        \Delta_{k+1} \in 
                        \begin{cases}
                            [\gamma_1 \|s_k\|, \gamma_2 \|s_k\|) & \text{if } \rho_k < \eta_1, \\
                            [\gamma_2 \|s_k\|, \|s_k\|) & \text{if } \rho_k \in [\eta_1, \eta_2), \\
                            [\|s_k\|, +\infty) & \text{if } \rho_k \geq \eta_2, 
                        \end{cases}
                    \end{equation}
                    where \( 0  < \gamma_1 \leq \gamma_2 < 1 \) and \(\eta \leq \eta_1 \leq \eta_2 < 1 \) with $\eta_1 \neq 0$.
                    \label{update: step-size dependent update}
                \end{update}

                \begin{example}
                    \citet{Powe75, Powe84} define with \( 0 < \gamma_1 \leq \gamma_2 < 1 < \gamma_3 \) and \( 0 \leq \eta < \eta_1 < 1 \), the following update rule:
                    \[
                        \Delta_{k+1} \in 
                                \begin{cases}
                                [\gamma_1 \| s_k \|, \gamma_2 \| s_k \|], & \text{if } \rho_k < \eta_1, \\
                                [\| s_k \|, \gamma_3 \| s_k \|], & \text{if } \rho_k \geq \eta_1   \\
                                \end{cases}
                    \]
                \end{example}
            The ratio \( \rho_k \) varies across iterations with the choice of model, the trust-region radius, and the nature of the objective function. 
            \citet{Hei03} proposed an update rule where the trust-region radius is proportional to the step length \( \|s_k\| \), scaled by an adaptive factor \( R_{\eta_1}(\rho_k) \) that depends on the ratio \( \rho_k \), for an example see Figure~\ref{fig: R_eta_1} in Appendix~\ref{appendix: example of R-function}.
            The rule adjusts the radius according to the quality of the step, and most sharply when the model approximates the function poorly.
            The update is defined as follows.
            \begin{update}{\RAdaptstep}
                \[
                     \Delta_{k+1} = R_{\eta_1}(\rho_k) \|s_k\|,
                \]
                where \( R_{\eta_1} : \mathbb{R} \to \mathbb{R}_+ \) is a \textbf{non-decreasing} function satisfying, with \(0 \leq  \eta  < \eta_1 < 1 \) and \( 0 < \gamma_1 \leq \gamma_2 < 1 \),
                \begin{align*}
                    & \lim_{t \to -\infty} R_{\eta_1}(t) = \gamma_1 ; \\
                    & \forall \, t < \eta,\,  R_{\eta_1}(t) \leq \gamma_2 ; \\
                    & R_{\eta_1}(\eta_1) = 1 + \gamma_2; \\
                    & \lim_{t \to +\infty} R_{\eta_1}(t) = \gamma_3,  \text{ where } \gamma_3 > 1 + \gamma_2.
                \end{align*}
                
                \label{update: adaptive step-size dependent update}
            \end{update}
                    
            \begin{remark}
                This update rule fits in the general framework of update rule~\ref{update: step-size dependent update} with \( \gamma_1 = R_{\eta_1}(-\infty) \), \( \gamma_2 = R_{\eta_1}(\eta_1) \), and \( \gamma_3 = R_{\eta_1}(+\infty) \). Then, for all iterations \( k \geq 0 \), we have $\Delta_{k+1} \in 
                    \begin{cases}
                        [\gamma_1 \|s_k\|, \gamma_2 \|s_k\|) & \text{if } \rho_k < \eta_1 \\
                        [(1 + \gamma_2) \|s_k\|, \gamma_3 \| s_k \|) & \text{if } \rho_k \geq \eta_1
                    \end{cases}$.
            \end{remark}

            \begin{proposition}
                Assume the subproblem satisfies Assumption~\ref{as: solver of the TR subproblem}. Then, for all \( k \geq 0 \), the trust-region radius \( \Delta_k \) satisfies Condition~\ref{cond: delta_k lower bound}, with  
                \[
                    \mykappa{lbd} = \min \left\{ 1, \, \frac{\Delta_0 M_0}{\| g_0 \|}, \, \gamma_1 \, (1 - \eta_2) \mykappa{mdc} \, , \, 1 - \xi \right\}
                \]
                and $\xi \in [0, 1)$ is a uniform upper bound on the forcing sequence \( \{ \xi_k \} \) defined in~\eqref{eqn:stop criteria forcing sequence}.
                \label{prop: lower bound Delta_k step-size dependent}
            \end{proposition}
            \begin{proof}
                The proof can be obtained by using Proposition~\ref{prop: lower bound on s_k} and adapting the proof by induction of Proposition~\ref{prop: delta_k lower bound incremental}. \\
                \emph{Base case.} For \( k = 0 \), we have \( \Delta_0 \geq \frac{\mykappa{lbd}}{M_0} \, \| g_0 \| \) by definition of \( \mykappa{lbd} \). \\
                \emph{Induction step.} Suppose the result holds for some \( k \geq 0 \). We will show it holds for \( k + 1 \).
                \begin{enumerate}
                    \item If $\rho_k < \eta_2$, from Corollary~\ref{cor: delta_k lower bound} we have
                            $
                                \norm{s_k} > \frac{\mykappa{mdc} (1 - \eta_2)}{L + \norm{H_k}} \, \norm{g_k}.
                            $
                            Then, we have
                            \[
                                \Delta_{k+1} \geq \gamma_1 \|s_k\| > \frac{\gamma_1 \, \mykappa{mdc} \, (1 - \eta_2)}{L + \|H_k\|} \, \|g_k\| \geq \frac{\mykappa{lbd}}{M_{k+1}} \, \min_{0 \leq j \leq k + 1 } \| g_j \|.  
                            \]
                    \item Else if $\rho_k \geq \eta_2$, we have $\Delta_{k+1} \geq \|s_k\|$. Then, from Proposition~\ref{prop: lower bound on s_k}, we have
                     \[
                        \Delta_{k+1} \geq \norm{s_k}
                            \geq \min\left\{ \Delta_k , \frac{1 - \xi}{\norm{H_k}} \norm{g_k} \right\}
                            \geq \min\{ \mykappa{lbd} , 1 - \xi \} \, \frac{\min_{0 \leq j \leq k} \norm{g_j}}{M_k}
                            \geq \frac{\mykappa{lbd}}{M_{k+1}} \min_{0 \leq j \leq k+1} \norm{g_j} .
                    \]
                \end{enumerate}
                In both cases the claim holds at $k+1$, completing the induction. The proof is thus complete.
            \end{proof}
            
            \begin{remark}
        It would be possible to combine the aggressive decrease strategy with an adaptive coefficient based on the ratio \( \rho_k \).
        For example, one could define the update as
        \[
            \Delta_{k+1} = 
            \begin{cases}
                R_{\eta_1}(\rho_k) \|s_k\| & \text{if } \rho_k < \eta_1, \\
                R_{\eta_1}(\rho_k) \Delta_k & \text{if } \rho_k \geq \eta_1,
            \end{cases}
        \]
        where \( R_{\eta_1}(\rho_k) \) is defined as in Update~\ref{update: adaptive step-size dependent update}.
       \end{remark}

    \subsection{Criticality-anchored update} \label{subsec:Related to criticality measures}

        \subsubsection{Inspiration and motivation}

        In \emph{Derivative-Free Optimisation} (DFO)~\citep{ConnScheVice09_introduction}, where the gradient is not available or is unreliable, the model is typically constructed using function evaluations only. In this context, the radius update rule must keep the model a valid approximation of the objective function. A large \( \Delta_k \) may lead to steps that are too aggressive and not reflective of the local curvature.  A \emph{criticality step} is introduced to guarantee that the sequence of iterates \( \{x_k\} \) approaches first-order critical points.  
        Near a stationary point, the gradient norm \( \|g_k\| \) becomes small, therefore the trust-region radius \( \Delta_k \) must also shrink appropriately. This can be done by introducing a criticality measure that quantifies how close the current iterate is to a stationary point, and the trust-region radius is forced to satisfy:
        \begin{equation}
            \Delta_k \leq \zeta \, \chi_k^\alpha, \quad \text{for some } \zeta > 0, \, \alpha \in (0, 1],
            \label{eqn:criticality step}
        \end{equation}
        where \( \chi_k \) is a measure of criticality (e.g., \( \chi_k = \|g_k\| \)).
        Then, the criticality step ensures that the trust-region radius shrinks appropriately as the iterates approach critical points, allowing theoretical proofs of convergence to stationary points, even in the absence of gradient or Hessian information.
        This DFO-inspired thread has produced unified convergence frameworks across trust-region and line-search methods, in both deterministic and stochastic settings~\citep{ChenMeniSche18, CurtSche20, BlanCartMeniSche19}.
        Their work fits in the general framework presented in Algorithm~\ref{alg:adaptive-framework} where the role of $\alpha_k$ is to control the length of the step $s_k$.
        For example, in the context of line-search methods, $\alpha_k$ can be seen as the step length along the search direction, and the model $m_k$ can be defined as a linear approximation of $f$ around $x_k$. And, in the context of trust-region methods, $\alpha_k$ is seen as the trust-region radius, and the model $m_k$ can be a quadratic approximation of $f$ around $x_k$. 
        \begin{algorithm}[htbp]
            \small
            \SetAlgoLined
            \DontPrintSemicolon

            \textbf{Step 0: Initialisation.} 
            Choose constants $\eta \in (0,1)$, $\gamma \in (1,\infty)$, and $\alpha_{\max} \in (0,\infty)$. \\
            Choose an initial iterate $x_0 \in \mathbb{R}^p$ and step-size parameter $\alpha_0 \in (0,\alpha_{\max}]$. \; ~\\

            \textbf{Step 1: Determine model and compute step.}
            Choose a local model $m_k$ of $f$ around $x_k$. 
            Compute a step $s_k(\alpha_k)$ such that the model reduction
            $
            m_k(x_k) - m_k(x_k+s_k(\alpha_k)) \ge 0
            $
            is sufficiently large. \; ~\\

            \textbf{Step 2: Check for sufficient reduction in $f$.}
            Check whether the reduction
            $
            f(x_k) - f(x_k+s_k(\alpha_k))
            $
            is sufficiently large relative to the model reduction
            $
            m_k(x_k) - m_k(x_k+s_k(\alpha_k))
            $
            using a condition parameterised by $\eta$. \; ~\\

            \textbf{Step 3: Successful iteration.}
            If sufficient reduction has been attained (along with other potential requirements), set
            $
            x_{k+1} \gets x_k + s_k(\alpha_k)$, and 
            $
            \alpha_{k+1} \gets \min\{\gamma \alpha_k,\, \alpha_{\max}\}.
            $
            \; ~\\

            \textbf{Step 4: Unsuccessful iteration.}
            Otherwise, set
            $
            x_{k+1} \gets x_k$, and 
            $
            \alpha_{k+1} \gets \gamma^{-1}\alpha_k .
            $
            \; ~\\

            \textbf{Step 5: Next iteration.}
            Set $k \gets k+1$. \;

            \caption{Adaptive deterministic framework~\citep{CurtSche20}.}
            \label{alg:adaptive-framework}
        \end{algorithm}
        To be comparable to a line search, the update rule must be more conservative. Step acceptance rests on the ratio $\rho_k$ and on the radius being small relative to the gradient norm.
        The update rule can be summarised as follows with \( \gamma > 1 \) and \( \zeta > 0 \),
            \[
            (x_{k+1}, \Delta_{k+1}) =
            \begin{cases}
                \bigl(x_k + s_k, \; \gamma \Delta_k \bigr), 
                    & \text{if } \rho_k \geq \eta_1 \text{ and } \Delta_k \leq \zeta \|g_k\|, \\[6pt]
                \bigl(x_k, \; \gamma^{-1} \Delta_k \bigr), 
                    & \text{otherwise}.
            \end{cases}
            \]
        \citet{GratRoyeViceZhan17} and~\citet{BandScheVice14} propose a similar update rule, but the acceptance of the step and the update of the radius are decoupled, which is closer to the framework we propose in Algorithm~\ref{alg:generic BTR}.
        If $\rho_k \ge \eta_1$, set $x_{k+1} = x_k + s_k$ and
        \[
        \Delta_{k+1} =
        \begin{cases}
        \min\{\gamma_3 \Delta_k,\; \Delta_{\max}\}, & \text{if } \Delta_k \leq \zeta \|g_k\|,\\[6pt]
        \gamma_1 \Delta_k, & \text{otherwise}.
        \end{cases}
        \]
        Otherwise, set $x_{k+1} = x_k$ and $\Delta_{k+1} = \gamma_1 \Delta_k$. 

        The variants below fit our framework and satisfy the lower bound condition on the trust-region radius.

        \subsubsection{DFO-like update} \label{subsec:DFO-like update mechanism}

            We present an update rule inspired from~\citet{BandScheVice14,BlanCartMeniSche19, CurtSche20, GratRoyeViceZhan17}. This reflects line-search intuition, where the step length is typically proportional to the gradient norm, suggesting that the trust region should scale accordingly. Building on this idea, we adopt the DFO-like update mechanism defined as follows.
            \begin{update}{\Rdfo}
                Choose constants \( \eta \leq \eta_1 < 1 \), \( 0 < \gamma_1 \leq \gamma_2 < 1 <  \gamma_3 \) and \( \zeta > 0 \).
                \[
                \Delta_{k+1} =
                \begin{cases}
                    [\gamma_1 \Delta_k \, ; \, \gamma_2 \Delta_k) & \text{if } \rho_k < \eta_1 , \\
                    [\gamma_2 \Delta_k \, ; \, \Delta_k) & \text{if } \rho_k \geq \eta_1 \text{ \bf but } \Delta_k > \zeta \|g_k\|, \\
                    [\Delta_k \, ; \, + \infty) & \text{if } \rho_k \geq \eta_1 \text{ \bf and } \Delta_k \leq \zeta \|g_k\|.
                \end{cases}
                \]
                \label{update:DFO-like update}
            \end{update}
            \begin{remark}
                Again, it is possible to modify the update with $\Delta_{k+1} \in [\Delta_k, \gamma_3 \Delta_k]$ when $\rho_k \geq \eta_1$ and $\Delta_k \leq \zeta \|g_k\|$ to have the conditions required for the convergence analysis when the Hessian satisfies the weak assumption.
            \end{remark}
            \begin{proposition}
                For all \( k \geq 0 \), the trust-region radius \( \Delta_k \) satisfies Condition~\ref{cond: delta_k lower bound}, with 
                \[
                    \mykappa{lbd} = \min \left\{ 1, \, \gamma_2 \, \zeta \, M_0 , \, \frac{\Delta_0 M_0}{\| g_0 \|}, \, \gamma_1 \, (1 - \eta_1) \mykappa{mdc} \right\}.
                \]
                \label{prop: delta_k lower bound DFO}
            \end{proposition}
            \begin{proof}
                The proof is done by induction on \( k \).

                \noindent \emph{Base case.} For \( k = 0 \), we have \( \Delta_0 \geq \frac{\mykappa{lbd}}{M_0} \, \| g_0 \| \) by definition of \( \mykappa{lbd} \).       

                \noindent \emph{Induction step.} Suppose the result holds for some \( k \geq 0 \). We will show it holds for \( k + 1 \).
                We consider the following cases. 
                \begin{enumerate}[(i)]
                    \item Let \( \rho_k \geq \eta_1 \) and \(\Delta_k  \leq \zeta  \|g_k\| \). Then, by the update rule~\ref{update:DFO-like update}, we have
                    $
                        \Delta_{k+1} \geq  \Delta_k \geq \frac{\mykappa{lbd}}{M_{k+1}} \, \min_{0 \leq j \leq k + 1 } \| g_j \|.
                    $
                    \item Now consider \( \rho_k < \eta_1 \) or \(\Delta_k  > \zeta  \|g_k\| \). \\
                    If \( \rho_k < \eta_1 \). By Corollary~\ref{cor: delta_k lower bound}, we have \( \Delta_k > \frac{\mykappa{mdc} \, (1 - \eta_1) }{L + \| H_k \| } \, \| g_k \| \). And, by the update rule~\ref{update:DFO-like update}, we have
                        \[
                            \Delta_{k+1} \geq \gamma_1 \Delta_k \geq \frac{\gamma_1 \, (1 - \eta_1) \mykappa{mdc}}{M_k} \, \min_{0 \leq j \leq k } \| g_j \|  \geq  \frac{\mykappa{lbd}}{M_{k+1}} \,  \min_{0 \leq j \leq k + 1 } \| g_j \|.
                        \]
                    Now consider \(\Delta_k  > \zeta  \|g_k\| \). By the update rule~\ref{update:DFO-like update}, we have
                        \[
                            \Delta_{k+1} \geq \gamma_2 \Delta_k \geq \gamma_2 \zeta \| g_k \| \geq \frac{\gamma_2 \, \zeta \, M_{0}}{M_{k+1}} \, \min_{0 \leq j \leq k + 1 } \| g_j \| \geq  \frac{\mykappa{lbd}}{M_{k+1}} \,  \min_{0 \leq j \leq k + 1 } \| g_j \|.
                        \]       
                    \end{enumerate}
                    So, we have shown in all cases \( \Delta_{k+1} \geq \frac{\mykappa{lbd}}{M_{k+1}} \, \min_{0 \leq j \leq k + 1 } \| g_j \| \).
                    The proof is complete.
            \end{proof}
            Update rule~\ref{update:DFO-like update} serves the analysis and the comparison. Two further rules from the literature~\citep{BandScheVice14} are similar in spirit but differ in the details of the mechanism.  In practice, most algorithms employ two thresholds: one that triggers an increase of the radius and another that triggers a decrease.
            When the agreement between model and function lies between these thresholds, the radius remains unchanged.
            \begin{example}
                Choose constants \( \eta \leq \eta_1 < 1 \), \( \gamma > 1 \) and \( 0 < \zeta_1 < \zeta_2 \). 
                If \( \rho_k \geq \eta_1 \), then 
                \[
                \Delta_{k+1} = 
                \begin{cases}
                    \gamma^{-1} \Delta_k & \text{if } \|g_k\| < \zeta_1 \Delta_k, \\
                    \Delta_k & \text{if } \zeta_1 \Delta_k \leq \|g_k\| < \zeta_2 \Delta_k, \\
                    \min\{\gamma \Delta_k, \Delta_{\max} \} & \text{if } \zeta_2 \Delta_k \leq \|g_k\|.
                \end{cases}
                \]
                Otherwise, \( \Delta_{k+1} = \gamma^{-1} \Delta_k \).
            \end{example}

        \subsection{Gradient-scaled update} \label{subsec:Relative gradient}

            \citet{YuanFan01} propose a mechanism to control the trust-region radius based on the gradient norm. The rule ties the radius to the norm of the gradient, and so to the local geometry of the objective. It contracts when the gradient is small and expands when the gradient is large.
            The update is defined as follows.
                \begin{update}{\Rgrad}
                    Let \( 0 < \gamma_1 \leq \gamma_2 < 1 < \gamma_3 \), \( \eta \leq \eta_1 \leq \eta_2 < 1 \). We have
                    $\Delta_k = \mu_k \| g_k \|$
                    where \( \mu_k \) follows the update
                    \[  
                        \mu_{k+1} =
                        \begin{cases}
                            [\gamma_1 \mu_k, \gamma_2 \mu_k), & \text{if } \rho_k < \eta_1, \\
                            [\gamma_2 \mu_k, \mu_k), & \text{if } \eta_1 \leq \rho_k < \eta_2, \\
                            [\mu_k, +\infty), & \text{if } \rho_k \geq \eta_2 %\text{ and } \|s_k\| > \frac{1}{2} \Delta_k, \\
                        \end{cases}
                    \]
                \label{update:relative gradient update}
                \end{update}
            \begin{example}
                \citet{Yuan15} studied the generalisation  $\Delta_k = \mu_k \| g_k \|^\alpha$ for some \( \alpha \in (0, 1] \), for which the worst-case complexity analysis can be found in~\citet{GrapigliaYuanYuan15,GrapigliaYuanYuan16}. The update rule for $\mu_k$ is defined as follows:
                \[  
                    \mu_{k+1} =
                        \begin{cases}
                            \gamma_1 \mu_k, & \text{if } \rho_k < \eta_1, \\
                            \gamma_2 \mu_k, & \text{if } \eta_1 \leq \rho_k < \eta_2, \\
                            \gamma_3 \mu_k, & \text{if } \rho_k \geq \eta_2 \text{ and } \|s_k\| > \frac{1}{2} \Delta_k, \\
                            \mu_k, & \text{otherwise}.
                        \end{cases}
                    \]
                This example showcases how to adjust the expansion to verify the convergence conditions when the Hessian satisfies the weak assumption. 
            \end{example}
             
             \begin{example}
                Moreover, following the discussions on the various update rules, we can consider combining the gradient-scaled update with the adaptive step-size dependent update.  \\
                Consider $ \Delta_k = \mu_k \| g_k \|$ with $\mu_k$ updated as in update rule~\ref{update:relative gradient update} and the step-size dependent update defined as follows 
                \[
                    \mu_{k+1} =
                        \begin{cases}
                            \mu_k & \text{if } \rho_k \geq \eta_2 \text{ and } \norm{s_k} \leq \dfrac{1}{2} \Delta_k, \\
                            R_{\eta}(\rho_k) \, \mu_k, & \text{otherwise }, 
                        \end{cases}
                \]
             \end{example}
            \begin{proposition}
                For all \( k \geq 0 \), we have the following lower bounds on the trust-region parameter \( \mu_k \) and the trust-region radius \( \Delta_k \):
                \begin{equation}
                    \mu_k \geq \frac{\mykappa{lbd} }{M_k} \quad \text{and} \quad  \Delta_k \geq \dfrac{\mykappa{lbd}}{M_k} \, \| g_k \|,
                \end{equation}
                where \( \mykappa{lbd} = \min \left\{ 1, \, \mu_0 \, M_0, \, \gamma_1 \, (1 - \eta_2) \mykappa{mdc} \right\} > 0\).
                \label{prop: mu_k lower bound relative gradient}
            \end{proposition}
            \begin{proof}
                By adapting the proof of Theorem~\ref{thm: successful iteration for low Delta}, and \( \Delta_k = \mu_k \| g_k \| \), we have for any $\eta \in [0, 1)$,
                \begin{equation}
                    \begin{aligned}
                        \mu_k &\leq \dfrac{\mykappa{mdc} (1 - \eta)}{L + \| H_k \|} \text{ implies that } \rho_k \\
                        &\geq \eta \quad \text{ and } \quad \rho_k \\
                        &< \eta \text{ implies that } \mu_k \\
                        &> \frac{\mykappa{mdc} \, (1 - \eta) }{L + \| H_k \|}.
                    \end{aligned}
                    \label{eqn: low mu_k consequence}
                \end{equation}
                We prove the result by induction on \( k \). 
                
                \noindent \emph{Base case.} For \( k = 0 \), we have \( \mu_0 \geq \frac{\mykappa{lbd}}{M_0} \) by definition of \( \mykappa{lbd} \).

                \noindent \emph{Induction step.} Suppose the result holds for some \( k \geq 0 \). We will show it holds for \( k + 1 \).
                \begin{enumerate}[(i)]
                    \item If \( \rho_k < \eta_2\), from~\eqref{eqn: low mu_k consequence} and update rule~\ref{update:relative gradient update}, 
                    $
                        \mu_{k+1} \geq    \gamma_1 \mu_k > \gamma_1 \frac{\mykappa{mdc} \, (1 - \eta_2) }{L + \| H_k \|} \geq \frac{\gamma_1 \, (1 - \eta_2) \mykappa{mdc}}{M_k} \geq \frac{\mykappa{lbd}}{M_{k+1}}. 
                    $
                    \item If \( \rho_k \geq \eta_2\), then 
                    $
                        \mu_{k+1} \geq  \mu_k \geq  \frac{\mykappa{lbd}}{M_k} \geq \frac{\mykappa{lbd}}{M_{k+1}}.
                    $
                \end{enumerate}
                We have for all $k \geq 0$, \( \mu_{k} \geq \frac{\mykappa{lbd}}{M_{k}} \) and \( \Delta_k = \mu_k \| g_k \| \geq \frac{\mykappa{lbd}}{M_k} \, \| g_k \| \).
            \end{proof}
            \citet{CurtisLubbertsRobinson2018} propose a slightly different update rule for the trust-region radius that yields first-order and second-order complexity bounds in a concise convergence analysis. Only considering first-order convergence, the update rule is defined as follows.
            \begin{update}{\Rtruncgrad}
                \[
                \mu_{k+1} \in
                    \begin{cases}
                        \gamma_1 \mu_k , & \text{if } \rho_k < \eta_1 \\
                        [\underline{\mu}, \bar{\mu}] & \text{otherwise} 
                    \end{cases}
                \]
                where \( 0 \leq \eta <\eta_1 < 1 \), \( \gamma_1 \in (0 ; 1 )\),  \( 0 < \underline{\mu} \leq \bar{\mu} < \infty\).
                \label{update:relative gradient update Curtis}
            \end{update}
            \begin{proposition}
                For all \( k \geq 0 \), the lower bounds on the trust-region parameter \( \mu_k \) and the trust-region radius \( \Delta_k \) are similar to Proposition~\ref{prop: mu_k lower bound relative gradient}, 
                where
                \[
                    \mykappa{lbd} = \min \left\{ 1 \, ; \, \mu_0 \, M_0 \, ; \, \underline{\mu} \, M_0 \, ; \, \gamma_1 \, (1 - \eta_1) \mykappa{mdc}\right\}.
                \]
                \label{prop: mu_k lower bound relative gradient Curtis} 
            \end{proposition}
            \begin{proof}
                The proof is similar to the one of Proposition~\ref{prop: mu_k lower bound relative gradient}. 
                We only need to adapt the case when $\rho_k \geq \eta_1$, then
                \[ 
                    \mu_{k+1} \geq \underline{\mu} \geq \underline{\mu} \, \frac{M_0}{M_{k+1}}  \left(\frac{M_{k+1}}{M_0}\right) \geq \frac{\mykappa{lbd}}{M_{k+1}} \quad \text{because} \quad \frac{M_{k+1}}{M_0} \geq 1.
                \]
            \end{proof}
            \begin{remark}
                Update rule~\ref{update:relative gradient update}~or~\ref{update:relative gradient update Curtis}, a priori do not satisfy Condition~\ref{cond: increase of the trust-region radius}. 
                In most convergence analyses, such as~\citet{CurtisLubbertsRobinson2018}, the model Hessian is assumed uniformly bounded, which only needs Condition~\ref{cond: decrease of the trust-region radius} to establish $\lim_{k \to +\infty} \norm{g_k} = 0$. 
            \end{remark}
            For the update rule~\ref{update:relative gradient update}, the convergence analysis can be extended 
            to the weaker Hessian growth Assumption~\ref{as: Linear growth bound model Hessian} by controlling 
            the expansion of the gradient coefficient \( \mu_k \). This is achieved by imposing a uniform bound 
            on \( \mu_k \) across iterations. In particular, when \( \rho_k \ge \eta_2 \), we require
            \[
            \mu_{k+1} \le \min\{\gamma_3 \mu_k,\;\overline{\mu}\},
            \]
            for some constant \( \overline{\mu} > 0 \).
            Proposition~\ref{prop: evolution gradient relative coefficient} bounds the growth of the gradient norm across iterations.
            \begin{proposition}
                Assume Lipschitz continuous gradients (\ref{as: Lipschitz continuous gradient}), and that there exists a constant $\overline{\mu} > 0$ such that for all $k \geq 0$, $\mu_k \leq \overline{\mu}$. 
                Then $\| g_{k+1} \| \leq (L \overline{\mu}  + 1 ) \| g_k \|$.
                \label{prop: evolution gradient relative coefficient}
            \end{proposition}
            \begin{proof}
                    $
                        \| g_{k+1} \| \leq L \| s_k \| + \| g_k \| \leq L \Delta_k + \| g_k \| = L \mu_k \| g_k \| + \| g_k \| = (L \mu_k + 1) \| g_k \| \leq (L \overline{\mu} + 1) \| g_k \|.
                    $
            \end{proof}
              
            \begin{proposition}
                Assume Lipschitz continuous gradients (\ref{as: Lipschitz continuous gradient}) and consider update rule~\ref{update:relative gradient update} where $\rho_k \geq \eta_2$ implies $\mu_{k+1} \in [\mu_k \, ; \, \min\{\gamma_3 \mu_k,\;\overline{\mu}\}]$, $(\gamma_3 > 1)$, and $\gamma_2 \, (L \overline{\mu}  + 1 ) < 1$.
                Then, Condition~\ref{cond: increase of the trust-region radius} holds with $\myeta{inc} = \eta_1$,
                \( \overline{\gamma_3} = \gamma_3 (L \overline{\mu}  + 1 )\) and \(\underline{\gamma_2} = \gamma_2 (L \overline{\mu}  + 1 )\).
                \label{prop: decrease and increase of the trust-region radius relative gradient}
            \end{proposition}
            \begin{proof}
                We use Prop.~\ref{prop: evolution gradient relative coefficient} to control the evolution of the gradient norm across iterations.
                When \( \rho_k < \eta_1 \), we have \( \mu_{k+1} \in [\gamma_1 \mu_k, \gamma_2 \mu_k) \). Then,     
                $
                    \Delta_{k+1} = \mu_{k+1} \| g_{k+1} \| \leq \gamma_2 \mu_k (L \overline{\mu}  + 1 ) \| g_{k} \| = \underline{\gamma_2} \Delta_k,
                $
                where \( \underline{\gamma_2} = \gamma_2 (L \overline{\mu}  + 1 ) < 1 \).\\
                Otherwise,  if \( \rho_k \geq \eta_1 \),
                    $
                        \Delta_{k+1} = \mu_{k+1} \| g_{k+1} \| \leq \gamma_3 \mu_k (L \overline{\mu}  + 1 ) \| g_k \| = \overline{\gamma_3} \Delta_k,
                    $
                    where \( \overline{\gamma_3} = \gamma_3 (L \overline{\mu}  + 1 ) > 1 \).
            \end{proof}
            \begin{remark}
                In Proposition~\ref{prop: decrease and increase of the trust-region radius relative gradient}, we consider the case $\eta < \eta_1$ with a strong condition on  $\gamma_2$ to decouple the acceptance of the step and the update of the radius. 
                The Lipschitz constant of the gradient is in general unknown, and the condition on $\gamma_2$ forces a conservative decrease of the radius.\\
                In practice, it is easier to consider $\eta >0$. 
                This implies that when
                \( \rho_k < \eta \), we have \( \mu_{k+1} \in [\gamma_1 \mu_k, \gamma_2 \mu_k) \) and  $x_{k+1} = x_k $. Then, 
                $
                    \Delta_{k+1} = \mu_{k+1} \| g_{k+1} \| \leq \gamma_2 \mu_k \| g_{k} \| \leq  \gamma_2 \Delta_k
                $.
                Then, we consider $ \myeta{inc} = \eta >0$ to satisfy Condition~\ref{cond: increase of the trust-region radius}.
                \label{remark: eta_dec and eta_inc relative gradient}
            \end{remark}

            The gradient-scaled update provides a lower bound on the trust-region parameter \( \mu_k \) that does not depend on the optimality threshold \( \varepsilon \). \\
            Therefore, by adapting the argument from the second part of the proof of Theorem~\ref{thm: Upper bound successful iter BTR}, we obtain $\frac{1}{\mu_k}  \leq \frac{\mykappa{umh}}{\mykappa{lbd}}$, for $k \in \mathcal{S}$, $\gamma_3^{-1} \frac{1}{\mu_k} \leq \frac{1}{\mu_{k+1}} $, and when $k \in \mathcal{U}$, $\gamma_2^{-1} \frac{1}{\mu_k} \leq \frac{1}{\mu_{k+1}}$.
            Therefore, we get $\frac{1}{\mu_0} \gamma_2^{- \vert \mathcal{U}_{T_{\varepsilon} - 1} \vert} \gamma_3^{- \vert \mathcal{S}_{T_{\varepsilon} - 1} \vert} \leq \frac{\mykappa{umh}}{\mykappa{lbd}}$, and
            \begin{equation}
                \ln\left(\frac{1}{\gamma_2}\right)\vert \mathcal{U}_{T_{\varepsilon} - 1} \vert  \leq \ln\left(\frac{\mu_0 \, \mykappa{umh}}{\mykappa{lbd}} \right) + \ln\left(\gamma_3\right) \vert \mathcal{S}_{T_{\varepsilon} - 1} \vert
            \end{equation}
            and, therefore, we get the following upper bound on unsuccessful iterations, 
            \begin{equation}
                \vert \mathcal{U}_{T_{\varepsilon} - 1} \vert  \leq \dfrac{1}{\vert \ln\gamma_2 \vert} \ln\left(\frac{\mu_0 \,  \mykappa{umh}}{\mykappa{lbd}}  \right) + \dfrac{\ln \gamma_3}{\vert \ln \gamma_2 \vert }	\vert \mathcal{S}_{T_{\varepsilon} - 1} \vert \quad \in \calO\left( \varepsilon^{-2} \right)
            \end{equation}
            where we can see that the logarithmic term does not depend on \( \varepsilon^{-1} \).

        \subsection{Summary of verifications} \label{subsec: first-order summary}

            Table~\ref{tab: first-order summary} collects the constants. Each $\mykappa{lbd}$ is the
            minimum of $1$, of the initial-radius term and of entries coming from the branches on which
            the rule contracts. The solver constant $1-\xi$ enters wherever the radius is rebuilt from
            the step length, which is where Assumption~\ref{as: solver of the TR subproblem} is needed.
            The constants are not bookkeeping only: the bound~\eqref{eqn: Upper bound successful iter BTR}
            on the number of successful iterations carries $\mykappa{lbd}^{-1}$, so a mechanism with a
            smaller constant pays in iterations.

            \begin{table}[tbp]
                \centering
                \caption{Verification of Condition~\ref{cond: delta_k lower bound}, mechanism by
                    mechanism. Every constant is additionally intersected with $1$ and with the
                    initial-radius term $\Delta_0 M_0/\norm{g_0}$, which is common to all mechanisms;
                    for the rules tied to the gradient norm that term reads $\mu_0 M_0$, the two being
                    equal since $\Delta_0 = \mu_0\norm{g_0}$. The constant $\xi \in [0,1)$ bounds the
                    forcing sequence of Assumption~\ref{as: solver of the TR subproblem}, needed only
                    where it appears. For $\RAdaptstep$, $\gone = R_{\eta_1}(-\infty)$.}
                \label{tab: first-order summary}
                \footnotesize
                \begin{tabular}{@{}l l p{76mm}@{}}
                    \toprule
                    Mechanism & Result & $\mykappa{lbd}$, remaining entries \\
                    \midrule
                    $\Rdelta$
                        & Prop.~\ref{prop: delta_k lower bound incremental}
                        & $\gone(1-\eta_2)\mykappa{mdc}$ \\[2pt]
                    $\Rdeltastep$, $\theta_k$ rescaling
                        & Cor.~\ref{cor: aggressive decrease lower bound}
                        & $\gamma_0(1-\eta_2)\mykappa{mdc}$,\; $1-\xi$ \\[2pt]
                    $\Rstep$, $\RAdaptstep$
                        & Prop.~\ref{prop: lower bound Delta_k step-size dependent}
                        & $\gone(1-\eta_2)\mykappa{mdc}$,\; $1-\xi$ \\
                    \midrule
                    $\Rdfo$
                        & Prop.~\ref{prop: delta_k lower bound DFO}
                        & $\gone(1-\eta_1)\mykappa{mdc}$,\; $\gtwo\zeta M_0$ \\[2pt]
                    $\Rgrad$
                        & Prop.~\ref{prop: mu_k lower bound relative gradient}
                        & $\gone(1-\eta_2)\mykappa{mdc}$ \\[2pt]
                    $\Rtruncgrad$
                        & Prop.~\ref{prop: mu_k lower bound relative gradient Curtis}
                        & $\gone(1-\eta_1)\mykappa{mdc}$,\; $\underline{\mu}M_0$ \\
                    \midrule
                    RTR, Alg.~\ref{alg:RTR}
                        & Thm.~\ref{thm: RTR lower bound}
                        & $\gone(1-\eta)\mykappa{mdc}$,\; $\gtwo(1-\xi)$,\;
                          $\gone(1-\tilde\eta_2)\mykappa{mdc}/(3-2\tilde\eta_2)$ \\[2pt]
                    RTR, radius-based variant
                        & Sec.~\ref{sec:Retrospective trust-region radius}
                        & $\gone(1-\eta)\mykappa{mdc}$,\;
                          $\gone(1-\tilde\eta_2)\mykappa{mdc}/(3-2\tilde\eta_2)$ \\[2pt]
                    RTR-NRTR, Alg.~\ref{alg:RTR-NRTR}
                        & Prop.~\ref{prop: lower bound mu_k and Delta_k retrospective}
                        & $\gone(1-\eta_1)\mykappa{mdc}$,\;
                          $\gone(1-\tilde\eta_2)\mykappa{mdc}/(3-2\tilde\eta_2)$ \\
                    \bottomrule
                \end{tabular}
            \end{table}

    \section{Discussion on a recent iteration complexity analysis framework} \label{sec:iteration complexity framework}

        The complexity bound of  Theorem~\ref{thm: Upper bound successful iter BTR} uses the Cauchy decrease assumption and Conditions~\ref{cond: delta_k lower bound}, \ref{cond: decrease of the trust-region radius} and~\ref{cond: increase of the trust-region radius} directly. 
        A complementary route, developed by~\citet{CartSche18, BlanCartMeniSche19} and
        surveyed by~\citet{CurtSche20}, and by~\citet{WangYuan2022}, replaces these  arguments with a single Lyapunov function that decreases on every  iteration. The argument is conceptually simpler when applicable, and yields tighter bounds in the convex and strongly convex cases.

        This template provides an alternative complexity argument for two of our admissible mechanisms. We apply it to the criticality-anchored update of Section~\ref{subsec:DFO-like update mechanism}. We generalise~\citet{CurtSche20} in two ways: we use three distinct scaling coefficients $\gamma_1, \gamma_2, \gamma_3$ (with a corresponding adjustment to $\nu$), and we decouple step acceptance from the radius update to handle $\eta = 0$.
        We then specialise the stochastic analysis of~\citet{WangYuan2022} to the deterministic gradient-scaled update of Section~\ref{subsec:Relative gradient}, deriving complexity bounds in the non-convex, convex, and strongly convex cases.
         
        \subsection{Unifying framework from Line-Search to Trust-Region Methods}

        Throughout this section, $T_\varepsilon$ denotes the first iteration satisfying the convergence criterion of interest, parametrised by $\varepsilon$.
		\begin{cond}{C.Cx}[the complexity framework]
		The statements hold with respect to $\big\{ (\Phi_k , \alpha_k) \big\}$ and $T_{\varepsilon}$:
			\begin{enumerate}
				\item There exists a scalar $\underline{\alpha}_{\varepsilon} \in (0 ; \infty)$ such that, 
						 when  $ \gamma_2  \alpha_k \leq \underline{\alpha}_{\varepsilon}$ then the iteration is successful, and consequently we have for all $k \geq 0$, $\alpha_k \geq \underline{\alpha}_{\varepsilon}$.
				\item There exists a nondecreasing function $h_{\varepsilon} : (0; \infty) \rightarrow (0 ; \infty)$, and a scalar $\Theta > 0$ such that for all $k < T_{\varepsilon}$, 
				\begin{equation}
					\label{decrease in Phi}
					\Phi_k - \Phi_{k+1} \geq \Theta h_{\varepsilon}(\alpha_k).
				\end{equation}	   
                where $\{\Phi_k\}_{k \geq 0} \geq 0$ is a sequence whose role is to measure the progress of the algorithm.
			\end{enumerate}
            \label{condition complexity analysis deterministic}
		\end{cond}	

		\begin{theorem}
			Under Condition~\ref{condition complexity analysis deterministic}, the stopping time $T_{\varepsilon}$ is finite and satisfies $\displaystyle T_{\varepsilon} \leq \dfrac{\Phi_0}{\Theta h_{\varepsilon}(\underline{\alpha}_{\varepsilon})}$.
            \label{thm: bound on stopping time deterministic}
		\end{theorem}    		
		\begin{proof}
			For all $k < T_{\varepsilon}$, we have $\Phi_k - \Phi_{k+1} \geq \Theta h_{\varepsilon}(\alpha_k)$. Therefore, remembering $\Phi_k \geq 0$ for all $k$, 
			$$\Phi_0 \geq \Phi_0 - \Phi_{T_{\varepsilon}}  \geq \sum_{k = 0}^{T_{\varepsilon} - 1} \Phi_k - \Phi_{k+1} \geq  \sum_{k = 0}^{T_{\varepsilon} - 1} \Theta h_{\varepsilon}(\underline{\alpha}_{\varepsilon}) \quad \text{and} \quad T_{\varepsilon} \leq \dfrac{\Phi_0}{\Theta h_{\varepsilon}(\underline{\alpha}_{\varepsilon})}.$$
		\end{proof}
        Theorem~\ref{thm: bound on stopping time deterministic} extends to a per-iteration decrease degrading as $1/M_k$.
        \begin{corollary}
            When considering that $ \sum_k \dfrac{1}{M_k} = \infty$, if we assume that Condition~\ref{condition complexity analysis deterministic} (2) is replaced by: \\
            (2') There exists a nondecreasing function $h_{\varepsilon} : (0; \infty) \rightarrow (0 ; \infty)$, and a scalar $\Theta > 0$ such that for all $k < T_{\varepsilon}$, 
                \begin{equation}
                    \label{decrease in Phi alternative}
                    \Phi_k - \Phi_{k+1} \geq \Theta \dfrac{h_{\varepsilon}(\alpha_k) }{M_k},   
                \end{equation}
            where $M_k$ is defined in~\eqref{eqn: M_k definition}.
            Then, the stopping time $T_{\varepsilon}$ is finite and satisfies $\displaystyle T_{\varepsilon} \leq \dfrac{M_{T_{\varepsilon}} \, \Phi_0}{\Theta h_{\varepsilon}(\underline{\alpha}_{\varepsilon})}$.
        \end{corollary}
        \begin{proof}
            The proof follows the same lines as the proof of Theorem~\ref{thm: bound on stopping time deterministic}, by considering that for all $k < T_{\varepsilon}$, \\
            $\Phi_k - \Phi_{k+1} \geq \Theta \dfrac{h_{\varepsilon}(\alpha_k) }{M_k} \geq \Theta \dfrac{h_{\varepsilon}(\underline{\alpha}_{\varepsilon}) }{M_{T_{\varepsilon}}}$. Then, for $N \in \mathbb{N}$, we have
            $$
                \Phi_0 \geq \Phi_0 - \Phi_{N} \geq \sum_{k = 0}^{N - 1} \Phi_k - \Phi_{k+1} \geq  \sum_{k = 0}^{N - 1} \Theta \dfrac{h_{\varepsilon}(\underline{\alpha}_{\varepsilon}) }{M_{T_{\varepsilon}}} \geq \Theta h_{\varepsilon}(\underline{\alpha}_{\varepsilon}) \sum_{k = 0}^{N - 1} \dfrac{1}{M_{k}}.
            $$
            Therefore, considering $N^+$ to be the largest integer such that $ \sum_{k = 0}^{N^+ - 1} \dfrac{1}{M_{k}} \leq \dfrac{\Phi_0}{\Theta h_{\varepsilon}(\underline{\alpha}_{\varepsilon})}$, we have $T_{\varepsilon} \leq N^+ < \infty$.
        \end{proof}

        \subsection{Application to DFO-Like update rule}

        We show that the criticality-anchored update of Section~\ref{subsec:DFO-like update mechanism} satisfies Condition~\ref{condition complexity analysis deterministic} with the progress metric of~\citet{CurtSche20}, generalised to three scaling coefficients. For $\nu \in (0, 1)$, define
        \begin{equation}
            \Phi_k  = \nu (f(x_k) - \kappa_{lbf}) + (1 - \nu)\Delta_k^2  \quad \text{ and } \quad   T_{\varepsilon} = \inf \big\{ k \geq 0, \; \Vert g(x_k) \Vert < \varepsilon \big\}.
            \label{eqn: progress metric}
        \end{equation}
        The update rules of Subsection~\ref{subsec:DFO-like update mechanism} satisfy the conditions of Theorem~\ref{thm: bound on stopping time deterministic}. The argument adapts~\citet{CurtSche20} to multiple scaling coefficients $\gamma_{1,2,3}$.
        \begin{proposition}
            There exists $\nu \in (0, 1)$ such that
            \begin{equation}
                \nu\, \eta_1\, \frac{\mykappa{mdc}}{2\zeta}\, 
                \min\!\left(\frac{1}{\zeta\, \mykappa{umh}}, 1\right)
                \;\geq\;
                (1-\nu)\big(\gamma_3^2 - \gamma_2^2\big).
                \label{eq:nu choice DFO}
            \end{equation}
            For any such $\nu$, Condition~\ref{condition complexity analysis deterministic} holds with $\alpha_k = \Delta_k$, $h_\varepsilon(x) = x^2$, and $\Theta = (1-\nu)(1-\gamma_2^2)$. Consequently,
            \[
                T_\varepsilon \;\leq\; 
                \frac{\nu\big(f(x_0) - \mykappa{lbf}\big) + (1-\nu)\Delta_0^2}
                    {(1-\nu)(1-\gamma_2^2)\, \underline{\Delta}_\varepsilon^2},
                \quad \text{where } 
                \underline{\Delta}_\varepsilon = \frac{\mykappa{lbd}}{\mykappa{umh}}\, \varepsilon.
            \]
            \label{prop: decrease in Phi DFO-like}
        \end{proposition}
       \begin{proof}
        For $k < T_\varepsilon$, we have $\|g_k\| \geq \varepsilon$. We show  $\Phi_k - \Phi_{k+1} \geq \Theta\, \Delta_k^2$ in two cases. \\
        \textbf{Case 1: $\rho_k \geq \eta_1$.} The Cauchy decrease (\ref{as: Cauchy decrease}) gives
                \[
                    f(x_k)-f(x_{k+1})
                    \;\ge\;
                    \eta_1 \frac{\mykappa{mdc}}{2}\,\|g_k\|
                    \min\!\left\{\Delta_k,\; \frac{\|g_k\|}{\mykappa{umh}}\right\}.
                \]
                \emph{Subcase 1a: $\Delta_k \leq \zeta \|g_k\|$.} Then $f(x_k) - f(x_{k+1}) \geq \eta_1\, c\, \Delta_k^2$ with $c = \frac{\mykappa{mdc}}{2\zeta}\, \min\!\big(1, \frac{1}{\zeta\, \mykappa{umh}}\big)$. Since $\Delta_{k+1} \leq \gamma_3 \Delta_k$,
                \[
                    \Phi_k - \Phi_{k+1} 
                    \;\geq\; \big(\nu\, \eta_1\, c - (1-\nu)(\gamma_3^2 - 1)\big)\, \Delta_k^2 
                    \;\geq\; (1-\nu)(1-\gamma_2^2)\, \Delta_k^2,
                \]
                where the last inequality uses~\eqref{eq:nu choice DFO}. \\
                \emph{Subcase 1b: $\Delta_k > \zeta \|g_k\|$.} Here $\Delta_{k+1} \leq \gamma_2 \Delta_k$, and by monotonic decrease of $f(x_k)$, we have
                \[
                    \Phi_k - \Phi_{k+1} = \nu \left( f(x_k) - f(x_{k+1}) \right) + (1 - \nu) \left( \Delta_k^2 - \Delta_{k+1}^2 \right) \geq (1 - \nu) \big( 1 - \gamma_2^{2} \big) \Delta_k^2.   
                \]                
            \textbf{Case 2: $\rho_k < \eta_1$.}  Again, $\Delta_{k+1} \leq \gamma_2 \Delta_k$ and argument is similar to Subcase 1b. \\
            In all cases, $\Phi_k - \Phi_{k+1} \geq \Theta\, \Delta_k^2$ with $\Theta = (1-\nu)(1-\gamma_2^2)$. The complexity bound follows from Theorem~\ref{thm: bound on stopping time deterministic}.
       \end{proof}

        \subsection{Application to the gradient-scaled update}

        We turn to the gradient-scaled update of Section~\ref{subsec:Relative gradient}, where $\Delta_k = \mu_k \|g_k\|$ where $\mu_k \leq \overline{\mu}$ with appropriate adjustments to the update rule.
        By Remark~\ref{remark: eta_dec and eta_inc relative gradient}, the acceptance threshold satisfies $\eta_1 > \eta$ ($\eta_1 \neq 0$), with the strong condition $\gamma_2 (L \overline{\mu} + 1) < 1$. This condition could be relaxed by considering $\eta > 0$. \\
        The analysis below specialises~\citet{WangYuan2022} to the deterministic case and treats the non-convex, convex, and strongly convex regimes in turn. A common condition on $\nu$ underlies all three.
        \begin{cond}{C.Nu}[choice of $\nu$]
            We choose $\nu \in (0 \, ; \, 1)$ sufficiently close to $1$ such that
            \[ 
                \nu \, \eta_1 \,  \frac{\mykappa{mdc}}{2 \, \overline{\mu}^2 \, \mykappa{umh}} \geq (1 - \nu) \left(L \overline{\mu} + 1 \right)^2 \left[ \gamma_3^2 - \gamma_2^2  \right].
            \]
            \label{cond: nu choice relative gradient}
        \end{cond}
        
        \subsubsection{Non-convex case}
        
        \begin{proposition}
            Consider $\Phi_k$ and $T_{\varepsilon}$ defined in~\eqref{eqn: progress metric}.
            Under Condition~\ref{cond: nu choice relative gradient}, for all $k < T_\varepsilon$,
            \[
                \Phi_k - \Phi_{k+1} \;\geq\; \Theta\, \Delta_k^2,
                \quad \text{where } 
                \Theta = (1-\nu)\Big(1 - (L\overline{\mu}+1)^2\, \gamma_2^2\Big) > 0.
            \]
            Hence Condition~\ref{condition complexity analysis deterministic} holds with $\alpha_k = \Delta_k$, $h_\varepsilon(x) = x^2$, and
            \begin{equation}
                T_{\varepsilon}  \leq \dfrac{\Phi_0}{\Theta \, h_{\varepsilon}(\underline{\Delta}_{\varepsilon})} \leq  \dfrac{ \nu (f(x_0) - \kappa_{lbf}) + (1 - \nu)\Delta_0^2}{(1 - \nu) \big( 1 - \gamma_2^{2} \big) \, h_{\varepsilon}(\underline{\Delta}_{\varepsilon})} \; \in \calO\left( \varepsilon^{-2} \right)
            \end{equation}
        \end{proposition}
        \begin{proof}
            Recall that by Proposition~\ref{prop: evolution gradient relative coefficient}, $\norm{g_{k+1}} \leq \left(L \overline{\mu} + 1 \right)  \| g_k \|$. \\
            \textbf{Case 1: $\rho_k \geq \eta_1$.} The Cauchy decrease gives
            \begin{align*}
                f(x_k) - f(x_{k+1}) &\geq \eta_1 \frac{\mykappa{mdc}}{2} \frac{\Delta_k}{\overline{\mu}} \, \min \left(\Delta_k \, ; \, \frac{\Delta_k}{\overline{\mu} \, \mykappa{umh}}\right) \\
                &\geq \eta_1 \, c \, \Delta_k^2, \quad \text{ where } c \\
                &= \frac{\mykappa{mdc}}{2 \, \overline{\mu}} \, \min \left(1 \, ; \, \frac{1}{\overline{\mu} \, \mykappa{umh}}\right) \\
                &= \frac{\mykappa{mdc}}{2 \, \overline{\mu}^2 \, \mykappa{umh}}.
            \end{align*}
            Since $\Delta_{k+1} \leq \gamma_3(L\overline{\mu} + 1)\Delta_k$,
            \[
                \Phi_k - \Phi_{k+1} 
                \;\geq\; \Big(\nu\, \eta_1\, c - (1-\nu)\big[\gamma_3^2(L\overline{\mu}+1)^2 - 1\big]\Big)\Delta_k^2 
                \;\geq\; (1-\nu)\Big(1 - (L\overline{\mu}+1)^2\, \gamma_2^2\Big) \Delta_k^2.
            \]
            Condition~\ref{cond: nu choice relative gradient} ensures the last inequality. \\
            \textbf{Case 2: $\rho_k < \eta_1$.} Then $\Delta_{k+1} = \mu_{k+1}\|g_{k+1}\| \leq \gamma_2(L\overline{\mu}+1)\Delta_k$, hence
            \[
                \Phi_k - \Phi_{k+1} 
                = (1-\nu)(\Delta_k^2 - \Delta_{k+1}^2) 
                \geq \Theta\, \Delta_k^2.
            \]
            Therefore, in both cases, $\Phi_k - \Phi_{k+1} \geq \Theta \Delta_k^2$, where $\Theta = (1 - \nu) \left( 1 - \gamma_2^{2} \right) > 0$, and Condition~\ref{condition complexity analysis deterministic} is satisfied with $h_{\varepsilon}(x) = x^2$, and $\alpha_k = \Delta_k$. The rest follows from Theorem~\ref{thm: bound on stopping time deterministic}.
        \end{proof}

        \begin{remark}
            If we consider $\eta > 0$, the case distinction can be made on $\rho_k < \eta$ and $\rho_k \geq \eta$, and the same complexity result is obtained with a better constant $\Theta = (1 - \nu) \left( 1 - \gamma_2^{2} \right)$ and no restriction on $\gamma_2$.
            For this, we choose $\nu$ such that 
            \[ \nu \, \eta \,  \frac{\mykappa{mdc}}{2 \, \overline{\mu}^2 \, \mykappa{umh}} \geq (1 - \nu)  \left[ \left(L \overline{\mu} + 1 \right)^2\gamma_3^2 - \gamma_2^2  \right].
            \]
            Then, for $\rho_k < \eta$, we have $x_{k+1} = x_k$, and
            $\Delta_{k+1} = \mu_{k+1} \| g_{k+1} \| \leq \gamma_2 \, \mu_k \, \| g_k \| = \gamma_2 \, \Delta_k$, and therefore, $\Phi_k - \Phi_{k+1} = (1 - \nu) \left( \Delta_k^2 - \Delta_{k+1}^2 \right) \geq (1 - \nu) ( 1 -  \gamma_2^{2} )\Delta_k^2$.
        \end{remark}
        
    \subsubsection{Convex cases}

        Assume $f$ is convex with bounded level set $\mathcal{L}_0 = \{x : f(x) \leq f(x_0)\}$ of diameter $D_0 = \max_{x, y \in \mathcal{L}_0} \|x - y\|$ and $G = \max_{x \in \mathcal{L}_0} \|\nabla f(x)\|$. 
        Following~\citet{WangYuan2022}, the stopping time is
        \begin{equation}
            T_\varepsilon^f = \inf\big\{k \geq 0 : f(x_k) - f^\star < \varepsilon\big\},
            \label{eqn: stopping time convex case}
        \end{equation}
        where $f^\star = \min_x f(x)$. The progress metric is defined as
        \begin{equation}
            \Psi_k = \frac{1}{\nu\, \varepsilon} - \frac{1}{\Phi_k},
            \quad \Phi_k = \nu(f(x_k) - f^\star) + (1-\nu)\mu_k\|\nabla f(x_k)\|^2, \quad \nu \in (0\, ; \, 1).
            \label{eqn: progress metric convex case}
        \end{equation}
        Convexity and Cauchy-Schwarz inequality give the following upper bound on the optimality gap. For all $x \in \calL_0$,
        \begin{equation}
        f(x) - f^\star \leq \nabla f(x)^\top (x - x^\star) \leq \| \nabla f(x) \| \, \| x - x^\star \| \leq D_0 \, \| \nabla f(x) \|.
        \label{eq:convexity upper bound}
        \end{equation}
        \begin{proposition}
            Consider $\Psi_k$ and $T_{\varepsilon}^f$ respectively defined in~\eqref{eqn: progress metric convex case}  and~\eqref{eqn: stopping time convex case}. 
            Under Condition~\ref{cond: nu choice relative gradient}, for all $k < T_\varepsilon^f$,
            \[
                \Psi_k - \Psi_{k+1} \;\geq\; \frac{\Theta}{\omega^2}\, \mu_k,
                \quad \omega = \nu D_0 + (1-\nu)\overline{\mu}\, G,
                \quad \Theta = (1-\nu)\big(1 - (L\overline{\mu}+1)^2\, \gamma_2^2\big).
            \]
            Condition~\ref{condition complexity analysis deterministic} holds with $\alpha_k = \mu_k$,  $h_\varepsilon(x) = x$, and
            \[
                T_{\varepsilon}^f \leq \dfrac{\Psi_0 \, \omega^2 \, \mykappa{umh} }{(1 - \nu)  \left( 1 - \left(L \overline{\mu} + 1 \right)^2 \gamma_2^2 \right) \mykappa{lbd}}  \; \in \calO\left( \varepsilon^{-1} \right).
            \]
        \end{proposition}

        \begin{proof}
            Recall that by Proposition~\ref{prop: evolution gradient relative coefficient}, $\norm{g_{k+1}} \leq \left(L \overline{\mu} + 1 \right)  \| g_k \|$. \\
            \textbf{Case 1: $\rho_k \geq \eta_1$.} By the Cauchy decrease (\ref{as: Cauchy decrease}), $\mu_{k+1} \leq \gamma_3 \, \mu_k$,  we have
            \begin{align*}
                \Phi_k - \Phi_{k+1} & = \nu (f(x_k) - f(x_{k+1})) + (1 - \nu) \left( \mu_k \| \nabla f(x_k) \|^2 - \mu_{k+1} \| \nabla f(x_{k+1}) \|^2 \right) \\
                    & \geq \nu \, \eta \, \frac{\mykappa{mdc}}{2 \, \overline{\mu}^2 \, \mykappa{umh}} \, \mu_k \| \nabla f(x_k) \|^2 \\
                        & \qquad + (1 - \nu) \left( 1 - \gamma_3^2\left(L \overline{\mu} + 1 \right)^2 \right) \, \mu_k \| \nabla f(x_k) \|^2 \\
                    & \geq (1 - \nu) \left( 1 - \left(L \overline{\mu} + 1 \right)^2 \gamma_2^{2} \right) \mu_k \| \nabla f(x_k) \|^2
            \end{align*}
            where the last inequality uses Condition~\ref{cond: nu choice relative gradient}. \\
            \textbf{Case 2: $\rho_k < \eta_1$.} Then, $
                        \Phi_k - \Phi_{k+1}  \geq (1 - \nu)  \left( 1 - \left(L \overline{\mu} + 1 \right)^2 \gamma_2^2 \right) \mu_{k} \| \nabla f(x_k) \|^2$. \\ 
            Combining with~\eqref{eq:convexity upper bound}, with $\omega = \nu D_0 + (1 - \nu) \overline{\mu} G$,
                \begin{equation}
                    \begin{aligned}
                        \Phi_k &\leq \nu D_0 \, \| \nabla f(x_k) \| + (1 - \nu) \overline{\mu} \, \| \nabla f(x_k) \|^2
                               \leq \omega \, \| \nabla f(x_k) \|, \\
                        \text{hence} \quad \Phi_k - \Phi_{k+1} &\geq \frac{\Theta}{\omega^2}\, \mu_k\, \Phi_k^2.
                    \end{aligned}
                    \label{eqn: upper bound Phi_k convex case}
                \end{equation}
                Dividing by $\Phi_k\, \Phi_{k+1}$ yields for all $k < T_\varepsilon^f$,
                \[
                    \Psi_k - \Psi_{k+1} 
                    = \frac{1}{\Phi_{k+1}} - \frac{1}{\Phi_k} 
                    \geq \frac{\Theta}{\omega^2}\, \mu_k.
                \]
                The lower bound $\mu_k \geq \underline{\mu} = \mykappa{lbd} / \mykappa{umh}$ from Proposition~\ref{prop: mu_k lower bound relative gradient} closes the argument.
            \end{proof}

        \subsubsection{Strongly convex case}

            Assume $f$ is $\sigma$-strongly convex ($\sigma>0$), and let $x^\star$ be the global minimiser with $f^\star=\min_x f(x) = f(x^\star)$.
            Define, 
            \[
                \Psi_k = \ln\left( \Phi_k \right) - \ln\left(\nu \,  \varepsilon \right), \quad  \Phi_k = \nu (f(x_k) - f^\star) + (1 - \nu) \mu_k \| \nabla f(x_k) \|^2 \quad \nu \in (0\, ; \, 1).
            \]
            \begin{proposition}
                Consider $\Psi_k$ and $T_{\varepsilon}^f$ defined above.
                Under Condition~\ref{cond: nu choice relative gradient}, for all 
                $k < T_\varepsilon^f$,
                \[
                    \Psi_k - \Psi_{k+1} \;\geq\; \Theta\, \frac{\mu_k}{\overline{\mu}},
                    \quad \Theta = \frac{2\sigma(1-\nu)\big(1 - (L\overline{\mu}+1)^2 \gamma_2^2\big)}
                                    {\nu/\overline{\mu} + 2\sigma(1-\nu)} \in (0, 1).
                \]
                Condition~\ref{condition complexity analysis deterministic} holds with $\alpha_k = \mu_k$, 
                $h_\varepsilon(x) = x/\overline{\mu}$, and 
                \[
                    T_{\varepsilon}^f \leq \dfrac{\Psi_0}{\Theta \, h_{\varepsilon}(\underline{\mu})} = \dfrac{\mykappa{umh} \, \overline{\mu} \, \Big(\ln\left( \Phi_0 \right) - \ln\left(\nu \,  \varepsilon \right)\Big) }{\mykappa{lbd} \, \Theta }  \; \in \calO\left( \ln\left( \varepsilon^{-1} \right) \right).
                \]
            \end{proposition}
            \begin{proof}
                The standard inequality $\|\nabla f(x)\|^2 \ge 2\sigma(f(x)-f^\star)$ gives
                $\Phi_k \le [\nu/(2\sigma) + (1-\nu)\overline{\mu}]\,\|\nabla f(x_k)\|^2$ for all $k < T_{\varepsilon}^f$.
                Again, with a similar analysis as in the convex case,
                \begin{align*}
                    \Phi_k - \Phi_{k+1}
                    &\geq (1 - \nu) \bigl( 1 - (L \overline{\mu} + 1)^2 \gamma_2^2 \bigr) \mu_{k} \| \nabla f(x_k) \|^2 \\
                    &\geq \frac{ 2 \sigma (1 - \nu) \bigl( 1 - (L \overline{\mu} + 1)^2 \gamma_2^2 \bigr) }
                            {\nu / \overline{\mu} + 2 \sigma (1 - \nu) } \, \frac{\mu_k}{\overline{\mu}} \, \Phi_k .
                \end{align*}
                Therefore $\Phi_{k+1} \leq [1 - \Theta \mu_k/\overline{\mu}] \Phi_k$ with $\Theta < 1$ and $\mu_k/\overline{\mu} \leq 1$, which shows a linear convergence rate.
                Using $-\ln (1 - x) \geq x$ for $x \in (0\,;\,1)$, we have
                $\Psi_k - \Psi_{k+1} = \ln(\Phi_k) - \ln(\Phi_{k+1}) \geq - \ln(1 - \Theta \mu_k/\overline{\mu}) \geq \Theta h_{\varepsilon}(\mu_k)$.
                Proposition~\ref{prop: mu_k lower bound relative gradient} provides the lower bound on $\mu_k$, and Theorem~\ref{thm: bound on stopping time deterministic} gets us the desired result.
            \end{proof}

\section{Conclusion} \label{sec:conclusion}

        \prop{For trust-region methods applied to unconstrained optimisation, we isolated} three structural conditions on the radius update rule: a lower bound relative to the gradient norm (Condition~\ref{cond: delta_k lower bound}), a contraction on unsuccessful iterations (Condition~\ref{cond: decrease of the trust-region radius}), and a controlled expansion on successful ones (Condition~\ref{cond: increase of the trust-region radius}). Algorithm~\ref{alg:generic BTR} provides a generic framework that holds independently of the specific update rule. Figure~\ref{fig: conditions on radius update} illustrates the relation between the conditions and the convergence properties of the method. Under uniformly bounded model Hessians, weak admissibility yields $\lim_{k\to\infty}\norm{g_k} = 0$ and strong admissibility yields the optimal worst-case complexity $\calO(\varepsilon^{-2})$. Strong admissibility extends the convergence guarantee to linearly growing model Hessians.

        We verified admissibility for five mechanism classes from the literature: fixed-factor, step-driven, retrospective, criticality-anchored, and gradient-scaled. We extended the retrospective analysis of~\citet{BastMalmMoufToinToma10} to linearly growing model Hessians, and adapted the Lyapunov framework of~\citet{CurtSche20, WangYuan2022} to the deterministic gradient-scaled update. The constants in the complexity bound depend on the mechanism and may affect practical performance.

        The proposed framework characterises first-order convergence and worst-case complexity. It does not address the asymptotic behaviour of $\Delta_k$, the local convergence rate near a minimiser, or second-order stationarity. Whether these properties admit a similar condition-based characterisation is open.
		We take up these questions in \PartII{}, where we analyse the asymptotic dynamics of $\Delta_k$, distinguishing the regime where the radius vanishes ($\sum_k \Delta_k < +\infty$ or $\sum_k \Delta_k^2 < +\infty$) from the regime where it stays bounded away from zero. We then study the local convergence behaviour, including the case where the trust-region constraint becomes inactive after finitely many iterations, allowing Newton or quasi-Newton steps and superlinear local rates.

		Beyond \PartII{}, two main directions stand out. The first is systematic numerical benchmarking of the five mechanism classes on standard test sets, to quantify the practical impact of the constants and to identify problem-dependent recommendations.
        The second is the extension to stochastic and noisy settings, where the radius can serve as a precision parameter for sampling, following the works from~\citet{BandScheVice14,BlanCartMeniSche19,PaquSche2020, CartSche18,JinScheinbergXie2025, JinScheinbergXie2024, ChenMeniSche18, CurtSche20}.
    
\begin{figure}[ht!]
                    \centering
                \begin{tikzpicture}[
                    x=1cm, y=1cm,
                    box/.style={draw, very thick, inner sep=4pt, minimum width=5.9cm},
                    hypbox/.style={box, minimum height=3.5cm},
                    conclbox/.style={box, minimum height=1.2cm},
                    arr/.style={->, very thick},
                    edgelab/.style={fill=white, inner sep=3pt, align=center, font=\small}
                    ]

                    %  Top row: the lower bound on the radius.
                    \node[box, anchor=north] (Aup) at (0,0) {
                        \begin{minipage}{5.5cm}
                            \centering
                            \textbf{Condition \ref{cond: delta_k lower bound}} \\[6pt]
                            $\displaystyle \Delta_k \geq \frac{\mykappa{lbd}}{M_k} \min_{0 \leq i \leq k} \| g_i \|$ \\[4pt]
                            $M_k = L + \max_{0 \leq i \leq k} \| H_i \|$.
                        \end{minipage}
                    };

                    %  Middle row: the two admissibility routes.  Both boxes carry the same
                    %  minimum height so their tops and bottoms align and no connector has to
                    %  cross a box to reach the row below.
                    \node[hypbox, anchor=north] (A) at (-3.3,-3.1) {
                        \begin{minipage}{5.5cm}
                            \small
                            \begin{enumerate}[(i)]
                                \item \textbf{\ref{as: model hessian}: }
                                $ \norm{H_k} \leq \mykappa{umh} $
                                \item \textbf{Condition \ref{cond: decrease of the trust-region radius}}: \\
                                For some $\myeta{dec} \in (0\, ; \, 1)$, \\
                                $\Delta_{k+1} \leq \gamma_2 \Delta_{k} $ if $\rho_k < \myeta{dec}$
                            \end{enumerate}
                        \end{minipage}
                    };
                    \node[hypbox, anchor=north] (B) at (3.3,-3.1) {
                        \begin{minipage}{5.5cm}
                            \small
                            \begin{enumerate}[(i)]
                                \item \textbf{\ref{as: Linear growth bound model Hessian}: }
                                $ \norm{H_k} \leq M_k $  \\ with $\sum_k \frac{1}{M_k} = +\infty$
                                \item \textbf{Condition \ref{cond: increase of the trust-region radius}}: \\
                                For some $\myeta{inc} \in (0\, ; \, 1)$, \\
                                $\Delta_{k+1} \leq \gamma_2 \Delta_{k} $ if $\rho_k < \myeta{inc}$ \\
                                $\Delta_{k+1} \leq \gamma_3 \Delta_{k} $ otherwise
                            \end{enumerate}
                        \end{minipage} };

                    %  Conclusions.
                    \node[conclbox, anchor=north] (C) at (0,-7.5) {
                        \begin{minipage}{5.5cm}
                            \centering
                            $\liminf_{k \to +\infty} \| g_k \| = 0$
                        \end{minipage}  };

                    \node[conclbox, anchor=north] (D) at (0,-11.0) {
                        \begin{minipage}{5.5cm}
                            \centering
                            $\lim_{k \to +\infty} \| g_k \| = 0$
                        \end{minipage}
                    };

                    %  Orthogonal routing throughout: every connector leaves a box downwards,
                    %  travels horizontally in the gap between rows, and enters the next box
                    %  from above.  No segment passes over a box or its text.
                    \draw[arr] (Aup.south) -- ++(0,-0.55) -| (A.north);
                    \draw[arr] (Aup.south) -- ++(0,-0.55) -| (B.north);

                    \draw[arr] (A.south) -- ++(0,-0.55) -| ([xshift=-1.6cm]C.north);
                    \draw[arr] (B.south) -- ++(0,-0.55) -| ([xshift= 1.6cm]C.north);

                    \draw[arr] (C.south) -- node[edgelab]{Generic BTR (Algo~\ref{alg:generic BTR}) with $\eta > 0$ \\ + \ref{as: model hessian}} (D.north);

                    \end{tikzpicture}
                    \caption{Influence diagram for first-order convergence of the generic trust-region method (Algorithm~\ref{alg:generic BTR}). Condition~\ref{cond: delta_k lower bound} lower-bounds the radius by the criticality measure; Condition~\ref{cond: decrease of the trust-region radius} contracts on unsuccessful iterations; Condition~\ref{cond: increase of the trust-region radius} caps expansion on successful ones. Under the uniform Hessian bound \ref{as: model hessian}, weak admissibility (\ref{cond: delta_k lower bound}$+$\ref{cond: decrease of the trust-region radius}) yields $\lim_{k\to\infty}\|g_k\|=0$, and strong admissibility (\ref{cond: delta_k lower bound}$+$\ref{cond: increase of the trust-region radius}) adds the optimal complexity $\mathcal{O}(\varepsilon^{-2})$. Under the weaker linear-growth assumption~\ref{as: Linear growth bound model Hessian}, strong admissibility still yields $\liminf_{k\to\infty}\|g_k\|=0$.}

                    \label{fig: conditions on radius update}
                \end{figure}

\clearpage
\section*{Acknowledgements}
This research was funded by a Discovery Grant of the Natural Sciences and Engineering
Research Council of Canada (NSERC), held by the second author.

\bibliography{References}         

\begin{thebibliography}{67}
\providecommand{\natexlab}[1]{#1}
\providecommand{\url}[1]{\texttt{#1}}
\expandafter\ifx\csname urlstyle\endcsname\relax
  \providecommand{\doi}[1]{doi: #1}\else
  \providecommand{\doi}{doi: \begingroup \urlstyle{rm}\Url}\fi

\bibitem[Bandeira et~al.(2014)Bandeira, Scheinberg, and
  Vicente]{BandScheVice14}
A.~S. Bandeira, K.~Scheinberg, and L.~N. Vicente.
\newblock Convergence of {Trust}-{Region} {Methods} {Based} on {Probabilistic}
  {Models}.
\newblock \emph{{SIAM} Journal on Optimization}, 24\penalty0 (3):\penalty0
  1238--1264, 2014.
\newblock \doi{10.1137/130915984}.

\bibitem[Bastin et~al.(2010)Bastin, Malmedy, Mouffe, Toint, and
  Tomanos]{BastMalmMoufToinToma10}
F.~Bastin, V.~Malmedy, M.~Mouffe, P.~L. Toint, and D.~Tomanos.
\newblock A retrospective trust-region method for unconstrained optimization.
\newblock \emph{Mathematical Programming}, 123\penalty0 (2):\penalty0 395--418,
  2010.
\newblock \doi{10.1007/s10107-008-0258-1}.

\bibitem[Birgin et~al.(2017)Birgin, Gardenghi, Martínez, Santos, and
  Toint]{BirgGardMartSantToin17}
E.~G. Birgin, J.~L. Gardenghi, J.~M. Martínez, S.~A. Santos, and P.~L. Toint.
\newblock Worst-case evaluation complexity for unconstrained nonlinear
  optimization using high-order regularized models.
\newblock \emph{Mathematical Programming}, 163\penalty0 (1):\penalty0 359--368,
  2017.
\newblock \doi{10.1007/s10107-016-1065-8}.

\bibitem[Blanchet et~al.(2019)Blanchet, Cartis, Menickelly, and
  Scheinberg]{BlanCartMeniSche19}
J.~Blanchet, C.~Cartis, M.~Menickelly, and K.~Scheinberg.
\newblock Convergence {Rate} {Analysis} of a {Stochastic} {Trust}-{Region}
  {Method} via {Supermartingales}.
\newblock \emph{{INFORMS} Journal on Optimization}, 1\penalty0 (2):\penalty0
  92--119, 2019.
\newblock \doi{10.1287/ijoo.2019.0016}.

\bibitem[Bollapragada et~al.(2018{\natexlab{a}})Bollapragada, Byrd, and
  Nocedal]{BollByrdNoce18SIAM}
R.~Bollapragada, R.~Byrd, and J.~Nocedal.
\newblock Adaptive sampling strategies for stochastic optimization.
\newblock \emph{SIAM Journal on Optimization}, 28\penalty0 (4):\penalty0
  3312--3343, 2018{\natexlab{a}}.
\newblock \doi{10.1137/17m1154679}.

\bibitem[Bollapragada et~al.(2018{\natexlab{b}})Bollapragada, Byrd, and
  Nocedal]{BollByrdNoce18IMA}
R.~Bollapragada, R.~H. Byrd, and J.~Nocedal.
\newblock Exact and inexact subsampled {N}ewton methods for optimization.
\newblock \emph{IMA Journal of Numerical Analysis}, 39\penalty0 (2):\penalty0
  545--578, 2018{\natexlab{b}}.
\newblock \doi{10.1093/imanum/dry009}.

\bibitem[Bottou et~al.(2018)Bottou, Curtis, and Nocedal]{BottCurtNoce18}
L.~Bottou, F.~E. Curtis, and J.~Nocedal.
\newblock Optimization methods for large-scale {M}achine {L}earning.
\newblock \emph{SIAM Review}, 60\penalty0 (2):\penalty0 223--311, 2018.
\newblock \doi{10.1137/16m1080173}.

\bibitem[Byrd et~al.(2012)Byrd, Chin, Nocedal, and Wu]{ByrdChinNoceWu12}
R.~H. Byrd, G.~M. Chin, J.~Nocedal, and Y.~Wu.
\newblock Sample size selection in optimization methods for {M}achine
  {L}earning.
\newblock \emph{Mathematical Programming}, 134\penalty0 (1):\penalty0 127--155,
  2012.
\newblock \doi{10.1007/s10107-012-0572-5}.

\bibitem[Carter(1991)]{Cart91}
R.~G. Carter.
\newblock On the global convergence of trust region algorithms using inexact
  gradient information.
\newblock \emph{{SIAM} Journal on Numerical Analysis}, 28:\penalty0 251--265,
  1991.
\newblock \doi{10.1137/0728014}.

\bibitem[Cartis and Scheinberg(2018)]{CartSche18}
C.~Cartis and K.~Scheinberg.
\newblock Global convergence rate analysis of unconstrained optimization
  methods based on probabilistic models.
\newblock \emph{Mathematical Programming}, 169\penalty0 (2):\penalty0 337--375,
  Jun 2018.
\newblock ISSN 1436-4646.
\newblock \doi{10.1007/s10107-017-1137-4}.

\bibitem[Cartis et~al.(2010)Cartis, Gould, and Toint]{CartGoulToin10}
C.~Cartis, N.~I.~M. Gould, and P.~L. Toint.
\newblock On the complexity of steepest descent, {N}ewton's and regularized
  {N}ewton's methods for nonconvex unconstrained optimization problems.
\newblock \emph{SIAM Journal on Optimization}, 20\penalty0 (6):\penalty0
  2833--2852, 2010.
\newblock \doi{10.1137/090774100}.

\bibitem[Cartis et~al.(2011{\natexlab{a}})Cartis, Gould, and
  Toint]{CartGoulToin11}
C.~Cartis, N.~I.~M. Gould, and P.~L. Toint.
\newblock Adaptive cubic regularisation methods for unconstrained optimization.
  {Part} {I}: motivation, convergence and numerical results.
\newblock \emph{Mathematical Programming}, 127\penalty0 (2):\penalty0 245--295,
  2011{\natexlab{a}}.
\newblock \doi{10.1007/s10107-009-0286-5}.

\bibitem[Cartis et~al.(2011{\natexlab{b}})Cartis, Gould, and
  Toint]{CartisGouldToint11-composite}
C.~Cartis, N.~I.~M. Gould, and P.~L. Toint.
\newblock On the {Evaluation} {Complexity} of {Composite} {Function}
  {Minimization} with {Applications} to {Nonconvex} {Nonlinear} {Programming}.
\newblock \emph{SIAM Journal on Optimization}, 21\penalty0 (4):\penalty0
  1721--1739, 2011{\natexlab{b}}.
\newblock \doi{10.1137/11082381X}.

\bibitem[Cartis et~al.(2012)Cartis, Gould, and Toint]{CartGoulToin12}
C.~Cartis, N.~I.~M. Gould, and P.~L. Toint.
\newblock Complexity bounds for second-order optimality in unconstrained
  optimization.
\newblock \emph{Journal of Complexity}, 28\penalty0 (1):\penalty0 93--108,
  2012.
\newblock \doi{10.1016/j.jco.2011.06.001}.

\bibitem[Cartis et~al.(2022)Cartis, Gould, and Toint]{CartGoulToin22}
C.~Cartis, N.~I.~M. Gould, and P.~L. Toint.
\newblock \emph{Evaluation Complexity of Algorithms for Nonconvex Optimization:
  Theory, Computation and Perspectives}.
\newblock Society for Industrial and Applied Mathematics, Philadelphia, PA,
  2022.
\newblock \doi{10.1137/1.9781611976991}.

\bibitem[Chen et~al.(2018)Chen, Menickelly, and Scheinberg]{ChenMeniSche18}
R.~Chen, M.~Menickelly, and K.~Scheinberg.
\newblock Stochastic optimization using a trust-region method and random
  models.
\newblock \emph{Mathematical Programming}, 169\penalty0 (2):\penalty0 447--487,
  2018.
\newblock \doi{10.1007/s10107-017-1141-8}.

\bibitem[Conn et~al.(2000)Conn, Gould, and Toint]{ConnGoulToin00}
A.~R. Conn, N.~I.~M. Gould, and P.~L. Toint.
\newblock \emph{Trust-{Region} {Methods}}.
\newblock {MOS}-{SIAM} {Series} on {Optimization}. {SIAM}, Philadelphia, PA,
  USA, 2000.
\newblock \doi{10.1137/1.9780898719857}.

\bibitem[Conn et~al.(2009{\natexlab{a}})Conn, Scheinberg, and
  Vicente]{ConnScheVice09}
A.~R. Conn, K.~Scheinberg, and L.~N. Vicente.
\newblock Global {Convergence} of {General} {Derivative}-{Free}
  {Trust}-{Region} {Algorithms} to {First}- and {Second}-{Order} {Critical}
  {Points}.
\newblock \emph{{SIAM} Journal on Optimization}, 20\penalty0 (1):\penalty0
  387--415, 2009{\natexlab{a}}.
\newblock \doi{10.1137/060673424}.

\bibitem[Conn et~al.(2009{\natexlab{b}})Conn, Scheinberg, and
  Vicente]{ConnScheVice09_introduction}
A.~R. Conn, K.~Scheinberg, and L.~N. Vicente.
\newblock \emph{Introduction to {Derivative}-{Free} {Optimization}}.
\newblock Society for Industrial and Applied Mathematics, USA,
  2009{\natexlab{b}}.

\bibitem[Curtis and Scheinberg(2020)]{CurtSche20}
F.~E. Curtis and K.~Scheinberg.
\newblock Adaptive {Stochastic} {Optimization}: {A} {Framework} for {Analyzing}
  {Stochastic} {Optimization} {Algorithms}.
\newblock \emph{IEEE Signal Processing Magazine}, 37\penalty0 (5):\penalty0
  32--42, 2020.
\newblock \doi{10.1109/MSP.2020.3003539}.

\bibitem[Curtis et~al.(2017)Curtis, Robinson, and
  Samadi]{CurtisRobinsonSamadi2017}
F.~E. Curtis, D.~P. Robinson, and M.~Samadi.
\newblock A trust region algorithm with a worst-case iteration complexity of
  $\mathcal{O}(\epsilon^{-3/2})$ for nonconvex optimization.
\newblock \emph{Mathematical Programming}, 162\penalty0 (1):\penalty0 1--32,
  Mar 2017.
\newblock ISSN 1436-4646.
\newblock \doi{10.1007/s10107-016-1026-2}.

\bibitem[Curtis et~al.(2018)Curtis, Lubberts, and
  Robinson]{CurtisLubbertsRobinson2018}
F.~E. Curtis, Z.~Lubberts, and D.~P. Robinson.
\newblock Concise complexity analyses for trust region methods.
\newblock \emph{Optimization Letters}, 12\penalty0 (8):\penalty0 1713--1724,
  2018.
\newblock \doi{10.1007/s11590-018-1286-2}.

\bibitem[Curtis et~al.(2021)Curtis, Robinson, Royer, and
  Wright]{CurtisRobinsonRoyerWright2021}
F.~E. Curtis, D.~P. Robinson, C.~W. Royer, and S.~J. Wright.
\newblock Trust-{Region} {Newton}-{CG} with {Strong} {Second}-{Order}
  {Complexity} {Guarantees} for {Nonconvex} {Optimization}.
\newblock \emph{{SIAM} Journal on Optimization}, 31\penalty0 (1):\penalty0
  518--544, 2021.
\newblock \doi{10.1137/19M130563X}.

\bibitem[Ekeland(1974)]{Ekeland74}
I.~Ekeland.
\newblock On the variational principle.
\newblock \emph{Journal of Mathematical Analysis and Applications}, 47\penalty0
  (2):\penalty0 324--353, 1974.
\newblock \doi{10.1016/0022-247X(74)90025-0}.

\bibitem[Fan(2006)]{Fan06}
J.~Fan.
\newblock Convergence {Rate} of {The} {Trust} {Region} {Method} for {Nonlinear}
  {Equations} {Under} {Local} {Error} {Bound} {Condition}.
\newblock \emph{Computational Optimization and Applications}, 34\penalty0
  (2):\penalty0 215--227, 2006.
\newblock \doi{10.1007/s10589-005-3078-8}.

\bibitem[Fan and Lu(2015)]{FanLu15}
J.~Fan and N.~Lu.
\newblock On the modified trust region algorithm for nonlinear equations.
\newblock \emph{Optimization Methods and Software}, 30\penalty0 (3):\penalty0
  478--491, 2015.
\newblock \doi{10.1080/10556788.2014.932943}.
\newblock Publisher: Taylor \& Francis.

\bibitem[Fan and Pan(2011)]{FanPan11}
J.~Fan and J.~Pan.
\newblock An improved trust region algorithm for nonlinear equations.
\newblock \emph{Computational Optimization and Applications}, 48\penalty0
  (1):\penalty0 59--70, 2011.
\newblock \doi{10.1007/s10589-009-9236-7}.

\bibitem[Fan et~al.(2016)Fan, Pan, and Hongyan]{FanPan16}
J.~Fan, J.~Pan, and S.~Hongyan.
\newblock A retrospective trust region algorithm with trust region converging
  to zero.
\newblock \emph{Journal of Computational Mathematics}, 34\penalty0
  (4):\penalty0 421--436, 2016.
\newblock \doi{10.4208/jcm.1601-m2015-0399}.

\bibitem[Friedlander and Schmidt(2012)]{FrieSchm12}
M.~P. Friedlander and M.~Schmidt.
\newblock Hybrid deterministic-stochastic methods for data fitting.
\newblock \emph{SIAM Journal on Scientific Computing}, 34\penalty0
  (3):\penalty0 A1380--A1405, 2012.
\newblock \doi{10.1137/110830629}.

\bibitem[Garmanjani et~al.(2016)Garmanjani, Júdice, and
  Vicente]{GarmanjaniJudiceVicente16}
R.~Garmanjani, D.~Júdice, and L.~N. Vicente.
\newblock Trust-{Region} {Methods} {Without} {Using} {Derivatives}: {Worst}
  {Case} {Complexity} and the {NonSmooth} {Case}.
\newblock \emph{SIAM Journal on Optimization}, 26\penalty0 (4):\penalty0
  1987--2011, 2016.
\newblock \doi{10.1137/151005683}.

\bibitem[Gould et~al.(2005)Gould, Orban, Sartenaer, and
  Toint]{GouldOrbanSartenaerToint05}
N.~I.~M. Gould, D.~Orban, A.~Sartenaer, and P.~L. Toint.
\newblock Sensitivity of trust-region algorithms to their parameters.
\newblock \emph{4OR}, 3\penalty0 (3):\penalty0 227--241, 2005.
\newblock \doi{10.1007/s10288-005-0065-y}.

\bibitem[Grapiglia et~al.(2015)Grapiglia, Yuan, and Yuan]{GrapigliaYuanYuan15}
G.~N. Grapiglia, J.~Yuan, and Y.-x. Yuan.
\newblock On the convergence and worst-case complexity of trust-region and
  regularization methods for unconstrained optimization.
\newblock \emph{Mathematical Programming}, 152\penalty0 (1):\penalty0 491--520,
  2015.
\newblock \doi{10.1007/s10107-014-0794-9}.

\bibitem[Grapiglia et~al.(2016)Grapiglia, Yuan, and Yuan]{GrapigliaYuanYuan16}
G.~N. Grapiglia, J.~Yuan, and Y.-x. Yuan.
\newblock Nonlinear {Stepsize} {Control} {Algorithms}: {Complexity} {Bounds}
  for {First}- and {Second}-{Order} {Optimality}.
\newblock \emph{Journal of Optimization Theory and Applications}, 171\penalty0
  (3):\penalty0 980--997, 2016.
\newblock \doi{10.1007/s10957-016-1007-x}.

\bibitem[Gratton et~al.(2008)Gratton, Sartenaer, and
  Toint]{GrattonSartenaerToint08}
S.~Gratton, A.~Sartenaer, and P.~L. Toint.
\newblock Recursive {Trust}-{Region} {Methods} for {Multiscale} {Nonlinear}
  {Optimization}.
\newblock \emph{{SIAM} Journal on Optimization}, 19\penalty0 (1):\penalty0
  414--444, 2008.
\newblock \doi{10.1137/050623012}.

\bibitem[Gratton et~al.(2018)Gratton, Royer, Vicente, and
  Zhang]{GratRoyeViceZhan17}
S.~Gratton, C.~W. Royer, L.~N. Vicente, and Z.~Zhang.
\newblock Complexity and global rates of trust-region methods based on
  probabilistic models.
\newblock \emph{IMA Journal of Numerical Analysis}, 38\penalty0 (3):\penalty0
  1579--1597, 07 2018.
\newblock ISSN 0272-4979.
\newblock \doi{10.1093/imanum/drx043}.

\bibitem[Hei(2003)]{Hei03}
L.~Hei.
\newblock A self-adaptive trust region algorithm.
\newblock \emph{Journal of Computational Mathematics}, 21\penalty0
  (2):\penalty0 229--236, 2003.
\newblock Publisher: Institute of Computational Mathematics and
  Scientific/Engineering Computing.

\bibitem[Jin et~al.(2024)Jin, Scheinberg, and Xie]{JinScheinbergXie2024}
B.~Jin, K.~Scheinberg, and M.~Xie.
\newblock High probability complexity bounds for adaptive step search based on
  stochastic oracles.
\newblock \emph{SIAM Journal on Optimization}, 34\penalty0 (3):\penalty0
  2411--2439, 2024.
\newblock \doi{10.1137/22M1512764}.

\bibitem[Jin et~al.(2025)Jin, Scheinberg, and Xie]{JinScheinbergXie2025}
B.~Jin, K.~Scheinberg, and M.~Xie.
\newblock Sample complexity analysis for adaptive optimization algorithms with
  stochastic oracles.
\newblock \emph{Mathematical Programming}, 209\penalty0 (1):\penalty0 651--679,
  2025.
\newblock \doi{10.1007/s10107-024-02078-z}.

\bibitem[Levenberg(1944)]{Levenberg_1944}
K.~Levenberg.
\newblock A method for the solution of certain non-linear problems in least
  squares.
\newblock \emph{Quarterly of Applied Mathematics}, 2\penalty0 (2):\penalty0
  164--168, 1944.
\newblock Publisher: Brown University.

\bibitem[Marquardt(1963)]{Marquardt_1963}
D.~W. Marquardt.
\newblock An {Algorithm} for {Least}-{Squares} {Estimation} of {Nonlinear}
  {Parameters}.
\newblock \emph{Journal of the Society for Industrial and Applied Mathematics},
  11\penalty0 (2):\penalty0 431--441, 1963.
\newblock \doi{10.1137/0111030}.

\bibitem[Moré(1983)]{More83}
J.~J. Moré.
\newblock Recent {Developments} in {Algorithms} and {Software} for {Trust}
  {Region} {Methods}.
\newblock In A.~Bachem, B.~Korte, and M.~Grötschel, editors,
  \emph{Mathematical {Programming} {The} {State} of the {Art}: {Bonn} 1982},
  pages 258--287. Springer Berlin Heidelberg, Berlin, Heidelberg, 1983.
\newblock \doi{10.1007/978-3-642-68874-4_11}.

\bibitem[Moré and Sorensen(1983)]{MoreSorensen1983}
J.~J. Moré and D.~C. Sorensen.
\newblock Computing a {Trust} {Region} {Step}.
\newblock \emph{SIAM Journal on Scientific and Statistical Computing},
  4\penalty0 (3):\penalty0 553--572, 1983.
\newblock \doi{10.1137/0904038}.

\bibitem[Nesterov(2014)]{Nest14}
Y.~Nesterov.
\newblock \emph{Introductory Lectures on Convex Optimization: A Basic Course}.
\newblock Springer Publishing Company, Incorporated, 1st edition, 2014.
\newblock ISBN 1461346916.

\bibitem[Nesterov(2018)]{Nest18}
Y.~Nesterov.
\newblock \emph{Lectures on Convex Optimization}.
\newblock Springer Optimization and Its Applications. Springer Cham, 2nd
  edition, 2018.
\newblock ISBN 978-3-319-91577-7.
\newblock \doi{10.1007/978-3-319-91578-4}.

\bibitem[Nesterov and Polyak(2006)]{NestPoly06}
Y.~Nesterov and B.~Polyak.
\newblock Cubic regularization of {Newton} method and its global performance.
\newblock \emph{Mathematical Programming}, 108\penalty0 (1):\penalty0 177--205,
  2006.
\newblock \doi{10.1007/s10107-006-0706-8}.

\bibitem[Nocedal and Wright(2006)]{NoceWrig06}
J.~Nocedal and S.~J. Wright.
\newblock \emph{Numerical {Optimization}}.
\newblock Springer {Series} in {Operations} {Research} and {Financial}
  {Engineering}. Springer New York, NY, New York, NY, USA, second edition,
  2006.
\newblock \doi{10.1007/978-0-387-40065-5}.

\bibitem[Paquette and Scheinberg(2020)]{PaquSche2020}
C.~Paquette and K.~Scheinberg.
\newblock A stochastic {L}ine-{S}earch method with expected complexity
  analysis.
\newblock \emph{SIAM Journal on Optimization}, 30\penalty0 (1):\penalty0
  349--376, 2020.
\newblock \doi{10.1137/18M1216250}.

\bibitem[Powell(2002)]{Powell2002}
M.~Powell.
\newblock {UOBYQA}: unconstrained optimization by quadratic approximation.
\newblock \emph{Mathematical Programming}, 92\penalty0 (3):\penalty0 555--582,
  2002.
\newblock \doi{10.1007/s101070100290}.

\bibitem[Powell(1970)]{Powe70}
M.~J.~D. Powell.
\newblock A {New} {Algorithm} for {Unconstrained} {Optimization}.
\newblock In J.~Rosen, O.~Mangasarian, and K.~Ritter, editors, \emph{Nonlinear
  {Programming}}, pages 31--65. Academic Press, Jan. 1970.
\newblock \doi{10.1016/B978-0-12-597050-1.50006-3}.

\bibitem[Powell(1975)]{Powe75}
M.~J.~D. Powell.
\newblock Convergence properties of a class of minization algorithms.
\newblock In O.~Mangasarian, R.~Meyer, and S.~Robinson, editors,
  \emph{Nonlinear {Programming} 2}, pages 1--27. Academic Press, Jan. 1975.
\newblock \doi{10.1016/B978-0-12-468650-2.50005-5}.

\bibitem[Powell(1984)]{Powe84}
M.~J.~D. Powell.
\newblock On the global convergence of trust region algorithms for
  unconstrained minimization.
\newblock \emph{Mathematical Programming}, 29\penalty0 (3):\penalty0 297--303,
  1984.
\newblock \doi{10.1007/BF02591998}.

\bibitem[Rieussec and Bastin(2026)]{Jeremy2026-II}
J.~Rieussec and F.~Bastin.
\newblock Structural conditions on trust-region radius updates. {P}art {II}:
  second-order convergence and asymptotic radius behaviour.
\newblock \emph{ArXiv preprint}, 2026.

\bibitem[Robbins and Monro(1951)]{RobbMonro51}
H.~Robbins and S.~Monro.
\newblock A stochastic approximation method.
\newblock \emph{The Annals of Mathematical Statistics}, 22\penalty0
  (3):\penalty0 400--407, 1951.

\bibitem[Steihaug(1983)]{Stei83}
T.~Steihaug.
\newblock The {Conjugate} {Gradient} {Method} and {Trust} {Regions} in {Large}
  {Scale} {Optimization}.
\newblock \emph{{SIAM} Journal on Numerical Analysis}, 20\penalty0
  (3):\penalty0 626--637, 1983.
\newblock \doi{10.1137/0720042}.

\bibitem[Toint(1979)]{Toin79}
P.~L. Toint.
\newblock On the {Superlinear} {Convergence} of an {Algorithm} for {Solving} a
  {Sparse} {Minimization} {Problem}.
\newblock \emph{SIAM Journal on Numerical Analysis}, 16\penalty0 (6):\penalty0
  1036--1045, 1979.
\newblock \doi{10.1137/0716076}.

\bibitem[Toint(1988)]{Toint88}
P.~L. Toint.
\newblock Global convergence of a a of trust-region methods for nonconvex
  minimization in hilbert space.
\newblock \emph{IMA Journal of Numerical Analysis}, 8\penalty0 (2):\penalty0
  231--252, 1988.
\newblock \doi{10.1093/imanum/8.2.231}.

\bibitem[Toint(2013)]{Toint13}
P.~L. Toint.
\newblock Nonlinear stepsize control, trust regions and regularizations for
  unconstrained optimization.
\newblock \emph{Optimization Methods and Software}, 28\penalty0 (1):\penalty0
  82--95, Feb. 2013.
\newblock ISSN 1055-6788.
\newblock \doi{10.1080/10556788.2011.610458}.

\bibitem[Wang and Yuan(2022)]{WangYuan2022}
X.~Wang and Y.~Yuan.
\newblock Stochastic trust-region methods with trust-region radius depending on
  probabilistic models.
\newblock \emph{Journal of Computational Mathematics}, 40\penalty0
  (2):\penalty0 294--334, 2022.
\newblock \doi{10.4208/jcm.2012-m2020-0144}.

\bibitem[Xu et~al.(2020{\natexlab{a}})Xu, Roosta, and Mahoney]{XuRoosMaho20MP}
P.~Xu, F.~Roosta, and M.~W. Mahoney.
\newblock {N}ewton-type methods for non-convex optimization under inexact
  {H}essian information.
\newblock \emph{Mathematical Programming}, 184\penalty0 (1):\penalty0 35--70,
  2020{\natexlab{a}}.
\newblock \doi{10.1007/s10107-019-01405-z}.

\bibitem[Xu et~al.(2020{\natexlab{b}})Xu, Roosta, and
  Mahoney]{XuRoosMaho20SIAM}
P.~Xu, F.~Roosta, and M.~W. Mahoney.
\newblock \emph{Second-order {O}ptimization for {N}on-convex {M}achine
  {L}earning: an {E}mpirical {S}tudy}, pages 199--207.
\newblock Proceedings of the 2020 SIAM International Conference on Data Mining
  (SDM). SIAM, 2020{\natexlab{b}}.
\newblock \doi{10.1137/1.9781611976236.23}.

\bibitem[Yao et~al.(2021)Yao, Xu, Roosta, and Mahoney]{YaoXuRoostaMahoney2021}
Z.~Yao, P.~Xu, F.~Roosta, and M.~W. Mahoney.
\newblock Inexact nonconvex {N}ewton-type methods.
\newblock \emph{INFORMS Journal on Optimization}, 3\penalty0 (2):\penalty0
  154--182, 2021.
\newblock \doi{10.1287/ijoo.2019.0043}.

\bibitem[Yuan(2000{\natexlab{a}})]{Yuan2000}
Y.~Yuan.
\newblock On the truncated conjugate gradient method.
\newblock \emph{Mathematical Programming}, 87\penalty0 (3):\penalty0 561--573,
  2000{\natexlab{a}}.
\newblock \doi{10.1007/s101070050012}.

\bibitem[Yuan(1985)]{Yuan85}
Y.-x. Yuan.
\newblock Conditions for convergence of trust region algorithms for nonsmooth
  optimization.
\newblock \emph{Mathematical Programming}, 31\penalty0 (2):\penalty0 220--228,
  1985.
\newblock \doi{10.1007/BF02591750}.

\bibitem[Yuan(1998)]{Yuan98}
Y.-x. Yuan.
\newblock An example of non-convergence of trust region algorithms.
\newblock In Y.-x. Yuan, editor, \emph{Advances in {Nonlinear} {Programming}:
  {Proceedings} of the 96 {International} {Conference} on {Nonlinear}
  {Programming}}, pages 205--215. Springer US, Boston, MA, 1998.
\newblock \doi{10.1007/978-1-4613-3335-7\_9}.

\bibitem[Yuan(2000{\natexlab{b}})]{Yuan00}
Y.-x. Yuan.
\newblock A review of trust region algorithms for optimization.
\newblock In J.~M. Ball and J.~C.~R. Hunt, editors, \emph{{ICIAM99}:
  {Proceedings} of the {Fourth} {International} {Congress} on {Industrial} \&
  {Applied} {Mathematics} {Edinburgh}}, page~0. Oxford University Press, Dec.
  2000{\natexlab{b}}.
\newblock \doi{10.1093/oso/9780198505143.003.0023}.

\bibitem[Yuan(2015)]{Yuan15}
Y.-x. Yuan.
\newblock Recent advances in trust region algorithms.
\newblock \emph{Mathematical Programming}, 151\penalty0 (1):\penalty0 249--281,
  2015.

\bibitem[Yuan and Fan(2001)]{YuanFan01}
Y.-x. Yuan and J.~Fan.
\newblock A new trust region algorithm with trust region radius converging to
  zero.
\newblock In \emph{Proceedings of the 5th {International} {Conference} on
  {Optimization}: {Techniques} and {Applications}}, pages 786--794, 2001.

\end{thebibliography}

\appendix

\section{Example of a R-function} \label{appendix: example of R-function}

    As an illustration inspired from the example plot given in~\citet{Hei03}, one may define \( R_{\eta_1}(t) \) to interpolate smoothly across the threshold $\eta_1$,
    \[
    R_{\eta_1}(t) =
    \begin{cases}
        \gamma_1 + (\gamma_2 - \gamma_1) \expo{\lambda_1 (t - \eta_1)}, & \text{if } t < \eta_1, \\[8pt]
        1 + \gamma_2 + \big(\gamma_3 - (1 + \gamma_2)\big) \left(1 - \expo{-\lambda_2 (t - \eta_1)}\right), & \text{if } t \geq \eta_1,
    \end{cases}
    \]
    where \( \lambda_1 > 0 \) and \( \lambda_2 > 0 \) control the exponential rates of increase and decay, respectively, and \(\gamma_1, \gamma_2\) are the parameters described above.

        \begin{figure}[htbp]
        \centering
        \begin{tikzpicture}[scale=0.7]
        \begin{axis}[
            width=12cm, height=8cm,
            xlabel={$t$}, ylabel={$R_{\eta}(t)$},
            xmin=-4, xmax=6, ymin=0, ymax=3.5,
            grid=major,
            legend pos=north west,
            legend style={font=\small},
        ]
            \pgfmathsetmacro{\eta}{0.3} % Threshold parameter
            \pgfmathsetmacro{\beta}{0.1} % Lower limit for R_eta(t) as t -> -infinity
            \pgfmathsetmacro{\gammaone}{0.1} % Reduction factor for t < eta
            \pgfmathsetmacro{\gammatwo}{0.5} % Increase factor for t >= eta
            \pgfmathsetmacro{\lambdaone}{1.0} % Exponential growth rate for t < eta
            \pgfmathsetmacro{\lambdatwo}{1.0} % Exponential decay rate for t >= eta
            \pgfmathsetmacro{\M}{3.0} % Upper limit for R_eta(t) as t -> +infinity

            \addplot[
            domain=-4:\eta,
            samples=200,
            thick,
            blue
            ] { \beta + (1 - \gammaone - \beta) * exp(\lambdaone * (x - \eta)) };

            \addplot[
                        domain=\eta:6,
                        samples=200,
                        thick,
                        blue
                        ] { 1 + \gammatwo + (\M - (1 + \gammatwo)) * (1 - exp(-\lambdatwo * (x - \eta))) };

            \addplot[
            red, dashed, thick
            ] coordinates {(\eta, 0) (\eta, 11)};

            \addplot[black, dashed, thick] coordinates {(-4, 1) (6, 1)};

            \node[circle, draw=black, fill=black, inner sep=1pt] at (axis cs:\eta,{1 + \gammatwo}) {};
            \node[circle, draw=black, fill=white, inner sep=1pt] at (axis cs:\eta,{1 - \gammaone}) {};

        \end{axis}
        \end{tikzpicture}
        \caption{Plot of $R_{\eta_1}(t)$ with parameters $\eta = 0.3$, $\beta = 0.1$, $\gamma_1 = 0.1$, $\gamma_2 = 0.5$, $\lambda_1 = 1.0$, $\lambda_2 = 1.0$, and $M = 3.0$.}
        \label{fig: R_eta_1}
    \end{figure}

\end{document}